\documentclass[10pt]{amsart}
\usepackage{amssymb}
\usepackage{bm}
\usepackage{graphicx}
\usepackage[centertags]{amsmath}
\usepackage{amsfonts}
\usepackage{amsthm}
\usepackage{graphicx}
\newtheorem{thm}{Theorem}[section]
\newtheorem{cor}[thm]{Corollary}
\newtheorem{lem}[thm]{Lemma}
\newtheorem{sbl}[thm]{Sublemma}
\newtheorem{prop}[thm]{Proposition}
\newtheorem{claim}[thm]{Claim}
\newtheorem{subclaim}[thm]{Subclaim}
\newtheorem{fact}[thm]{Fact}
\newtheorem{defn}[thm]{Definition}

\theoremstyle{definition}

\newtheorem{examp}{Example}[section]

\newenvironment{Hales-Jewett}{\noindent\textbf{Hales--Jewett Theorem.}\begin{itshape}\noindent}{\end{itshape}}
\newenvironment{Halpern-Lauchli}{\noindent\textbf{Halpern--L\"{a}uchli Theorem.}\begin{itshape}\noindent}{\end{itshape}}
\newenvironment{Carlson-Simpson}{\noindent\textbf{Carlson--Simpson Theorem.}\begin{itshape}\noindent}{\end{itshape}}
\newenvironment{Density Hales-Jewett}{\noindent\textbf{Density Hales--Jewett Theorem.}\begin{itshape}\noindent}{\end{itshape}}
\newenvironment{Density Halpern-Lauchli}{\noindent\textbf{Density Halpern--L\"{a}uchli Theorem.}\begin{itshape}\noindent}{\end{itshape}}
\newenvironment{Theorem A}{\noindent\textbf{Theorem A.}\begin{itshape}\noindent}{\end{itshape}}
\newenvironment{Theorem B}{\noindent\textbf{Theorem B.}\begin{itshape}\noindent}{\end{itshape}}
\newenvironment{Theorem C}{\noindent\textbf{Theorem C.}\begin{itshape}\noindent}{\end{itshape}}

\newcommand{\nn}{\mathbb{N}}
\newcommand{\ee}{\varepsilon}
\newcommand{\con}{\smallfrown}
\newcommand{\meg}{\geqslant}
\newcommand{\mik}{\leqslant}
\newcommand{\lex}{<_{\mathrm{lex}}}
\newcommand{\dens}{\mathrm{dens}}
\newcommand{\cv}{\mathrm{c}}
\newcommand{\qv}{\mathrm{q}}
\newcommand{\ave}{\mathbb{E}}
\newcommand{\hj}{\mathrm{HJ}}
\newcommand{\dhj}{\mathrm{DHJ}}
\newcommand{\dcs}{\mathrm{DCS}}
\newcommand{\reg}{\mathrm{Reg}}

\newcommand{\cs}{\mathrm{CS}}

\newcommand{\subtr}{\mathrm{Subtr}}
\newcommand{\bl}{\mathbf{L}}

\newcommand{\bv}{\mathbf{V}}

\newcommand{\bx}{\mathbf{x}}
\newcommand{\by}{\mathbf{y}}

\begin{document}

\title{A density version of the Carlson--Simpson theorem}

\author{Pandelis Dodos, Vassilis Kanellopoulos and Konstantinos Tyros}

\address{Department of Mathematics, University of Athens, Panepistimiopolis 157 84, Athens, Greece}
\email{pdodos@math.uoa.gr}

\address{National Technical University of Athens, Faculty of Applied Sciences,
Department of Mathematics, Zografou Campus, 157 80, Athens, Greece}
\email{bkanel@math.ntua.gr}

\address{Department of Mathematics, University of Toronto, Toronto, Canada M5S 2E4}
\email{ktyros@math.toronto.edu}

\thanks{2000 \textit{Mathematics Subject Classification}: 05D10.}
\thanks{\textit{Key words}: words, left variable words, density.}

\maketitle


\begin{abstract}
We prove a density version of the Carlson--Simpson Theorem. Specifically we show the following.
\medskip

\noindent \textit{For every integer $k\meg 2$ and every set $A$ of words over $k$ satisfying
\[ \limsup_{n\to\infty} \frac{|A\cap [k]^n|}{k^n} >0 \]
there exist a word $c$ over $k$ and a sequence $(w_n)$ of left variable words over $k$ such that the set
\[ \{c\}\cup \big\{c^{\con}w_0(a_0)^{\con}...^{\con}w_n(a_n):n\in\nn \text{ and } a_0,...,a_n\in [k]\big\} \]
is contained in $A$.}
\medskip

\noindent While the result is infinite-dimensional its proof is based on an appropriate finite and quantitative
version, also obtained in the paper.
\end{abstract}

\tableofcontents


\section{Introduction}

\numberwithin{equation}{section}

\noindent 1.1. \textbf{Overview.} Our topic is Ramsey Theory, the general area of Combinatorics that studies the basic pigeonhole
principles of discrete structures and organizes, in a systematic way, the results obtained by iterating them.
\medskip

\noindent 1.1.1. \textit{The coloring versions.} The first pigeonhole principle relevant to our discussion in this paper is the
Hales--Jewett Theorem \cite{HJ}. To state it we need to introduce some pieces of notation and some terminology. For every integer $k\meg 2$
let $[k]^{<\nn}$ be the set of all finite sequences having values in $[k]:=\{1,...,k\}$. The elements of $[k]^{<\nn}$ are referred to as
\textit{words over $k$}, or simply \textit{words} if $k$ is understood. If $n\in\nn$, then $[k]^n$ stands for the set of words of length $n$.
We fix a letter $v$ that we regard as a variable. A \textit{variable word over $k$} is a finite sequence having values in $[k]\cup\{v\}$ where
the letter $v$ appears at least once. If $w$ is a variable word and $a\in[k]$, then $w(a)$ is the word obtained by substituting all appearances
of the letter $v$ in $w$ by $a$. A \textit{combinatorial line} of $[k]^n$ is a set of the form $\{w(a):a\in[k]\}$ where $w$ is a variable word
over $k$ of length $n$.
\medskip

\begin{Hales-Jewett}
For every $k,r\in\nn$ with $k\meg 2$ and $r\meg 1$ there exists an integer $N$ with the following property. If $n\meg N$, then for every
$r$-coloring of $[k]^n$ there exists a combinatorial line of $[k]^n$ which is monochromatic. The least integer $N$ with this property will
be denoted by $\hj(k,r)$.
\end{Hales-Jewett}
\medskip

The Hales--Jewett Theorem is the bread and butter of Ramsey Theory and is often regarded as an abstract version of the van der Waerden Theorem
\cite{vW}. The exact asymptotics of the numbers $\hj(k,r)$ are still unknown. The best known upper bounds are primitive recursive and are due 
to S. Shelah \cite{Sh1}.

The second pigeonhole principle relevant to our discussion is the Halpern--L\"{a}uchli Theorem \cite{HL}, a rather deep result that concerns
partitions of finite products of infinite trees.
\medskip

\begin{Halpern-Lauchli}
For every finite tuple $(T_1,...,T_d)$ of uniquely rooted and finitely branching trees without maximal nodes and every finite coloring of the
level product
\begin{equation} \label{1e1}
\bigcup_{n\in\nn} T_1(n)\times ...\times T_d(n)
\end{equation}
of $(T_1,...,T_d)$ there exist strong subtrees $(S_1, ..., S_d)$ of $(T_1,...,T_d)$ having common level set such that the level
product of $(S_1,...,S_d)$ is monochromatic.
\end{Halpern-Lauchli}
\medskip

We recall that a subtree $S$ of a tree $(T,<)$ is said to be \textit{strong} if: (a) $S$ is uniquely rooted, (b) there exists an infinite subset
$L_T(S)=\{l_0< l_1< ...\}$ of $\nn$, called the \textit{level set} of $S$, such that for every $n\in\nn$ the $n$-level $S(n)$ of $S$ is a subset
of $T(l_n)$, and (c) for every $s\in S$ and every immediate successor $t$ of $s$ in $T$ there exists a unique immediate successor $s'$ of $s$ in
$S$ with $t\mik s'$. The notion of a strong subtree was highlighted with the work of K. Milliken \cite{Mi1,Mi2} who used the Halpern--L\"{a}uchli
Theorem to show that the family of strong subtrees of a uniquely rooted and finitely-branching tree is partition regular.

The Hales--Jewett Theorem and the Halpern--L\"{a}uchli Theorem are pigeonhole principles of quite different nature. Nevertheless, they do admit
a common extension which is due to T. J. Carlson and S. G. Simpson \cite{CS}. To state it we recall that a \textit{left variable word over $k$}
is a variable word over $k$ whose leftmost letter is the variable $v$. The concatenation of two words $x$ and $y$ over $k$ is denoted
by $x^{\con}y$.
\medskip

\begin{Carlson-Simpson}
For every integer $k\meg 2$ and every finite coloring of the set of all words over $k$ there exist a word $c$ over $k$ and a sequence
$(w_n)$ of left variable words over $k$ such that the set
\begin{equation} \label{1e2}
\{c\}\cup \big\{c^{\con}w_0(a_0)^{\con}...^{\con}w_n(a_n): n\in\nn \text{ and } a_0,...,a_n\in [k]\big\}
\end{equation}
is monochromatic.
\end{Carlson-Simpson}
\medskip

The Carlson--Simpson Theorem belongs to the circle of results that provide information on the structure of the wildcard\footnote[1]{We recall
that if $w=(w_i)_{i=0}^{n-1}$ is a variable word over $k$ of length $n$, then its \textit{wildcard set} is defined to be the set
$\big\{i\in\{0,...,n-1\}: w_i=v\big\}$.} set of the variable word obtained by the Hales--Jewett Theorem; see, e.g., \cite{BL,HM,McC,Sh2,Wa}.
This extra information (namely, that the sequence $(w_n)$ consists of left variable words) can be used to derive the Halpern--L\"{a}uchli Theorem
when the trees $T_1,...,T_d$ are homogeneous\footnote[2]{A tree $T$ is \textit{homogeneous} if it is uniquely rooted and there exists an integer
$b\meg 2$ such that every $t\in T$ has exactly $b$ immediate successors; e.g., every dyadic, or triadic tree is homogeneous.}, a special case which
is sufficient for all known combinatorial applications of the Halpern--L\"{a}uchli Theorem (see \cite{PV}).
\medskip

\noindent 1.1.2. \textit{The density versions.} It is a remarkably fruitful phenomenon that many pigeonhole principles have a density version.
These density versions are strengthenings of their coloristic counterparts and assert that every large subset of a ``structure" must contain a
``substructure". In fact, the first pigeonhole principle we discussed so far, namely the Hales--Jewett Theorem, admits a density version which
is due to H. Furstenberg and Y. Katznelson \cite{FK2}.
\medskip

\begin{Density Hales-Jewett}
For every integer $k\meg 2$ and every $0<\delta\mik 1$ there exists an integer $N$ with the following property. If $n\meg N$, then every
subset $A$ of $[k]^n$ with $|A|\meg\delta k^n$ contains a combinatorial line of $[k]^n$. The least integer $N$ with this property will be
denoted by $\dhj(k,\delta)$.
\end{Density Hales-Jewett}
\medskip

The density Hales--Jewett Theorem is a fundamental result of Ramsey Theory. It has several strong results as consequences, most notably
the famous Szemer\'{e}di Theorem on arithmetic progressions \cite{Sz1} and its multidimensional version \cite{FK1}. The best known upper
bounds for the numbers $\dhj(k,\delta)$ are obtained in \cite{Pol} and have an Ackermann-type dependence with respect to $k$.

It turns out that the Halpern--L\"{a}uchli Theorem also has a density version that was obtained relatively recently in \cite{DKK}.
\medskip

\begin{Density Halpern-Lauchli}
For every finite tuple $(T_1,...,T_d)$ of homogeneous trees and every subset $A$ of the level product of $(T_1,...,T_d)$ satisfying
\begin{equation} \label{1e3}
\limsup_{n\to\infty} \frac{|A\cap \big( T_1(n)\times ... \times T_d(n)\big)|}{|T_1(n)\times ... \times T_d(n)|}>0
\end{equation}
there exist strong subtrees $(S_1, ..., S_d)$ of $(T_1,...,T_d)$ having common level set such that
the level product of $(S_1,...,S_d)$ is a subset of $A$.
\end{Density Halpern-Lauchli}
\medskip

We should point out that the assumption in the above result that the trees $T_1,...,T_d$ are homogeneous is not redundant. On the contrary,
various examples given in \cite{BV} show that it is essentially optimal.
\medskip

\noindent 1.2. \textbf{The main results.} In view of the above it is natural to ask whether the Carlson--Simpson Theorem has a density analogue
which would extend, among others, both the density Hales--Jewett Theorem and the density Halpern--L\"{a}uchli Theorem. Our goal in this
paper is to answer this question affirmatively. Specifically we show the following theorem.
\medskip

\begin{Theorem A}
For every integer $k\meg 2$ and every set $A$ of words over $k$ satisfying
\begin{equation} \label{1e4}
\limsup_{n\to\infty} \frac{|A\cap [k]^n|}{k^n} >0
\end{equation}
there exist a word $c$ over $k$ and a sequence $(w_n)$ of left variable words over $k$ such that the set
\begin{equation} \label{1e5}
\{c\}\cup \big\{c^{\con}w_0(a_0)^{\con}...^{\con}w_n(a_n):n\in\nn \text{ and } a_0,...,a_n\in [k]\big\}
\end{equation}
is contained in $A$.
\end{Theorem A}
\medskip

The proof of Theorem A follows a strategy that was already applied in a closely related context and was described in some detail in 
\cite[\S 1.3]{DKT2}. It consists of reducing Theorem A to an appropriate finite version. This finite version, which represents the combinatorial
core of Theorem A, is the content of the following theorem which is the second main result of the paper.
\medskip

\begin{Theorem B}
For every integer $k\meg 2$, every integer $m\meg 1$ and every $0<\delta\mik 1$ there exists an integer $N$ with the following property. If $L$
is a finite subset of $\nn$ of cardinality at least $N$ and $A$ is a set of words over $k$ satisfying $|A\cap [k]^n|\meg \delta k^n$ for every
$n\in L$, then there exist a word $c$ over $k$ and a finite sequence $(w_n)_{n=0}^{m-1}$ of left variable words over $k$ such that the set
\begin{equation} \label{1e6}
\{c\}\cup \big\{c^{\con}w_0(a_0)^{\con}...^{\con}w_n(a_n): n\in\{0,...,m-1\} \text{ and } a_0,...,a_n\in [k]\big\}
\end{equation}
is contained in $A$. The least integer $N$ with this property will be denoted by $\dcs(k,m,\delta)$.
\end{Theorem B}
\medskip

The main point in Theorem B is that the result is independent of the position of the finite set $L$. Its proof is based on a density increment
strategy -- a powerful method pioneered by K. F. Roth \cite{Roth} -- and yields explicit upper bounds for the numbers $\dcs(k,m,\delta)$.
These upper bounds are admittedly rather weak. They are in line, however, with several other bounds obtained recently in the area; see, e.g.,
\cite{DKT1,Gowers-MSz,Pol,RNSSK}.

Although Theorem B refers to left variable words, it can be used to obtain variable words with quite divergent structure. Specifically, given
two sequences $(p_n)$ and $(w_n)$ of variable words over $k$, we say that the sequence $(w_n)$ is of \textit{pattern} $(p_n)$ if $p_n$ is an
initial segment of $w_n$ for every $n\in\nn$. So, for instance, if $q_n=(v)$ for every $n\in\nn$, then a sequence $(w_n)$ of variable words
over $k$ is of pattern $(q_n)$ if and only if it consists of left variable words. We show the following theorem.
\medskip

\begin{Theorem C}
Let $k\in\nn$ with $k\meg 2$ and $(p_n)$ be an arbitrary sequence of variable words over $k$. Then for every set $A$ of words over $k$ satisfying
\begin{equation} \label{1e7}
\limsup_{n\to\infty} \frac{|A\cap [k]^n|}{k^n} >0
\end{equation}
there exist a word $c$ over $k$ and a sequence $(w_n)$ of variable words over $k$ of pattern $(p_n)$ such that the set
\begin{equation} \label{1e8}
\{c\}\cup \big\{c^{\con}w_0(a_0)^{\con}...^{\con}w_n(a_n):n\in\nn \text{ and } a_0,...,a_n\in [k]\big\}
\end{equation}
is contained in $A$.
\end{Theorem C}
\medskip

Of course, there is also a finite version of Theorem C in the spirit of Theorem B. This is the content of Theorem \ref{11t1} in the main text.
\medskip

\noindent 1.3. \textbf{Structure of the paper.} The paper is organized as follows. In \S 2 we set up our notation and terminology, and 
we recall some tools needed for the proof of the main results. Of particular importance is the notion of a \textit{Carlson--Simpson tree}
introduced in \S 2.5. It is the analogue, within the context of left variable words, of the notion of a combinatorial subspace.

The next four sections contain several preparatory results needed for the proof of Theorem B. This material is not only independent of the
rest of the paper but also of independent interest. In \S 3 we state and prove a ``regularity lemma'' for subsets of $[k]^{<\nn}$. The lemma
asserts that every dense subset of $[k]^{<\nn}$ is inherently pseudorandom and is proved via an energy increment strategy, an influential
method introduced by E. Szemer\'{e}di \cite{Sz2}. In the next section, \S 4, we present a partition result for Carlson--Simspon trees which
is, essentially, a variant of the classical Graham--Rothschild Theorem \cite{GR}. Finally, in \S 5 and \S 6 we develop a method of ``gluing''
a pair $x$ and $y$ of words over $k$. The method can be thought of as a natural extension of the familiar practice of concatenating $x$ and $y$.
It is encoded by what we call a \textit{convolution operation} which is introduced and studied in \S 5. Iterations of convolution operations
are studied in \S 6. We emphasize that the results in \S 6 are invoked only in \S 9. However, the material in \S 3, \S 4 and \S 5 is heavily
used and the reader is advised to gain some familiarity with the contents of these sections before reading the rest of the paper.

The next four sections are devoted to the proof of Theorem B. The results in \S 7 are independent of the rest of the argument. In particular,
this section can be read separately. The main part of the proof is contained in \S 8 and \S 9. The reader will find a detailed outline and an
exposition of the key ideas in \S 8.1 and \S 9.1. The proof of Theorem B is completed in \S 10. 

Finally, the last section of the paper contains a discussion on some consequences of Theorem B, including the proofs of Theorem A and Theorem C.


\section{Background material}

\numberwithin{equation}{section}

By $\nn=\{0,1,2,...\}$ we shall denote the natural numbers. For every integer $n\meg 1$ we set $[n]=\{1,...,n\}$. If $X$ is a nonempty finite
set, then by $\ave_{x\in X}$ we shall denote the average $\frac{1}{|X|} \sum_{x\in X}$ where, as usual, $|X|$ stands for the cardinality of $X$.
For every function $f:\nn\to\nn$ and every $\ell\in\nn$ by $f^{(\ell)}:\nn\to\nn$ we shall denote the $\ell$-th iteration of $f$ defined recursively
by the rule $f^{(0)}(n)=n$ and $f^{(\ell+1)}(n)=f\big(f^{(\ell)}(n)\big)$ for every $n\in\nn$.

Let $X$ be a nonempty (possibly infinite) set and $A$ be a subset of $X$. For every nonempty finite subset $Y$ of $X$ the \textit{density} of
$A$ in $Y$ is defined by
\begin{equation} \label{2e1}
\dens_Y(A)= \frac{|A\cap Y|}{|Y|}.
\end{equation}
If it is clear from the context to which set $Y$ we are referring (for instance, if $Y$ coincides with $X$), then we shall drop the subscript
$Y$ and we shall denote the above quantity simply by $\dens(A)$.
\medskip

\noindent 2.1. \textbf{Words.} For every $k\in\nn$ with $k\meg 2$ and every $n\in\nn$ let $[k]^n$ be the set of all sequences
of length $n$ having values in $[k]$. Precisely, $[k]^0$ contains just the empty sequence while if $n\meg 1$, then
\begin{equation} \label{2e2}
[k]^n=\big\{ (s_0,...,s_{n-1}): s_i\in [k] \text{ for every } i\in\{0,...,n-1\}\big\}.
\end{equation}
Also let
\begin{equation} \label{2e3}
[k]^{<n}=\bigcup_{\{i\in\nn:i<n\}} [k]^i.
\end{equation}
Notice, in particular, that $[k]^{<0}$ is empty. We set
\begin{equation} \label{2e4}
[k]^{<\nn}=\bigcup_{n\in\nn} [k]^n.
\end{equation}
The elements of $[k]^{<\nn}$ are called \textit{words over $k$}, or simply \textit{words} if $k$ is understood. The \textit{length}
of a word $x$ over $k$, denoted by $|x|$, is defined to be the unique natural number $n$ such that $x\in [k]^n$. For every $i\in\nn$
with $i\mik |x|$ by $x|_i$ we shall denote the word of length $i$ which is an initial segment of $x$. The concatenation of two words
$x, y$ will be denoted by $x^{\con}y$.
\medskip

\noindent 2.2. \textbf{Located words.} For every $k\in\nn$ with $k\meg 2$ and every (possibly empty) finite subset $J$ of $\nn$ by $[k]^J$
we shall denote the set of all functions from $J$ into $[k]$. An element of the set
\begin{equation} \label{2e5}
\bigcup_{J\subseteq \nn \text{ finite}} [k]^J
\end{equation}
will be called a \textit{located word over $k$}. If $x\in[k]^J$ is a located word over $k$ and $S$ is a subset of $J$, then $x|_S$ stands for the
restriction of $x$ on $S$; notice that $x|_S\in [k]^S$. Moreover, for every $x\in [k]^I$ and every $y\in [k]^J$, where $I$ and $J$ are two finite
subsets of $\nn$ with $I\cap J=\varnothing$, by $(x,y)$ we shall denote the unique element $z$ of $[k]^{I\cup J}$ satisfying $z|_I=x$ and $z|_J=y$.

Of course, every word over $k$ is a located word over $k$. Indeed, notice that
\begin{equation} \label{2e6}
[k]^{\{m\in\nn:m<n\}}=[k]^n
\end{equation}
for every $n\in\nn$.  Conversely, we may identify located words over $k$ with words over $k$ as follows. Let $J$ be a nonempty finite subset of
$\nn$. We set $j=|J|$ and we write the set $J$ in increasing order as $\{n_0<...< n_{j-1}\}$. The \textit{canonical isomorphism} associated to $J$
is the bijection $\mathrm{I}_J:[k]^j\to [k]^J$ defined by the rule
\begin{equation} \label{2e7}
\mathrm{I}_J(x)(n_i) = x(i)
\end{equation}
for every $i\in\{0,...,j-1\}$. Observing that $[k]^\varnothing= [k]^0=\{\varnothing\}$, we define the canonical isomorphism
$\mathrm{I}_{\varnothing}$ associated to the empty set to be the identity.
\medskip

\noindent 2.3. \textbf{Variable words.} Let $k,m\in\nn$ with $k\meg 2$ and $m\meg 1$, and fix a tuple $v_0,...,v_{m-1}$ of distinct letters.
An \textit{$m$-variable word over $k$} is a finite sequence having values in  $[k]\cup\{v_0,...,v_{m-1}\}$ such that: (a) for every
$i\in \{0,...,m-1\}$ the letter $v_i$ appears at least once, and (b) if $m\meg 2$, then for every $i,j\in\{0,...,m-1\}$ with $i<j$ all
occurrences of $v_i$ precede all occurrences of $v_j$. For every $m$-variable word $w$ over $k$ and every $a_0,...,a_{m-1}\in [k]$ by
$w(a_0,...,a_{m-1})$ we shall denote the unique word over $k$ obtained by substituting in $w$ all appearances of the letter $v_i$ with $a_i$
for every $i\in \{0,...,m-1\}$. A \textit{left variable word over $k$} is an $1$-variable word over $k$ whose leftmost letter is the variable $v$.
\medskip

\noindent 2.4. \textbf{Combinatorial subspaces.} Let $k,m\in\nn$ with $k\meg 2$ and $m\meg 1$. An \textit{$m$-dimensional combinatorial
subspace} of $[k]^{<\nn}$ is a set of the form
\begin{equation} \label{2e8}
V=\big\{ w(a_0,...,a_{m-1}): a_0,...,a_{m-1}\in [k]\big\}
\end{equation}
where $w$ is an $m$-variable word over $k$. The $1$-dimensional combinatorial subspaces are called \textit{combinatorial lines}.

For every $m$-dimensional combinatorial subspace $V$ of $[k]^{<\nn}$ and every $\ell\in [m]$ let $\mathrm{Subs}_{\ell}(V)$ be the set of all
$\ell$-dimensional combinatorial subspaces of $[k]^{<\nn}$ which are contained in $V$. We will need the following special case of the
Graham--Rothschild Theorem \cite{GR}.
\begin{thm} \label{2t1}
For every integer $k\meg 2$, every pair of integers $d\meg m\meg 1$ and every integer $r\meg 1$ there exists an integer $N$ with the
following property. If $n\meg N$ and $V$ is an $n$-dimensional combinatorial subspace of $[k]^{<\nn}$, then for every $r$-coloring of the set
$\mathrm{Subs}_m(V)$ there exits $W\in \mathrm{Subs}_d(V)$ such that the set $\mathrm{Subs}_m(W)$ is monochromatic. The least integer $N$ with
this property will be denoted by $\mathrm{GR}(k,d,m,r)$.
\end{thm}
Detailed expositions as well as infinite extensions of Theorem \ref{2t1} can be found in various places in the literature; see, e.g.,
\cite{BBH,C,FK3,McCbook,PV}. Also we remark that there exist primitive recursive upper bounds for the numbers $\mathrm{GR}(k,d,m,r)$
which are due to S. Shelah \cite{Sh1}.
\medskip

\noindent 2.5. \textbf{Carlson--Simpson trees.} We are about to introduce a family of combinatorial objects which will be of particular
importance throughout the paper.
\begin{defn} \label{2d2}
Let $k\in\nn$ with $k\meg 2$. A \emph{Carlson--Simpson tree} of $[k]^{<\nn}$ is a set of the form
\begin{equation} \label{2e9}
W=\{c\} \cup \big\{ c^{\con}w_0(a_0)^{\con}...^{\con}w_n(a_n): n\in\{0,...,m-1\} \text{ and } a_0,...,a_n\in [k]\big\}
\end{equation}
where $c$ is a word over $k$ and $(w_n)_{n=0}^{m-1}$ is a nonempty finite sequence of left variable words over $k$.
\end{defn}
It is easy to see that the sequence $(c,w_0,...,w_{m-1})$ that generates a Carlson--Simpson tree $W$ via formula \eqref{2e9} is unique. This
unique sequence will be called the \textit{generating} sequence of $W$. The corresponding natural number $m$ will be called the \emph{dimension}
of $W$ and will be denoted by $\dim(W)$. The $1$-dimensional Carlson--Simpson trees will be called \textit{Carlson--Simpson lines}.

Let $W$ be an $m$-dimensional Carlson--Simpson tree of $[k]^{<\nn}$ and $(c,w_0,...,w_{m-1})$ be its generating sequence.
The \textit{$0$-level} $W(0)$ of $W$ is defined by
\begin{equation} \label{2e10}
W(0)=\{c\}.
\end{equation}
Observe that $W(0)$ is contained in $[k]^{\ell_0}$ where $\ell_0$ is the length of $c$. Moreover, for every $n\in [m]$ the
\textit{$n$-level} $W(n)$ of $W$ is defined by
\begin{equation} \label{2e11}
W(n)=\big\{ c^{\con}w_0(a_0)^{\con}...^{\con}w_{n-1}(a_{n-1}): a_0,...,a_{n-1}\in [k]\big\}.
\end{equation}
Notice that $W(n)$ is an $n$-dimensional combinatorial subspace of $[k]^{<\nn}$ and is contained in $[k]^{\ell_n}$ where $\ell_n$
is the sum of the lengths of $c, w_0,..., w_{n-1}$. The set $\{\ell_0<...<\ell_m\}$ will be called the \textit{level set}
of $W$ and will be denoted by $L(W)$.

For every $m$-dimensional Carlson--Simpson tree $W$ of $[k]^{<\nn}$ and every $\ell\in [m]$ by $\mathrm{Subtr}_{\ell}(W)$ we shall
denote the set of all $\ell$-dimensional Carlson--Simpson trees of $[k]^{<\nn}$ which are contained in $W$. An element of
$\mathrm{Subtr}_{\ell}(W)$ will be called an \textit{$\ell$-dimensional Carlson--Simpson subtree} of $W$, or simply
\textit{Carlson--Simpson subtree} of $W$ if the dimension $\ell$ is understood.

The archetypical example of a Carlson--Simpson tree of $[k]^{<\nn}$ of dimension $m$ is the set $[k]^{<m+1}$. In fact, every
Carlson--Simpson tree of dimension $m$ can be thought of as a ``copy'' of $[k]^{<m+1}$ inside $[k]^{<\nn}$. Specifically, let
$W$ be an $m$-dimensional Carlson--Simpson tree of $[k]^{<\nn}$ and $(c,w_0,...,w_{m-1})$ be its generating sequence.
The \textit{canonical isomorphism} associated to $W$ is the bijection $\mathrm{I}_W:[k]^{<m+1}\to W$ defined by
$\mathrm{I}_W(\varnothing)=c$ and
\begin{equation} \label{2e12}
\mathrm{I}_W\big((a_0,...,a_{n-1})\big)= c^{\con}w_0(a_0)^{\con}...^{\con}w_{n-1}(a_{n-1})
\end{equation}
for every $n\in [m]$ and every $(a_0,...,a_{n-1})\in [k]^n$. The canonical isomorphism $\mathrm{I}_W$ preserves all structural properties
one is interested in while working in the category of Carlson--Simpson trees. For instance, if $\ell\in [m]$ and $V$ is a Carlson--Simpson
subtree of $[k]^{<m+1}$ of dimension $\ell$, then its image $\mathrm{I}_W(V)$ under the canonical isomorphism is an $\ell$-dimensional 
Carlson--Simpson subtree of $W$. Thus, for most practical purposes, we may identify $W$ with $[k]^{<m+1}$ via the canonical isomorphism
$\mathrm{I}_W$.

More generally, let $W$ and $U$ be two Carlson--Simpson trees of $[k]^{<\nn}$ of the same dimension. The \textit{canonical isomorphism}
associated to the pair $W, U$ is the bijection $\mathrm{I}_{W,U}:W\to U$ defined by the rule
\begin{equation} \label{2e13}
\mathrm{I}_{W, U}(t) = (\mathrm{I}_U\circ \mathrm{I}_W^{-1})(t)
\end{equation}
where $\mathrm{I}_W$ and $\mathrm{I}_U$ are the canonical isomorphisms associated to $W$ and $U$. Of course, the map $\mathrm{I}_{W,U}$
will be used to transfer information from $W$ to $U$ and vice versa.

Finally, for every $m$-dimensional Carlson--Simpson tree $W$ of $[k]^{<\nn}$ and every $k'\in\{2,...,k\}$ we define the
\textit{$k'$-restriction} $W\upharpoonright k'$ of $W$ to be the set
\begin{equation} \label{2e14}
\{c\} \cup \big\{ c^{\con}w_0(a'_0)^{\con}...^{\con}w_n(a'_n): n\in\{0,...,m-1\} \text{ and } a'_0,...,a'_n\in [k']\big\}
\end{equation}
where $(c,w_0,...,w_{m-1})$ stands for the generating sequence of $W$. Notice that the canonical isomorphism of $W$ maps $[k']^{<m+1}$
onto $W\upharpoonright k'$. Therefore, $W\upharpoonright k'$ can be naturally identified as a Carlson--Simpson tree of $[k']^{<\nn}$.
\medskip

\noindent 2.6. \textbf{Insensitive sets.} Let $k\in\nn$ with $k\meg 2$ and $x,y$ be two words over $k$. Also let $i,j\in [k]$ with $i\neq j$.
We say that $x$ and $y$ are \textit{$(i,j)$-equivalent} if: (a) $x$ and $y$ have common length, and (b) if $n$ is the common length
of $x$ and $y$, then for every $s\in [k]\setminus \{i,j\}$ and every $r\in \nn$ with $r<n$ we have $x(r)=s$ if and only if $y(r)=s$.

If $n\in\nn$ and $A$ is a subset of $[k]^n$, then $A$ is said to be \emph{$(i,j)$-insensitive} if for every $x\in A$ and every $y\in [k]^n$ if $x$
and $y$ are $(i,j)$-equivalent, then $y\in A$. The notion of an $(i,j)$-insensitive set was introduced by S. Shelah \cite{Sh1} and highlighted in
the polymath proof \cite{Pol} of the density Hales--Jewett Theorem. It can be naturally extended to subsets of $[k]^{<\nn}$ as follows.
\begin{defn} \label{2d3}
Let $k\in\nn$ with $k\meg 2$ and $i,j\in [k]$ with $i\neq j$. Also let $A$ be a subset of $[k]^{<\nn}$. We say that $A$ is
\emph{$(i,j)$-insensitive} if for every $n\in\nn$ the set $A\cap [k]^n$ is $(i,j)$-insensitive.  If $W$ is Carlson--Simpson tree of $[k]^{<\nn}$,
then we say that $A$ is \emph{$(i,j)$-insensitive in} $W$ if $\mathrm{I}^{-1}_W(A\cap W)$ is an $(i,j)$-insensitive subset of $[k]^{<\nn}$
where $\mathrm{I}_W$ is the canonical isomorphism associated to $W$.
\end{defn}
It is easy to see that the family of all $(i,j)$-insensitive subsets of $[k]^{<\nn}$ is closed under intersections, unions and complements.
The same remark, of course, applies to the family of all $(i,j)$-insensitive sets in a Carlson--Simpson tree $W$ of $[k]^{<\nn}$.
\medskip

\noindent 2.7. \textbf{Furstenberg--Weiss measures.} Let $k,m\in\nn$ with $k\meg 2$ and $m\meg 1$. The \textit{Furstenberg--Weiss measure}
$\mathrm{d}_{\mathrm{FW}}^m$ associated to $[k]^{<m+1}$ is the probability measure on $[k]^{<\nn}$ defined by
\begin{equation} \label{2e15}
\mathrm{d}_{\mathrm{FW}}^m(A)= \ave_{n\in\{0,...,m\}} \dens_{[k]^n}(A).
\end{equation}
This class of measures was introduced by H. Furstenberg and B. Weiss \cite{FW} and has proven to be useful in various problems
in Ramsey Theory (see, e.g., \cite{DKT1,PST}). We will need the following two variants.
\begin{defn} \label{2d4}
Let $k\in\nn$ with $k\meg 2$.
\begin{enumerate}
\item[(i)] For every Carlson--Simpson tree $W$ of $[k]^{<\nn}$ the \emph{Furstenberg--Weiss measure} $\mathrm{d}_{\mathrm{FW}}^W$
associated to $W$ is the probability measure on $[k]^{<\nn}$ defined by
\begin{equation} \label{2e16}
\mathrm{d}_{\mathrm{FW}}^W(A)= \ave_{n\in\{0,...,\dim(W)\}} \dens_{W(n)}(A).
\end{equation}
\item[(ii)] For every nonempty finite subset $L$ of $\nn$ the \emph{generalized Furstenberg--Weiss measure} $\mathrm{d}_L$ associated to $L$
is the probability measure on $[k]^{<\nn}$ defined by
\begin{equation} \label{2e17}
\mathrm{d}_L(A)= \ave_{n\in L} \dens_{[k]^n}(A).
\end{equation}
\end{enumerate}
\end{defn}
It is, of course, clear that if $L$ is an initial interval of $\nn$ of cardinality $\ell\meg 2$, then the generalized Furstenberg--Weiss measure
$\mathrm{d}_L$ associated to $L$ coincides with the Furstenberg--Weiss measure $\mathrm{d}_{\mathrm{FW}}^{\ell-1}$.
\medskip

\noindent 2.8. \textbf{Probabilistic preliminaries.} We record, for future use, three probabilistic facts.
The first one is an immediate consequence of Markov's inequality.
\begin{lem} \label{2l5}
Let $(\Omega,\Sigma,\mu)$ be a probability space and $0<\delta\mik 1$. Also let $(A_i)_{i=1}^n$ be a finite family of measurable events in
$(\Omega,\Sigma,\mu)$ such that $\mu(A_i)\meg\delta$ for every $i\in [n]$. Then, setting $L_\omega=\{i\in [n]:\omega\in A_i\}$ for every
$\omega\in\Omega$, we have
\begin{equation} \label{2e18}
\mu\big( \{ \omega: |L_\omega|\meg (\delta/2)n \}\big) \meg \delta/2.
\end{equation}
\end{lem}
\begin{proof}
For every $i\in [n]$ let $\mathbf{1}_{A_i}$ be the indicator function of the event $A_i$ and set $Z=\frac{1}{n}\sum_{i=1}^n \mathbf{1}_{A_i}$.
Then $\ave[Z]\meg\delta$ and the result follows.
\end{proof}
To state the second result we recall that if $(\Omega,\Sigma,\mu)$ is a probability space and $Y\in\Sigma$ with $\mu(Y)>0$, then $\mu_Y$ stands
for the conditional probability measure of $\mu$ relative to $Y$ defined by
\begin{equation} \label{2e19}
\mu_Y(A)=\frac{\mu(A\cap Y)}{\mu(Y)}
\end{equation}
for every $A\in\Sigma$.
\begin{lem} \label{2l6}
Let $(\Omega,\Sigma,\mu)$ be a probability space and $0<\lambda,\beta,\ee\mik 1$. Let $A$ and $B$ be two measurable events in $(\Omega,\Sigma,\mu)$
with $A\subseteq B$ and such that $\mu(A)\meg\lambda\mu(B)$ and $\mu(B)\meg\beta$. Suppose that $\mathcal{Q}=(Q_i)_{i=1}^n$ is a finite family
of pairwise disjoint measurable events in $(\Omega,\Sigma,\mu)$ such that $\mu(B\setminus \cup \mathcal{Q})\mik\ee\beta/2$ and $\mu(Q_i)>0$
for every $i\in [n]$. Then, setting
\begin{equation} \label{2e20}
I=\big\{i\in [n]:\mu_{Q_i}(A)\meg(\lambda-\ee)\mu_{Q_i}(B) \text{ and } \mu_{Q_i}(B)\meg\beta\ee/4\big\},
\end{equation}
we have
\begin{equation} \label{2e21}
\sum_{i\in I}\mu(Q_i)\meg\beta\ee/4.
\end{equation}
In particular, if $\mu(Q_i)=\mu(Q_j)$ for every $i,j\in[n]$, then $|I|\meg (\beta\ee/4)n$.
\end{lem}
\begin{proof}
Notice, first, that $\mu(A\setminus \cup\mathcal{Q})\mik\ee\beta/2$. This is easily seen to imply that
\begin{equation} \label{2e22}
\sum_{i=1}^n\frac{\mu(A\cap Q_i)}{\mu(B)} \meg \lambda - \ee/2.
\end{equation}
For every $i\in [n]$ let $a_i=\mu_{Q_i}(A)/\mu_{Q_i}(B)$, $b_i=\mu_{Q_i}(B)$ and $c_i=\mu(Q_i)/\mu(B)$ with the convention that $a_i=0$
if $\mu(B\cap Q_i)=0$. Then inequality \eqref{2e22} can be reformulated as
\begin{equation} \label{2e23}
\sum_{i=1}^n a_i b_i c_i \meg \lambda - \ee/2.
\end{equation}
Notice that
\begin{equation} \label{2e24}
\sum_{i=1}^n b_i c_i \mik 1 \text{ and } \sum_{i=1}^n c_i\mik \frac{1}{\beta}.
\end{equation}
Also observe that $I=\{i\in [n]: a_i\meg\lambda-\ee \text{ and } b_i\meg \beta\ee/4\}$. Since $0\mik a_i,b_i\mik 1$ for every $i\in [n]$,
combining \eqref{2e23}, \eqref{2e24} and the previous remarks, we see that $\sum_{i\in I} c_i\meg \ee/4$ and the proof is completed.
\end{proof}
The final result of this subsection is the following.
\begin{lem} \label{2l7}
Let $0< \theta< \ee\mik 1$ and $n\in\nn$ with $n\meg (\ee^2-\theta^2)^{-1}$. If $(A_i)_{i=1}^{n}$ is a family of measurable events in a probability
space $(\Omega,\Sigma,\mu)$ satisfying $\mu(A_i)\meg \ee$ for every $i\in [n]$, then there exist $i,j\in [n]$ with $i\neq j$ such that
$\mu(A_i\cap A_j) \meg \theta^2$.
\end{lem}
\begin{proof}
We set $X=\sum_{i=1}^n \mathbf{1}_{A_i}$ where $\mathbf{1}_{A_i}$ is the indicator function of the event $A_i$ for every $i\in [n]$. 
Then $\mathbb{E}[X]\meg \ee n$ so, by convexity,
\begin{equation} \label{2e25}
\sum_{i\in [n]} \sum_{j\in [n]\setminus \{i\}} \mu(A_i\cap A_j) = \mathbb{E}[X(X-1)] \meg \ee n (\ee n-1).
\end{equation}
Therefore, there exist $i,j\in [n]$ with $i\neq j$ such that $\mu(A_i\cap A_j)\meg \theta^2$.
\end{proof}


\section{A regularity lemma for subsets of $[k]^{<\nn}$}

\numberwithin{equation}{section}

\noindent 3.1. \textbf{Statement of the main result.} Our goal in this section is to prove a ``regularity lemma'' for subsets
of $[k]^{<\nn}$. Roughly speaking, the lemma asserts that if $n$ is large enough and $A$ is a subset of $[k]^n$, then we may find a set
of coordinates $I\subseteq \{m\in\nn:m<n\}$ of preassigned cardinality such that the set $A$, viewed as a subset of the product
$[k]^I\times [k]^{\{m\in\nn:m<n\}\setminus I}$, behaves like a randomly chosen set.

To put things in a proper perspective we need, first, to determine the kind of randomness we are referring to. This is the content of the
following definition.
\begin{defn} \label{3d1}
Let $k\in\nn$ with $k\meg 2$ and $\mathcal{F}$ be a family of subsets of $[k]^{<\nn}$. Also let $0<\ee\mik 1$ and $L$ be a nonempty finite
subset of $\nn$. The family $\mathcal{F}$ will be called \emph{$(\ee, L)$-regular} provided that for every $A\in\mathcal{F}$, every $n\in L$,
every (possibly empty) subset $I$ of $\{l\in L: l<n\}$ and every $y\in [k]^{I}$ we have
\begin{equation} \label{3e1}
| \dens\big(\{w\in [k]^{\{m\in\nn:m<n\}\setminus I}: (y,w)\in A\cap [k]^{n}\}\big)-\dens(A\cap [k]^n) |\mik\ee.
\end{equation}
\end{defn}
Notice that for every $y\in [k]^I$ the set $\{w\in [k]^{\{m\in\nn:m<n\}\setminus I}: (y,w)\in A\cap [k]^{n}\}$ is just the section $A\cap [k]^n$
at $y$. So what Definition \ref{3d1} guarantees is that for every $n\in L$ and every $I\subseteq \{l\in L: l<n\}$ the density of the sections
of $A\cap [k]^n$ along elements of $[k]^I$ are essentially equal to the density of $A\cap [k]^n$.

We are now ready to state the main result of this section.
\begin{lem} \label{3l2}
For every $0<\ee\mik 1$ and every $k,\ell,q\in\nn$ with $k\meg 2$ and $\ell,q\meg 1$ there exists an integer $n$ with the following property.
If $N$ is a finite subset of $\nn$ with $|N|\meg n$ and $\mathcal{F}$ is a family of subsets of $[k]^{<\nn}$ with $|\mathcal{F}|=q$, then there
exists a subset $L$ of $N$ with $|L|=\ell$ such that $\mathcal{F}$ is $(\ee, L)$-regular. The least integer $n$ with this property will be
denoted by $\reg(k,\ell,q,\ee)$.
\end{lem}
The proof of Lemma \ref{3l2} will be given in \S 3.2. It is based on an energy increment strategy, a powerful method introduced by E. Szemer\'{e}di
in his proof of the celebrated regularity lemma \cite{Sz2}. The argument is, of course, effective and yields explicit upper bounds for
the numbers $\reg(k,\ell,q,\ee)$.
\medskip

\noindent 3.2. \textbf{Proof of Lemma \ref{3l2}.} We begin with the following definition which is the most important ingredient of the proof.
\begin{defn} \label{3d3}
Let $k,n\in\nn$ with $k\meg 2$. Also let $I$ be a (possibly empty) subset of $\{m\in\nn:m<n\}$. For every subset $A$ of $[k]^n$ we define the
\emph{energy} of $A$ with respect to $I$ to be the quantity
\begin{equation} \label{3e2}
\mathrm{e}_I(A)=\ave_{y\in [k]^I}\dens(A_y)^2
\end{equation}
where $A_y=\{w\in [k]^{\{m\in\nn:m<n\}\setminus I}: (y, w)\in A\}$ is the section of $A$ at $y$.
\end{defn}
We will isolate some basic properties of the energy which are needed for the proof. To this end, we need to introduce some pieces of notation.
Specifically, let $k,n\in\nn$ with $k\meg 2$ and $I, J$ be two subsets of $\{m\in\nn:m<n\}$ with $I\cap J=\varnothing$. We set
$M=\{m\in\nn: m<n\}\setminus (I\cup J)$. If we are given a subset $A$ of $[k]^n$, then we may view the set $A$ as a subset of the product
$[k]^I\times [k]^J\times [k]^M$ and so we may define the section $A_{(y,z)}=\{v\in [k]^M: (y,z,v)\in A\}$ for every $(y,z)\in [k]^I\times [k]^J$.
Notice that
\begin{equation} \label{3e3}
\dens(A_y)=\ave_{z\in [k]^{J}} \dens(A_{(y,z)}) \text{ and } \dens(A_z)=\ave_{y\in [k]^{I}}\dens(A_{(y,z)})
\end{equation}
for every $y\in [k]^I$ and every $z\in [k]^J$. We have the following.
\begin{fact} \label{3f4}
Let $k,n\in\nn$ with $k\meg 2$. Also let $I$ be a subset of $\{m\in\nn: m<n\}$. Then for every subset $A$ of $[k]^n$ we have that
$\mathrm{e}_I(A)\mik 1$. Moreover, if $J$ is a subset of $\{m\in\nn:m<n\}$ with $I\cap J=\varnothing$, then
\begin{equation} \label{3e4}
\mathrm{e}_{I\cup J}(A)-\mathrm{e}_J(A) = \ave_{z\in [k]^J} \ave_{y\in [k]^I}\Big( \dens(A_{(y,z)})- \ave_{y\in [k]^I} \dens(A_{(y,z)})\Big)^2.
\end{equation}
In particular, $\mathrm{e}_J(A)\mik \mathrm{e}_{I\cup J}(A)$.
\end{fact}
\begin{proof}
The fact that $\mathrm{e}_I(A)\mik 1$ follows immediately by Definition \ref{3d3}. Observe that
\begin{equation} \label{3e5}
\mathrm{e}_{I\cup J}(A) = \ave_{z\in [k]^J}\Big( \ave_{y\in [k]^I}\dens(A_{(y,z)})^2\Big)
\end{equation}
and
\begin{equation} \label{3e6}
\mathrm{e}_J(A)\stackrel{(\ref{3e3})}{=} \ave_{z\in [k]^J} \Big( \ave_{y\in [k]^I} \dens(A_{(y,z)}) \Big)^2.
\end{equation}
Combining (\ref{3e5}) and (\ref{3e6}) the result follows.
\end{proof}
The first step towards the proof of Lemma \ref{3l2} is the following.
\begin{sbl} \label{3sbl5}
Let $k,n\in\nn$ with $k\meg 2$. Also let $I$ and $J$ be two subsets of $\{m\in\nn: m<n\}$ with $I\cap J=\varnothing$.
Finally let $A$ be a subset of $[k]^n$ and $0<\ee<k^{-|I|}$. If $\mathrm{e}_{I\cup J}(A)-\mathrm{e}_J(A)\mik \ee^4$, then
\begin{equation} \label{3e7}
\dens\Big( \big\{z\in [k]^J: |\dens(A_{(y,z)})-\dens(A_z)|\mik\ee \text{ for every } y\in [k]^I \big\}\Big)\meg 1-\ee.
\end{equation}
\end{sbl}
\begin{proof}
We set $Y=[k]^I$ and $Z=[k]^J$. For every $z\in Z$ let $f_z:Y\to[0,1]$ be the random variable defined by $f_z(y)=\dens(A_{(y,z)})$. Let
$E(f_z)=\ave_{y\in Y}f_z(y)$ be the expected value of $f_z$ and $\mathrm{Var}(f_z)=E(f_z^2)-E(f_z)^2$ be its variance. Notice that
$E(f_z)=\dens(A_z)$. By (\ref{3e4}), we see that $\ave_{z\in Z}\mathrm{Var}(f_z) = \mathrm{e}_{I\cup J}(A)-\mathrm{e}_J(A)$.
Hence, by our assumptions, we have
\begin{equation} \label{3e8}
\ave_{z\in Z}\mathrm{Var}(f_z) \mik \ee^4
\end{equation}
and so, by Markov's inequality,
\begin{equation} \label{3e9}
\dens\big(\{z\in Z:\mathrm{Var}(f_z)\mik\ee^3\}\big) \meg 1-\ee.
\end{equation}
Fix $z_0\in Z$ with $\mathrm{Var}(f_{z_0})\mik\ee^3$. By Chebyshev's inequality, we have
\begin{equation} \label{3e10}
\dens\big(\{y\in Y:|f_{z_0}(y)- E(f_{z_0})|\mik \ee\}\big) \meg 1-\ee
\end{equation}
and since $\ee<|Y|^{-1}$ we get that $|f_{z_0}(y)- E(f_{z_0})|\mik \ee$ for every $y\in Y$. This is equivalent to say that
$|\dens(A_{(y,z_0)})-\dens(A_{z_0})|\mik\ee$ for every $y\in [k]^I$ and the proof is completed.
\end{proof}
Sublemma \ref{3sbl5} will be used in the following form.
\begin{cor} \label{3c6}
Let $k,n\in\nn$ with $k\meg 2$. Also let $I$ and $J$ be two subsets of $\{m\in\nn:m<n\}$ with $I\cap J=\varnothing$.
Finally let $A$ be a subset of $[k]^n$ and $0<\ee<k^{-|I|}$. If $\mathrm{e}_{I\cup J}(A)-\mathrm{e}_J(A)\mik \ee^4/16$, then
$|\dens(A_y)-\dens(A)|\mik\ee$ for every $y\in [k]^I$.
\end{cor}
\begin{proof}
We set $\ee_0=\ee/2$ and
\begin{equation} \label{3e11}
Z_0=\big\{ z\in [k]^J: |\dens(A_{(y,z )})-\dens(A_z)| \mik\ee_0 \text{ for every } y\in [k]^I\}.
\end{equation}
By Sublemma \ref{3sbl5}, we have  $\dens([k]^J\setminus Z_0)\mik\ee_0$. Hence, for every $y\in [k]^I$,
\begin{eqnarray} \label{3e12}
|\dens(A_y)-\dens(A)| & \stackrel{(\ref{3e3})}{\mik} & \ave_{z\in [k]^J} |\dens(A_{(y,z)})-\dens(A_z)| \\
& \mik & \ave_{z\in Z_0} |\dens(A_{(y,z)})-\dens(A_z)| + \ee_0\nonumber \\
& \stackrel{(\ref{3e11})}{\mik} & \ee_0 +\ee_0 = \ee \nonumber
\end{eqnarray}
as desired.
\end{proof}
We proceed to the second step of the proof of Lemma \ref{3l2}.
\begin{sbl} \label{3sbl7}
Let $k,m,q\in\nn$ with $k\meg 2$ and $q\meg1$ and $0<\ee< k^{-m}$. Also let $N$ be a finite subset of $\nn$ with
$|N|\meg \big(q\lfloor 16\ee^{-4}\rfloor+1\big)m +1$ and $\mathcal{F}$ be a family of subsets of $[k]^{\max(N)}$ with $|\mathcal{F}|=q$.
Then, setting $N'=N\setminus \{\max(N)\}$, there exists a subinterval $M$ of $N'$ (i.e., $M$ is of the form $J\cap N'$ for some interval
$J$ of $\nn$) with $|M|=m$ and such that for every $A\in\mathcal{F}$, every subset $I$ of $M$ and every $y\in k^{I}$ we have
$|\dens(A_y)-\dens(A)|\mik \ee$.
\end{sbl}
\begin{proof}
Clearly we may assume that $m\meg 1$. We set $r_0=q\lfloor 16\ee^{-4}\rfloor+1$. Write the first $r_0\cdot m$ elements
of $N'$ in increasing order as $\{n_0<n_1<\ldots<n_{r_0\cdot m-1}\}$. For every $p\in \{0,...,r_0-1\}$ let
\begin{equation} \label{3e13}
I_p=\big\{ n_{p\cdot m+ j}: j\in\{0,...,m-1\}\big\} \ \text{ and} \ J_p=\{j\in\nn: j<n_{p\cdot m}\}.
\end{equation}
Notice that $\max(J_p)<\min(I_p)<\max(N)$. Moreover, $I_p\cup J_p\subseteq J_{p+1}$ if $p\mik r_0-2$. Hence, by Fact \ref{3f4}, we have
\begin{equation} \label{3e14}
\mathrm{e}_{J_p}(A) \mik \mathrm{e}_{I_p\cup J_p}(A) \mik \mathrm{e}_{J_{p+1}}(A)\mik \mathrm{e}_{I_{p+1}\cup J_{p+1}}(A)\mik 1
\end{equation}
for every $p\in\{0,... ,r_0-2\}$ and every $A\in\mathcal{F}$. For every $A\in\mathcal{F}$ let
\begin{equation} \label{3e15}
P_A=\big\{ p\in \{0,...,r_0-1\}: \mathrm{e}_{I_p\cup J_p}(A)-\mathrm{e}_{J_p}(A)>\ee^4/16\big\}.
\end{equation}
The previous discussion implies that the set $P_A$ has cardinality at most $\lfloor 16\ee^{-4}\rfloor$. Therefore, we may select
$p_0\in \{0,...,r_0-1\}$ such that $p_0\notin P_A$ for every $A\in\mathcal{F}$; in particular,
$\mathrm{e}_{I_{p_0}\cup J_{p_0}}(A)-\mathrm{e}_{J_{p_0}}(A)\mik\ee^4/16$.
Since $I_{p_0}\cap J_{p_0}=\varnothing$, by Corollary \ref{3c6}, we conclude that
\begin{equation} \label{3e16}
|\dens(A_y)-\dens(A)|\mik \ee
\end{equation}
for every $y\in [k]^{I_{p_0}}$ and every $A\in \mathcal{F}$.

We set $M=I_{p_0}$. We will show that with this choice all requirements of the sublemma are satisfied. Indeed, notice that $M$ is
a subinterval of $N'$ with $|M|=m$. We fix $A\in\mathcal{F}$. Also let $I\subseteq M$ and $y\in [k]^I$ be arbitrary. Observe that for
every $z\in [k]^{M\setminus I}$ we have $(y,z)\in [k]^{I_{p_0}}$. Hence,
\begin{eqnarray} \label{3e17}
|\dens(A_y)-\dens(A)| & = & | \ave_{z\in [k]^{M\setminus I}}\dens(A_{(y,z)})-\dens(A)| \\
& \mik & \ave_{z\in [k]^{M\setminus I}} |\dens(A_{(y,z)})-\dens(A)| \stackrel{\eqref{3e16}}{\mik} \ee \nonumber
\end{eqnarray}
and the proof is completed.
\end{proof}
We are in the position to complete the proof of Lemma \ref{3l2}. To this end, we need to introduce some numerical invariants.
Specifically, for every $0<\ee\mik 1$ and every $k,\ell,q\in\nn$ with $k\meg 2$ and $\ell,q\meg 1$ let
\begin{equation} \label{3e18}
\rho =\rho(k,\ell,q,\ee)= \min\{\ee,k^{-\ell}/2\}
\end{equation}
and define $F_{k,\ell,q,\ee}:\nn\to\nn$ by the rule
\begin{equation} \label{3e19}
F_{k,\ell,q,\ee}(m)= (q\lfloor 16\rho^{-4}\rfloor+1)m+1.
\end{equation}
\begin{proof}[Proof of Lemma \ref{3l2}]
We will show that
\begin{equation} \label{3e20}
\reg(k,\ell,q,\ee)\mik F_{k,\ell,q,\ee}^{(\ell)}(0)
\end{equation}
for every $0<\ee\mik 1$ and every $k,\ell,q\in\nn$ with $k,\ell\meg 2$ and $q\meg 1$. Indeed, let $N$ be a finite subset of $\nn$ with
$|N|\meg F_{k,\ell,q,\ee}^{(\ell)}(0)$ and fix a family $\mathcal{F}$ of subsets of $[k]^{<\nn}$ with $|\mathcal{F}|=q$.
We select a subset $M_0$ of $N$ with $|M_0|=F_{k,\ell,q,\ee}^{(\ell)}(0)$. By repeated applications of Sublemma \ref{3sbl7}, we may
construct a family $\{M_1, ..., M_{\ell-1}\}$ of finite subsets of $M_0$ such that for every $i\in [\ell-1]$
\begin{enumerate}
\item[(a)] $|M_i|=F_{k,\ell,q,\ee}^{(\ell-i)}(0)$,
\item[(b)] $M_i$ is a subinterval of $M_{i-1}\setminus \{\max(M_{i-1})\}$ and
\item[(c)] for every $A\in\mathcal{F}$, every subset $I$ of $M_i$ and every $y\in k^{I}$ we have
\begin{equation} \label{3e21}
|\dens\big( \{w\in [k]^{C}: (y,w)\in A\cap [k]^{\max(M_{i-i})}\}\big)-\dens (A\cap [k]^{\max(M_{i-1})})|\mik \ee
\end{equation}
where $C=\{m\in\nn:m<\max(M_{i-1})\}\setminus I$.
\end{enumerate}
We set $L=\{\max(M_{\ell-1})< ... <\max(M_0)\big\}$. Using properties (b) and (c) it is easy to check that the family $\mathcal{F}$ is
$(\ee,L)$-regular, as desired.
\end{proof}


\section{A variant of the Graham--Rothschild Theorem for left variable words}

\numberwithin{equation}{section}

Recall that for every Carlson--Simpson tree $V$ of $[k]^{<\nn}$ and every $\ell\in[\dim(V)]$ by $\mathrm{Subtr}_{\ell}(V)$ we denote the set
of all $\ell$-dimensional Carlson--Simpson subtrees of $V$. This section is devoted to the proof of the following partition result.
\begin{thm} \label{4t1}
For every integer $k\meg 2$, every pair of integers $d\meg m\meg 1$ and every integer $r\meg 1$ there exists an integer $N$ with the following
property. If $n\meg N$ and $W$ is an $n$-dimensional Carlson--Simpson tree of $[k]^{<\nn}$, then for every $r$-coloring of the set
$\mathrm{Subtr}_m(W)$ there exists $U\in\mathrm{Subtr}_d(W)$ such that the set $\mathrm{Subtr}_m(U)$ is monochromatic. The least integer
$N$ with this property will be denoted by $\mathrm{CS}(k,d,m,r)$.
\end{thm}
Theorem \ref{4t1} is, of course, a variant of Theorem \ref{2t1}. It can be hardly characterized as new since it follows using fairly
standard arguments. Nevertheless, we have decided to include a proof for two reasons. The first one is self-containedness. Secondly,
because we want to emphasize the bounds we get from the argument for the numbers $\mathrm{CS}(k,d,m,r)$.

We start by introducing some pieces of notation. Specifically, let $k,d,m\in\nn$ with $k\meg 2$ and $d\meg m\meg 1$. Also let $W$ be
a $d$-dimensional Carlson--Simpson tree of $[k]^{<\nn}$ and $V\in \mathrm{Subtr}_m(W)$. The \textit{depth} of $V$ in $W$, denoted by
$\mathrm{depth}_W(V)$, is defined to be the unique integer $i\in\{m,...,d\}$ such that the $m$-level $V(m)$ of $V$ is contained in
the $i$-level $W(i)$ of $W$, or equivalently, $V(m)\in\mathrm{Subs}_m\big(W(i)\big)$. We set
\begin{equation} \label{4e1}
\mathrm{Subtr}_{m}^{\max}(W)=\big\{ V\in\mathrm{Subtr}_m(W): \mathrm{depth}_W(V)=\dim(W) \big\}.
\end{equation}
That is, $\mathrm{Subtr}_{m}^{\max}(W)$ is the set of all $m$-dimensional Carlson--Simpson subtrees of $W$ of maximal depth. Part of our
interest in this subclass is justified by the following simple, though important, fact. Its proof is a rather straightforward consequence
of the relevant definitions.
\begin{fact} \label{4f2}
For every integer $d\meg 1$, every $d$-dimensional Carlson--Simpson tree $W$ of $[k]^{<\nn}$ and every $m\in [d]$ the map
\begin{equation} \label{4e2}
\subtr_m^{\max}(W)\ni V \mapsto V(m)\in \mathrm{Subs}_m\big(W(d)\big)
\end{equation}
is a bijection.
\end{fact}
Combining Theorem \ref{2t1} and Fact \ref{4f2} we get the following corollary.
\begin{cor} \label{4c3}
Let $k\in\nn$ with $k\meg 2$. Also let $d,m,r\in\nn$ with $d\meg m \meg1$ and $r\meg 1$.
If $n\meg \mathrm{GR}(k,d,m,r)$, then for every $n$-dimensional Carlson--Simpson tree $W$ of $[k]^{<\nn}$ and every $r$-coloring of
the set $\mathrm{Subtr}_m^{\max}(W)$ there exists $U\in\subtr_d^{\max}(W)$ such that the set $\subtr_m^{\max}(U)$ is monochromatic.
\end{cor}
The proof of Theorem \ref{4t1} is based on a strengthening of Corollary \ref{4c3}. To state it, it is convenient to introduce the following
definition. For every $k,m,r\in\nn$ with $k\meg 2$ and $m,r\meg 1$ we define the function $g_{k,m,r}:\nn\to\nn$ by the rule
$g_{k,m,r}(n)=0$ if $n<m-1$ and
\begin{equation} \label{4e3}
g_{k,m,r}(n)=\mathrm{GR}(k,n+1,m,r)
\end{equation}
if $n\meg m-1$. We have the following lemma.
\begin{lem} \label{4l4}
Let $k,m,r\in\nn$ with $k\meg 2$ and $m,r\meg 1$. Also let $q,n\in\nn$ with $q\meg 1$ and $n\meg g_{k,m,r}^{(q)}(m)$. Then for every
$n$-dimensional Carlson--Simpson tree $W$ of $[k]^{<\nn}$ and every $r$-coloring of the set $\subtr_{m}(W)$ there exists
$U\in\subtr_{m+q}^{\max}(W)$ with the following property. For every pair $S,T\in\subtr_m(U)$ with $\mathrm{depth}_U(S)=\mathrm{depth}_U(T)$
the Carlson--Simpson trees $S$ and $T$ have the same color.
\end{lem}
\begin{proof}
We fix a coloring $c:\mathrm{Subtr}_m(W)\to [r]$. For every $i\in\{0,...,q\}$ we set $n_i=g_{k,m,r}^{(q-i)}(m)$.
Notice that, by \eqref{4e3}, for every $i\in\{0,...,q-1\}$ we have
\begin{equation} \label{4e4}
n_i=\mathrm{GR}(k,n_{i+1}+1,m,r) \meg n_{i+1}+1\meg n_q=m.
\end{equation}
We select $U_0\in\mathrm{Subtr}^{\mathrm{max}}_{n_0}(W)$. By \eqref{4e4} and Corollary \ref{4c3}, we may construct a family $\{U_1,...,U_q\}$
of Carlson--Simpson subtrees of $U_0$ with the following properties.
\begin{enumerate}
\item[(a)] For every $i\in [q]$ we have $\dim(U_i)=n_i+1$.
\item[(b)] We have $U_1\in \subtr^{\max}_{n_1+1}(U_0)$. Moreover, if $q\meg 2$, then for every $i\in [q-1]$ we have
$U_{i+1}\in \subtr^{\max}_{n_{i+1}}(U'_{i})$ where $U'_i=U_{i}\setminus U_{i}(n_{i}+1)$.
\item[(c)] For every $i\in [q]$ the set $\subtr_m^{\max}(U_i)$ is monochromatic with respect to $c$.
\end{enumerate}
For every $i\in [q]$ let $(c_i, w_0^{(i)}, ..., w_{n_{i}}^{(i)})$ be the generating sequence of $U_i$. We define $U$ to be the
Carlson--Simpson tree of $[k]^{<\nn}$ generated by the sequence
\begin{equation} \label{4e5}
(c_q, w_0^{(q)}, ..., w_{n_q}^{(q)})^{\con}(w_{n_{q-1}}^{(q-1)},...,w_{n_2}^{(2)}, w_{n_1}^{(1)}).
\end{equation}
We will show that $U$ is as desired. Indeed, notice first that
\begin{equation} \label{4e6}
\dim(U)=(n_q+1)+(q-1)=m+q.
\end{equation}
Also observe that $U\in\mathrm{Subtr}^{\mathrm{max}}_{m+q}(U_0)$ and so $U\in\mathrm{Subtr}^{\mathrm{max}}_{m+q}(W)$. Finally let
$\ell\in [q]$ be arbitrary and set $i_{\ell}=q-\ell+1\in [q]$. By the definition of $U$ and (b) above, we see that $U(m+\ell)$ is contained
in $U_{i_\ell}(n_{i_\ell}+1)$. Hence, for every pair   $S,T\in\mathrm{Subtr}_m(U)$ with $\mathrm{depth}_U(S)=\mathrm{depth}_U(T)=m+\ell$ we have
that $S,T\in\mathrm{Subtr}_m^{\max}(U_{i_\ell})$. Invoking (c), we conclude that $c(S)=c(T)$ and the proof is completed.
\end{proof}
We are ready to proceed to the proof of Theorem \ref{4t1}.
\begin{proof}[Proof of Theorem \ref{4t1}]
Let $k\in\nn$ with $k\meg 2$. Also let $d,m,r\in\nn$ with $d\meg m\meg 1$ and $r\meg 1$. We will show that
\begin{equation} \label{4e7}
\mathrm{CS}(k,d,m,r) \mik g_{k,m,r}^{(d\cdot r-m)}(m).
\end{equation}
Indeed, let $n\meg g_{k,m,r}^{(d\cdot r-m)}(m)$ and $W$ be an arbitrary $n$-dimensional Carlson--Simpson tree of $[k]^{<\nn}$. We fix
a coloring $c:\subtr_m(W)\to [r]$. By Lemma \ref{4l4}, there exists a Carlson--Simpson subtree $R$ of $W$ with $\dim(R)=d\cdot r$ such that
for every $S\in\mathrm{Subtr}_m(R)$ the color $c(S)$ of $S$ depends only on the depth of $S$ in $R$. Therefore, by the classical pigeonhole
principle, there exist a subset $I$ of $\{0,..., d\cdot r\}$ with $|I|=d+1$ and $r_0\in [r]$ such that for every $i\in I$ and every
$S\in\mathrm{Subtr}_m(R)$ with $\mathrm{depth}_R(S)=i$ we have $c(S)=r_0$. Let $U$ be any $d$-dimensional Carlson--Simpson subtree
of $R$ which is contained in the set $\bigcup_{i\in I} R(i)$. By the previous discussion, we see that the coloring $c$ restricted on
$\mathrm{Subtr}_m(U)$ is constantly equal to $r_0$. The proof of Theorem \ref{4t1} is thus completed.
\end{proof}


\section{The convolution operation}

\numberwithin{equation}{section}

The concatenation of two finite sequences provides us with a canonical way to ``glue'' a pair of elements of $[k]^{<\nn}$. Our goal in this
section is to describe a different ``gluing'' method which will be of fundamental importance throughout the paper.

The method is particularly easy to grasp for pairs of sequences of given length. Specifically, let $n,m\meg 1$ and fix a subset $L$ of
$\{0,..., n+m-1\}$ of cardinality $n$. Given an element $x$ of $[k]^n$ and an element $y$ of $[k]^m$, the outcome of the ``gluing'' method
for the pair $x,y$ is the unique element $z$ of $[k]^{n+m}$ which is ``equal'' to $x$ on $L$ and to $y$ on the rest of the coordinates.
This simple process can, of course, be extended to arbitrary pairs of $[k]^{<\nn}$. This is the content of the following definition.
\begin{defn} \label{5d1}
Let $k\in\nn$ with $k\meg 2$ and $L=\{l_0<...<l_{|L|-1}\}$ be a nonempty finite subset of $\nn$. For every $i\in\{0,...,|L|-1\}$ we set
\begin{equation} \label{5e1}
L_i=\{l\in L: l<l_i\} \text{ and } \overline{L}_i=\{n\in\nn: n<l_i \text{ and } n\notin L_i\}.
\end{equation}
Also let $n_L=\max(L)-|L|+1$ and set
\begin{equation} \label{5e2}
X_L=[k]^{n_L}.
\end{equation}
We define the \emph{convolution operation} $\cv_L:[k]^{<|L|}\times X_L\to [k]^{<\nn}$ associated to $L$ as follows. For every $i\in\{0,...,|L|-1\}$,
every $t\in [k]^i$ and every $x\in X_L$ we set
\begin{equation} \label{5e3}
\cv_L(t,x)=\big(\mathrm{I}_{L_i}(t),\mathrm{I}_{\overline{L}_i}(x|_{|\overline{L}_i|})\big)\in [k]^{l_i}
\end{equation}
where $\mathrm{I}_{L_i}$ and $\mathrm{I}_{\overline{L}_i}$ are the canonical isomorphisms defined in \S 2.2.

More generally, let $V$ be a Carlson--Simpson tree of $[k]^{<\nn}$ and assume that $L$ is contained in $\{0,...,\dim(V)\}$.
The \emph{convolution operation} $\cv_{L,V}:[k]^{<|L|}\times X_L\to V$ associated to $(L,V)$ is defined by the rule
\begin{equation} \label{5e4}
\cv_{L,V}(t,x) = \mathrm{I}_V\big( \cv_L(t,x)\big)
\end{equation}
where $\mathrm{I}_V$ is the canonical isomorphism defined in \S 2.5.
\end{defn}
Before we proceed let us give a specific example. Let $k=5$ and $L=\{1,3,7,9\}$, and notice that $X_L=[5]^6$. In particular, the convolution
operation $\cv_L$ associated to the set $L$ is defined for pairs in $[5]^{<4}\times [5]^6$. Then for the pair $t=(1,2)$ and $x=(3,5,4,2,4,1)$
we have
\begin{equation} \label{5e5}
\cv_L(t,x)=(3,\mathbf{1},5,\mathbf{2},4,2,4)
\end{equation}
where in \eqref{5e5} we indicated with boldface letters the contribution of $t$.

The rest of this section is devoted to the study of convolution operations. We notice that all properties described below follow by carefully
manipulating the relevant definitions. In fact, once the basic definitions have been properly understood, most of the material of this section
should be regarded as fairly straightforward.

We begin with the following fact.
\begin{fact} \label{5f2}
Let $k\in\nn$ with $k\meg2$. Let $V$ be a Carlson--Simpson tree of $[k]^{<\nn}$ and $L=\{l_0<...<l_{|L|-1}\}$ be a nonempty finite subset
of $\{0,...,\dim(V)\}$. For every $t\in [k]^{<|L|}$ we set
\begin{equation} \label{5e6}
\Omega_t=\big\{ \cv_{L,V}(t,x): x\in X_L\big\}.
\end{equation}
Then for every $t,t'\in [k]^{<|L|}$ with $t\neq t'$ we have $\Omega_t\cap\Omega_{t'}=\varnothing$.
Moreover, for every $i\in\{0,...,|L|-1\}$ the family $\{\Omega_t:t\in[k]^{i}\}$ forms an equipartition of $V(l_i)$.
\end{fact}
\begin{proof} 
By the definition of the convolution operation, we see that $\Omega_t\cap\Omega_{t'}=\varnothing$ if $t\neq t'$. It is also easy to check that
the family $\{\Omega_t:t\in[k]^{i}\}$ forms a partition of $V(l_i)$. Therefore, to complete the proof it is enough to observe that for every
$i\in\{0,..., |L|-1\}$ and every $t\in [k]^i$ we have
\begin{equation} \label{5e7}
\mathrm{I}_V^{-1}(\Omega_t)= \big\{x\in [k]^{l_i}: x|_{L_i}= \mathrm{I}_{L_i}(t) \big\}.
\end{equation}
Clearly this implies that $|\Omega_t|=|\Omega_{t'}|$ for every $t,t'\in [k]^i$. 
\end{proof}
Using similar elementary observations we get the following.
\begin{fact} \label{5f3}
Let $k, V$ and $L$ be as in Fact \ref{5f2}. For every $t\in [k]^{<|L|}$ and every $s\in[k]^{<\nn}$ we set
\begin{equation} \label{5e8}
Y^t_s=\{x\in X_L:\cv_{L,V}(t,x)=s\}.
\end{equation}
Then for every $t\in [k]^{<|L|}$ and every $s,s'\in \Omega_t$ with $s\neq s'$, where $\Omega_t$ is as in \eqref{5e6}, the sets $Y^t_s$ and
$Y^t_{s'}$ are nonempty disjoint subsets of $X_L$. Moreover, the family $\{Y^t_s:s\in \Omega_t\}$ forms an equipartition of $X_L$.
\end{fact}
We will also need the following fact.
\begin{fact} \label{5f4}
Let $k, V$ and $L$ be as in Fact \ref{5f2}. For every $t\in [k]^{<|L|}$ and every $s\in [k]^{<\nn}$ let $\Omega_t$ and $Y^t_s$ be as in
\eqref{5e6} and \eqref{5e8} respectively. Then for every $i\in\{0,...,|L|-1\}$ and every $t,t'\in [k]^i$ there exists a map
$g_{t,t'}:\Omega_t\to\Omega_{t'}$ with the following properties.
\begin{enumerate}
\item[(i)] For every $s\in\Omega_t$ we have that $Y^t_s=Y^{t'}_{g_{t,t'}(s)}$.
\item[(ii)] The map $g_{t,t'}$ is a bijection.
\item[(iii)] The map $g_{t,t'}$ preserves the lexicographical order.
\item[(iv)] If $t$ and $t'$ are $(r,r')$-equivalent for some $r,r'\in[k]$ with $r\neq r'$ (see \S 2.6), then $s$ and $g_{t,t'}(s)$ are
$(r,r')$-equivalent for every $s\in\Omega_t$.
\end{enumerate}
\end{fact}
\begin{proof}
For every $s\in\Omega_t$ we select $x_s\in Y_s^t$ and we set
\begin{equation} \label{5e9}
g_{t,t'}(s)=\cv_{L,V}(t',x_s).
\end{equation}
It is easy to check that $g_{t,t'}$ is well-defined and satisfies the above properties.
\end{proof}
We proceed with the following lemma.
\begin{lem} \label{5l5}
Let $k\in\nn$ with $k\meg2$.  Let $V$ be a Carlson--Simpson tree of $[k]^{<\nn}$ and $L=\{l_0<...<l_{|L|-1}\}$ be a nonempty finite subset
of $\{0,...,\dim(V)\}$. Also let $t\in [k]^{<|L|}$ and $A\subseteq [k]^{<\nn}$ and set $B=\cv_{L,V}^{-1}(A)$. Then the following hold.
\begin{enumerate}
\item[(i)] We have $\dens_{\Omega_t}(A)=\dens_{\{t\}\times X_L}(B)$ where $\Omega_t$ is as in \eqref{5e6}.
\item[(ii)] For every $i\in \{0,...,|L|-1\}$ we have $\dens_{V(l_i)}(A)=\dens_{[k]^i\times X_L}(B)$.
\end{enumerate}
\end{lem}
\begin{proof}
For every $s\in\Omega_t$ let $Y^t_s$ be as in \eqref{5e8}. By the definition of $B$ and $\Omega_t$,
\begin{eqnarray} \label{5e10}
B\cap(\{t\}\times X_L) & = & \big\{(t,x):\cv_{L,V}(t,x)\in A\cap \Omega_t\big\}  \\
& = & \bigcup_{s\in A\cap\Omega_t} \big\{(t,x):\cv_{L,V}(t,x)=s\big\} \nonumber \\
& = & \bigcup_{s\in A\cap\Omega_t} \{t\}\times Y_s^t. \nonumber
\end{eqnarray}
By Fact \ref{5e3}, for every $s\in\Omega_t$ we have
\begin{equation} \label{5e11}
\frac{|Y_s^t|}{|X_L|}=\frac{1}{|\Omega_t|}.
\end{equation}
Therefore,
\begin{eqnarray} \label{5e12}
\dens_{\{t\}\times X_L}(B) & = & \frac{|B\cap(\{t\}\times X_L)|}{|\{t\}\times X_L|}
\stackrel{\eqref{5e10}}{=} \sum_{s\in A\cap\Omega_t}\frac{|\{t\}\times Y_s^t|}{|\{t\}\times X_L|}\\
& = & \sum_{s\in A\cap\Omega_t}\frac{|Y_s^t|}{|X_L|} \stackrel{\eqref{5e11}}{=}\frac{|A\cap\Omega_t|}{|\Omega_t|}= \dens_{\Omega_t}(A). \nonumber
\end{eqnarray}
This completes the proof of the first part of the lemma. To see that part (ii) is satisfied, let $i\in\{0,...,|L|-1\}$ be arbitrary.
By Fact \ref{5f2}, we see that
\begin{equation} \label{5e13}
\dens_{V(l_i)}(A)=\ave_{t\in[k]^i}\dens_{\Omega_t}(A).
\end{equation}
By \eqref{5e13} and the first part of the lemma, the result follows.
\end{proof}
The final three lemmas of this section contain some coherence properties of convolution operations. The first one shows that
convolution operations preserve Carlson--Simpson trees.
\begin{lem}\label{5l6}
Let $k, V$ and $L$ be as in Lemma \ref{5l5}. Also let $W$ be a Carlson--Simpson subtree of $[k]^{<|L|}$ and $x\in X_L$. Then, setting
\begin{equation} \label{5e14}
W_x=\big\{ \cv_{L,V}(w,x): w\in W \big\},
\end{equation}
we have that $W_x$ is a Carlson--Simpson subtree of $V$ of dimension $\dim(W)$. Moreover, for every $i\in\{0,...,\dim(W)\}$ we have
\begin{equation} \label{5e15}
W_x(i)=\big\{ \cv_{L,V}(w,x): w\in W(i)\big\}.
\end{equation}
\end{lem}
\begin{proof}
Let $\ell=|L|$ and $\{l_0<...<l_{\ell-1}\}$ be the increasing enumeration of $L$. Set
\begin{equation} \label{5e16}
J_0=\{n\in\nn:n<l_0\} \text{ and } J_i=\{n\in\nn:l_{i-1}-i+1\mik n\mik l_{i}-i-1\}
\end{equation}
for every $i\in [\ell-1]$. Observe that the family $\{J_0, J_1,..., J_{\ell-1}\}$ forms a partition of $\{0,...,l_{\ell-1}-\ell\}$
into successive intervals some of which are possibly empty. Let
\begin{equation} \label{5e17}
c=\mathrm{I}_{J_0}^{-1}(x|_{J_0}).
\end{equation}
Also, for every $i\in\{0,...,\ell-2\}$ we define
\begin{equation} \label{5e18}
w_{i}=v^\con \mathrm{I}^{-1}_{J_{i+1}}(x|_{J_{i+1}}).
\end{equation}
Clearly $w_i$ is a left variable word over $k$ for every $i\in\{0,...,\ell-2\}$. Let $S_x$ be the Carlson--Simpson tree generated by the
sequence $(c,w_0,..., w_{\ell-2})$ and observe that $\cv_L(t,x)=\mathrm{I}_{S_x}(t)$ for every $t\in [k]^{<|L|}$. Hence, for every
$t\in[k]^{<|L|}$ we have
\begin{equation} \label{5e19}
\cv_{L,V}(t,x)=\mathrm{I}_V\big(\mathrm{I}_{S_x}(t)\big).
\end{equation}
Using \eqref{5e19} and invoking the definition of $W_x$ in \eqref{5e14}, the result follows.
\end{proof}
The next result enables us to transfer quantitative information from the space $[k]^{<\nn}$ to the space on which the convolution operations
are acting.
\begin{lem}\label{5l7}
Let $k, V$ and $L$ be as in Lemma \ref{5l5}. Also let $W$ be a Carlson--Simpson subtree of $[k]^{<|L|}$ and $x\in X_L$ and define $W_x$
as in \eqref{5e14}. If $A$ is a subset of $[k]^{<\nn}$ and $B=\cv_{L,V}^{-1}(A)$, then for every $i\in\{0,...,\dim(W)\}$ we have
\begin{equation} \label{5e20}
\dens_{W_x(i)}(A)=\dens_{W(i)\times \{x\}}(B).
\end{equation}
\end{lem}
\begin{proof}
We define the Carlson--Simpson tree $S_x$ exactly as we did in the proof of Lemma \ref{5l6}. By \eqref{5e19}, for every $t\in[k]^{<|L|}$ we have
that $(t,x)\in B$  if and only if $\mathrm{I}_V(\mathrm{I}_{S_x}(t))\in A$ and the result follows.
\end{proof}
We close this section with the following lemma.
\begin{lem} \label{5l8}
Let $k, V$ and $L$ be as in Lemma \ref{5l5}. Also let $W$ be a Carlson--Simpson subtree of $[k]^{<|L|}$ and $A$ be a subset of $[k]^{<\nn}$.
Then for every $i\in\{0,...,\dim(W)\}$ we have
\begin{equation} \label{5e21}
\dens_{\cv_{L,V}( W(i)\times X_L)}(A)=\ave_{x\in X_L}\dens_{W_x(i)}(A)
\end{equation}
where $W_x$ is as in \eqref{5e14}. In particular, for every $i\in \{0,..., |L|-1\}$ we have
\begin{equation} \label{5e22}
\dens_{V(l_i)}(A)=\ave_{x\in X_L}\dens_{R_x(i)}(A)
\end{equation}
where $R_x=\{ \cv_{L,V}(t,x): t\in [k]^{<|L|}\}$ for every $x\in X_L$.
\end{lem}
\begin{proof}
Let $i\in \{0,...,\dim(W)\}$. There exists a unique $l\in \{0,..., |L|-1\}$ such that $W(i)$ is contained in $[k]^l$. By Fact \ref{5f2}, the
family $\{\Omega_t:t\in W(i)\}$ forms an equipartition of $\cv_{L,V}(W(i)\times X_L)$. Therefore, setting $B=\cv_{L,V}^{-1}(A)$, by Lemma \ref{5l5},
\begin{eqnarray} \label{5e23}
\dens_{\cv_{L,V}(W(i)\times X_L)}(A) &=&   \ave_{t\in W(i)}\dens_{\Omega_t}(A) \\
&= & \ave_{t\in W(i)}\dens_{\{t\}\times X_L}(B) \nonumber \\
& = & \dens_{W(i)\times X_L}(B) \nonumber \\
& = & \ave_{x\in X_L}\dens_{W(i)\times \{x\}}(B) \nonumber \\
& \stackrel{\eqref{5e20}}{=} & \ave_{x\in X_L}\dens_{W_x(i)}(A). \nonumber
\end{eqnarray}
Finally notice that $V(l_i)=\cv_{L,V}([k]^i\times X_L)$ for every $i\in\{0,...,|L|-1\}$. Therefore, equality \eqref{5e22} follows
by \eqref{5e21} and the proof is completed.
\end{proof}


\section{Iterated convolutions}

\numberwithin{equation}{section}

Our goal in this section is to study iterations of convolution operations. We remark that this material will be used only in \S 9.
The exact statements that we need are isolated in \S 6.2.
\medskip

\noindent 6.1. \textbf{Definitions and basic properties.} We start with the following definition.
\begin{defn} \label{6d1}
Let $\bl=(L_n)_{n=0}^d$ be a finite sequence of nonempty finite subsets of $\nn$. Also let $k\in\nn$ with $k\meg 2$ and $\bv=(V_n)_{n=0}^d$
be a finite sequence of Carlson--Simpson trees of $[k]^{<\nn}$ with the same length as $\bl$. We say that the pair $(\bl,\bv)$ is
$k$-\emph{compatible}, or simply \emph{compatible} if $k$ is understood, provided that for every $n\in \{0,...,d\}$ we have
$L_n\subseteq \{0,...,\dim(V_n)\}$ and, if $n<d$, then $V_{n+1}\subseteq [k]^{<|L_n|}$.
\end{defn}
Notice that if $(\bl,\bv)$ is a compatible pair and $\bl'$,$ \bv'$ are initial subsequences of $\bl$, $\bv$ with common length, then the pair
$(\bl',\bv')$ is also compatible. Also observe that for every compatible pair $(\bl,\bv)=\big((L_n)_{n=0}^d,(V_n)_{n=0}^d\big)$ and every
$n\in \{0,...,d\}$ we can define the convolution operation $\cv_{L_n,V_n}: [k]^{<|L_n|}\times\ X_{L_n}\to V_n$ associated to $(L_n,V_n)$ as
described in \S 5. What Definition \ref{6d1} guarantees is that for compatible pairs we can iterate these operations. This is the content of
the following definition.
\begin{defn} \label{6d2}
Let $k\in\nn$ with $k\meg 2$ and $(\bl,\bv)=\big((L_n)_{n=0}^d,(V_n)_{n=0}^d\big)$ be a $k$-compatible pair. We set
\begin{equation} \label{6e1}
X_{\bl}=\prod_{n=0}^d X_{L_n}.
\end{equation}
By recursion on $d$ we define the \emph{iterated convolution operation}
\begin{equation} \label{6e2}
\cv_{\bl,\bv}:[k]^{<|L_d|}\times X_{\bl}\to V_0
\end{equation}
associated to $(\bl,\bv)$ as follows. For $d=0$ this is the $\cv_{L_0,V_0}$ convolution operation defined in \eqref{5e4}.

If $d\meg 1$, then let $\bl'=(L_n)_{n=0}^{d-1}$ and $\bv'=(V_n)_{n=0}^{d-1}$ and assume that the operation $\cv_{\bl',\bv'}$ has been defined. We set
\begin{equation} \label{6e3}
\cv_{\bl,\bv}(s,x_0,...,x_d)=\cv_{\bl',\bv'}\big(\cv_{L_d,V_d}(s,x_d),x_0,...,x_{d-1}\big)
\end{equation}
for every $s\in[k]^{<|L_d|}$ and every $(x_0,...,x_d)\in X_{\bl}$. In this case, the \emph{quotient map}
\begin{equation} \label{6e4}
\qv_{\bl,\bv}: [k]^{<|L_d|}\times X_{\bl} \to [k]^{<|L_{d-1}|}\times X_{\bl'}
\end{equation}
associated to $(\bl,\bv)$ is defined by the rule
\begin{equation} \label{6e5}
\qv_{\bl,\bv}(t,\bx,x)=\big( \cv_{L_d,V_d}(t,x),\bx\big)
\end{equation}
for every $t\in [k]^{<|L_d|}$ and every $(\bx,x)\in X_{\bl'}\times X_{L_d}$.
\end{defn}
The rest of this subsection is devoted to several lemmas establishing properties of iterated convolutions. Just as in \S 5, all this material
follows by carefully manipulating the relevant definitions. We begin with the following elementary fact.
\begin{fact} \label{6f3}
Let $k\meg 2$ and $(\bl,\bv)=\big((L_n)_{n=0}^d,(V_n)_{n=0}^d\big)$ be a $k$-compatible pair. If $d\meg1$ and
$(\bl',\bv')=\big((L_n)_{n=0}^{d-1}, (V_n)_{n=0}^{d-1}\big)$, then $\cv_{\bl,\bv}=\cv_{\bl',\bv'}\circ \qv_{\bl,\bv}$.
\end{fact}
The next two lemmas are multidimensional analogues of Lemmas \ref{5l6} and \ref{5l7}.
\begin{lem} \label{6l4}
Let $k\meg 2$ and $(\bl,\bv)=\big((L_n)_{n=0}^d,(V_n)_{n=0}^d\big)$ be a $k$-compatible pair. Also let $W$ be a Carlson--Simpson subtree
of $[k]^{<|L_d|}$ and $\bx\in X_{\bl}$. Then the set
\begin{equation} \label{6e6}
W_{\bx}=\big\{ \cv_{\bl,\bv}(w,\bx): w\in W\big\}
\end{equation}
is a Carlson--Simpson subtree of $V_0$ with the same dimension as $W$. Moreover, for every $i\in\{0,...,\dim(W)\}$ we have
\begin{equation}\label{6e7}
W_{\bx}(i)=\big\{ \cv_{\bl,\bv}(w,\bx): w\in W(i)\big\}.
\end{equation}
\end{lem}
\begin{proof}
Both assertions are proved by induction on $d$ and using similar arguments. We will give the details only for the first one. The case ``$d=0$"
is the content of Lemma \ref{5l6}. So, let $d\meg1$ and assume that the result has been proved up to $d-1$. Fix a compatible pair
$(\bl,\bv)=\big((L_n)_{n=0}^d,(V_n)_{n=0}^d\big)$ and let $W$ and $\bx$ be as in the statement of the lemma. Write $\bx=(x_0,...,x_d)$ and set
$\bx'=(x_0,...,x_{d-1})$ and $S=\{\cv_{L_d,V_d}(w,x_d):w\in W\}$. Also let $\bl'=(L_n)_{n=0}^{d-1}$ and $\bv'=(V_n)_{n=0}^{d-1}$ and observe that
the pair $(\bl',\bv')$ is compatible. By Lemma \ref{5l6}, $S$ is a Carlson--Simpson subtree of $V_d$ with $\dim(S)=\dim(W)$. By Definition
\ref{6d1}, we have that $S$ is contained in $[k]^{<|L_{d-1}|}$. Therefore, applying the inductive assumptions for the pair $(\bl',\bv')$, we see
that $S_{\bx'}$ is a Carlson--Simpson subtree of $V_0$ of dimension $\dim(W)$. Noticing that $S_{\bx'}$ coincides with $W_{\bx}$ the result follows.
\end{proof}
\begin{lem}\label{6l5}
Let $k,\bl,\bv,W$ and $\bx$ be as Lemma \ref{6l4}. Also let $A$ be a subset of $[k]^{<\nn}$ and set $B=\cv_{\bl,\bv}^{-1}(A)$. Then for every
$i\in\{0,...,\dim(W)\}$ we have
\begin{equation} \label{6e8}
\dens_{W(i)\times\{\bx\}}(B)=\dens_{W_{\bx}(i)}(A).
\end{equation}
\end{lem}
\begin{proof}
By induction on $d$. The case ``$d=0$" follows from Lemma \ref{5l7}. Let $d\meg1$ and assume that the result has been proved up to $d-1$.
Fix a $k$-compatible pair $(\bl,\bv)=\big((L_n)_{n=0}^d,(V_n)_{n=0}^d\big)$ and let $W,\bx,A$ and $B$ be as in the statement of the lemma. Write
$\bx=(x_0,...,x_d)$ and define $\bx', S, \bl'$ and $\bv'$ precisely as in the proof of Lemma \ref{6l4}. We set $C=\cv_{\bl',\bv'}^{-1}(A)$.
For every $\by\in X_{\bl'}$ let $C_{\by}$ and $B_{\by}$ be the sections of $C$ and $B$ at $\by$. By Fact \ref{6f3}, we see
that $B_{\by}=\cv_{L_d,V_d}^{-1}(C_{\by})$. Hence,
\begin{eqnarray} \label{6e9}
\dens_{W(i)\times\{\bx\}}(B) & = & \dens_{W(i)\times\{\bx'\}\times\{x_d\}}(B) \\
& = & \dens_{W(i)\times\{\bx'\}\times\{x_d\}}(B_{\bx'}\times \{\bx'\}) \nonumber \\
& = & \dens_{W(i)\times\{x_d\}}(B_{\bx'}) \nonumber \\
& = & \dens_{W(i)\times\{x_d\}}\big(\cv_{L_d,V_d}^{-1}(C_{\bx'})\big) \nonumber
\end{eqnarray}
Invoking Lemma \ref{5l7} we have
\begin{equation} \label{6e10}
\dens_{W(i)\times\{x_d\}}\big(\cv_{L_d,V_d}^{-1}(C_{\bx'})\big) = \dens_{S(i)}(C_{\bx'}).
\end{equation}
Next observe that
\begin{equation} \label{6e11}
\dens_{S(i)}(C_{\bx'})= \dens_{S(i)\times \{\bx'\}}(C_{\bx'}\times\{\bx'\})=\dens_{S(i)\times\{\bx'\}}(C).
\end{equation}
As we have already pointed out in the proof of Lemma \ref{6l4}, the set $S_{\bx'}$ coincides with $W_{\bx}$. Since $C=\cv^{-1}_{\bl',\bv'}(A)$,
we may apply our inductive hypothesis to the Carlson--Simpson tree $S$, the element $\bx'$ and the set $A$ to infer that
\begin{equation}\label{6e12}
\dens_{S(i)\times\{\bx'\}}(C)=\dens_{S_{\bx'}(i)}(A)=\dens_{W_{\bx}(i)}(A).
\end{equation}
Combining equalities (\ref{6e9}) up to (\ref{6e12}) the result follows.
\end{proof}
We proceed with the following lemma.
\begin{lem} \label{6l6}
Let $k\meg 2$ and $(\bl,\bv)=\big((L_n)_{n=0}^d,(V_n)_{n=0}^d\big)$ be a $k$-compatible pair. Assume that $d\meg 1$ and let 
$\bl'=(L_n)_{n=0}^{d-1}$ and $\bv'=(V_n)_{n=0}^{d-1}$. Let $C$ be a subset of $[k]^{<|L_{d-1}|}\times X_{\bl'}$ and set $B=\qv_{\bl,\bv}^{-1}(C)$. 
Finally let $t\in [k]^{<|L_d|}$ and set
\begin{equation} \label{6e13}
\Omega_t=\{ \cv_{L_d,V_d}(t,x):x\in X_{L_d}\}.
\end{equation}
Then $\qv_{\bl,\bv}^{-1}(\Omega_t\times X_{\bl'})=\{t\}\times X_{\bl}$ and
\begin{equation}\label{6e14}
\dens_{\Omega_t\times X_{\bl'}}(C)= \dens_{\{t\}\times X_{\bl}}(B).
\end{equation} 
\end{lem}
\begin{proof}
By the definition of the quotient map $\qv_{\bl,\bv}$ in \eqref{6e5}, we have
\begin{equation} \label{6e15}
\qv_{\bl,\bv}^{-1}(\Omega_t\times X_{\bl'})=\cv_{L_d,V_d}^{-1}(\Omega_t)\times X_{\bl'}.
\end{equation}
Since $\cv_{L_d,V_d}^{-1}(\Omega_t)=\{t\}\times X_{L_d}$, by \eqref{6e15}, we see that
$\qv_{\bl,\bv}^{-1}(\Omega_t\times X_{\bl'})=\{t\}\times X_{\bl}$. 

Now for every $\bx'\in X_{\bl'}$ let $C_{\bx'}$ and $B_{\bx'}$ be the sections of $C$ and $B$ at $\bx'$ respectively. Observe that
$C_{\bx'}\subseteq [k]^{<|L_{d-1}|}$ and $B_{\bx'}\subseteq [k]^{<|L_{d}|}\times X_{L_d}$. Also notice that $B_{\bx'}=\cv_{L_d,V_d}^{-1}(C_{\bx'})$
for every $\bx'\in X_{\bl'}$. Hence, by Lemma \ref{5l5},
\begin{equation}\label{6e16}
\dens_{\{t\}\times X_{L_d}}(B_{\bx'})= \dens_{\Omega_t}(C_{\bx'})
\end{equation}
for every $\bx'\in X_{\bl'}$. Therefore,
\begin{eqnarray} \label{6e17}
\dens_{\{t\}\times X_{\bl}}(B) & = & \ave_{\bx'\in X_{\bl'}}\dens_{\{t\}\times X_{L_d}}(B_{\bx'}) \\
& \stackrel{\eqref{6e16}}= & \ave_{\bx'\in X_{\bl'}}\dens_{\Omega_t}(C_{\bx'}) = \dens_{\Omega_t\times X_{\bl'}}(C) \nonumber
\end{eqnarray}
as desired.
\end{proof}
We close this subsection with the following consequence of Lemma \ref{6l6}.
\begin{cor} \label{6c7}
Let $k\meg 2$ and $(\bl,\bv)=\big((L_n)_{n=0}^d,(V_n)_{n=0}^d\big)$ be a $k$-compatible pair. Assume that $d\meg1$ and let
$\bl'=(L_n)_{n=0}^{d-1}$ and $\bv'=(V_n)_{n=0}^{d-1}$. Let $A$ be a subset of $[k]^{<\nn}$ and set $A^d=\cv_{\bl,\bv}^{-1}(A)$ and
$A^{d-1}=\cv_{\bl',\bv'}^{-1}(A)$. Then for every $t\in [k]^{<|L_d|}$
\begin{equation} \label{6e18}
\dens_{X_{\bl}}(A^d_t) =\ave_{s\in \Omega_t} \dens_{X_{\bl'}}(A^{d-1}_s)
\end{equation}
where $A^d_t$ is the section of $A^d$ at $t$, $\Omega_t\subseteq V_d\subseteq [k]^{<|L_{d-1}|}$ is as in \eqref{6e13} and $A^{d-1}_s$ is the
section of $A^{d-1}$ at $s$.
\end{cor}
\begin{proof} 
We fix $t\in [k]^{<|L_d|}$. Notice that
\begin{equation} \label{6e19}
A^d=\cv_{\bl,\bv}^{-1}(A)=\qv_{\bl,\bv}^{-1}\big(\cv_{\bl',\bv'}^{-1}(A)\big)=\qv_{\bl,\bv}^{-1}(A^{d-1}).
\end{equation}
By \eqref{6e19} and  Lemma \ref{6l6} applied to the sets ``$B=A^d$'' and ``$C=A^{d-1}$'' we obtain
\begin{eqnarray} \label{6e20}
\dens_{X_{\bl}}(A^d_t) & = &  \dens_{\{t\}\times X_{\bl}}(A^d) \\
& \stackrel{\eqref{6e14}}{=} & \dens_{\Omega_t\times X_{\bl'}}(A^{d-1}) = \ave_{s\in \Omega_t} \dens_{X_{\bl'}}(A^{d-1}_s) \nonumber
\end{eqnarray}
and the proof is completed.
\end{proof}
\noindent 6.2. \textbf{Consequences.} As we have already mentioned, in this subsection we will collect some results which will be of particular
importance in \S 9. The first two of them follow by repeated applications of Corollary \ref{6c7}. The details are left to the reader.
\begin{cor} \label{6c8}
Let $k\meg 2$ and $(\bl,\bv)=\big((L_n)_{n=0}^d,(V_n)_{n=0}^d\big)$ be a $k$-compatible pair. Let $0<\gamma\mik 1$ and $A$ be a subset of
$[k]^{<\nn}$. We set $A^d=\cv_{\bl,\bv}^{-1}(A)$ and $A^0=\cv_{L_0,V_0}^{-1}(A)$. Suppose that $\dens_{X_{L_0}}(A^0_s)\meg \gamma$ for every
$s\in [k]^{<|L_0|}$ where $A^0_s$ is the section of $A^0$ at $s$. If $t\in [k]^{<|L_d|}$ and $A^d_t$ is the section of $A^d$ at $t$,
then $\dens_{X_{\bl}}(A^d_t)\meg\gamma$.
\end{cor}
\begin{cor} \label{6c9}
Let $k\meg 2$ and $(\bl,\bv)=\big((L_n)_{n=0}^d,(V_n)_{n=0}^d\big)$ be a $k$-compatible pair. Let $0<\lambda\mik 1$ and $A$ and $B$ be two
subsets of $[k]^{<\nn}$ with $A\subseteq B$. We set $A^d=\cv_{\bl,\bv}^{-1}(A)$,  $A^{0}=\cv_{L_0,V_0}^{-1}(A)$, $B^d=\cv_{\bl,\bv}^{-1}(B)$
and $B^{0}=\cv_{L_0,V_0}^{-1}(B)$. Suppose that $\dens_{X_{L_0}}(A^0_s)\meg \lambda \cdot \dens_{X_{L_0}}(B^0_s)$ for every $s\in [k]^{<|L_0|}$
where $A^0_s$ and $B^0_s$ are the sections of $A^0$ and $B^0$ at $s$. If $t\in [k]^{<|L_d|}$, then $\dens_{X_{\bl}}(A^d_t)\meg \lambda \cdot
\dens_{X_{\bl}}(B^d_t)$ where $A^d_t$ and $B^d_t$ are the sections of $A^d$ and $B^d$ at $t$.
\end{cor}
The final result is a consequence of Lemma \ref{6l6}.
\begin{cor} \label{6c10}
Let $k\meg 2$ and $(\bl,\bv)=\big((L_n)_{n=0}^d,(V_n)_{n=0}^d\big)$ be a $k$-compatible pair. Assume that $d\meg 1$ and let
$\bl'=(L_n)_{n=0}^{d-1}$ and $\bv'=(V_n)_{n=0}^{d-1}$. Let $t\in [k]^{<|L_d|}$ and $C_0$ be a nonempty subset of $\Omega_{t}\times X_{\bl'}$
where $\Omega_t$ is as in \eqref{6e13}. Also let $C_1\subseteq [k]^{<|L_{d-1}|}\times X_{\bl'}$. Then, setting $B_i=\qv_{\bl,\bv}^{-1}(C_i)$
for every $i\in\{0,1\}$, we have that $\dens_{C_0}(C_1)=\dens_{B_0}(B_1)$.
\end{cor}
\begin{proof} 
Notice that $\qv_{\bl,\bv}^{-1}(C_0\cap C_1)=B_0\cap B_1$. Since $\qv_{\bl,\bv}^{-1}(\Omega_t\times X_{\bl'})=\{t\}\times X_{\bl}$,
we see that $B_0\subseteq \{t\}\times X_{\bl}$. Therefore, 
\begin{eqnarray} \label{6e21}
\dens_{C_0}(C_1) & = & \frac{|C_0\cap C_1|}{|C_0|}  = \frac{\dens_{\Omega_t\times X_{\bl'}}(C_0\cap C_1)}{\dens_{\Omega_t\times X_{\bl'}}(C_0)} \\
& \stackrel{\eqref{6e14}}{=} & \frac{\dens_{\{t\}\times X_{\bl}}(B_0\cap B_1)}{\dens_{\{t\}\times X_{\bl}}(B_0)} \nonumber \\
& = & \frac{|B_0\cap B_1|}{|B_0|} = \dens_{B_0}(B_1) \nonumber
\end{eqnarray}
as desired.
\end{proof}


\section{Preliminary tools for the proof of Theorem B}

\numberwithin{equation}{section}

In this section we will gather some results that are part of the proof of Theorem B but are not directly related to the main argument.
Specifically, in \S 7.1 we prove the first instance of Theorem B which can be seen as a variant of the classical Sperner Theorem \cite{Sp}.
In \S 7.2 we show how one can estimate the number $\dcs(k,m+1,\delta)$ assuming that the numbers $\dcs(k,m,\beta)$ have been defined for every
$0<\beta\mik 1$. This result is part of an inductive scheme that we will discuss in detail in \S 8.1. Finally, in \S 7.3 we present some consequences.
\medskip

\noindent 7.1. \textbf{Estimating the numbers $\dcs(2,1,\delta)$.} We have the following proposition.
\begin{prop} \label{7p1}
Let $0<\delta\mik 1$. Also let $A$ be a subset of $[2]^{<\nn}$ and $N$ be a finite subset of $\nn$ such that
\begin{equation} \label{7e1}
|N|\meg \reg\big(2,\cs(2,\lceil 17\delta^{-2}\rceil,1,2)+1,1,\delta/4\big).
\end{equation}
If $|A\cap [2]^n|\meg \delta 2^n$ for every $n\in N$, then there exists a Carlson--Simpson line $R$ of $[2]^{<\nn}$ which is
contained in $A$. In particular,
\begin{equation} \label{7e2}
\dcs(2,1,\delta)\mik \reg\big(2,\cs(2,\lceil 17\delta^{-2}\rceil,1,2)+1,1,\delta/4\big).
\end{equation}
\end{prop}
We should point out that the estimate for the numbers $\dcs(2,1,\delta)$ obtained by Proposition \ref{7p1} is rather weak and far from
being optimal. However, the proof of Proposition \ref{7p1} is conceptually close to the proof of the general case of Theorem B, and as such,
should serve as a motivating introduction to the main argument.

We start with the following lemma.
\begin{lem} \label{7l2}
Let $A$ and $N$ be as in Proposition \ref{7p1}. Then there exists $L\subseteq N$ with
\begin{equation} \label{7e3}
|L|=\cs(2,\lceil 17\delta^{-2}\rceil,1,2)+1
\end{equation}
and satisfying the following property. Let $\cv_L:[2]^{<|L|}\times X_L\to [2]^{<\nn}$ be the convolution operation associated to $L$
and set $B=\cv_{L}^{-1}(A)$. Then for every $t\in [2]^{<|L|}$ we have $\dens(B_t)\meg 3\delta/4$ where $B_t=\{x\in X_L:(t,x)\in B\}$ is the section
of $B$ at $t$.
\end{lem}
\begin{proof}
By Lemma \ref{3l2} and our assumptions on the size of the set $N$, there exists a subset $L$ of $N$ with $|L|=\cs(2,\lceil 17\delta^{-2}\rceil,1,2)
+1$ and such that the family $\mathcal{F}:=\{A\}$ is $(\delta/4,L)$-regular. Write the set $L$ in increasing order as $\{l_0<...<l_{|L|-1}\}$ and
for every  $i\in \{0,...,|L|-1\}$ let $L_i$ and $\overline{L}_i$ be as in \eqref{5e1}. Since $|A\cap [2]^n|\meg \delta 2^n$ for every $n\in L$ and
the singleton $\{A\}$ is $(\delta/4,L)$-regular, we see that
\begin{equation} \label{7e4}
\dens\big(\{w\in [2]^{\overline{L}_i}: (y,w)\in A\cap [2]^{l_i}\}\big)\meg 3\delta/4
\end{equation}
for every $i\in\{0,...,|L|-1\}$ and every $y\in [2]^{L_i}$. By the definition of the convolution operation in \eqref{5e3}, for every $i\in\{0,...,|L|-1\}$
and every $t\in [2]^i$ we have
\begin{equation} \label{7e5}
\Omega_t\stackrel{\eqref{5e6}}{=} \big\{ \cv_{L}(t,x):x\in X_L\big\}=\big\{z\in [2]^{l_i}: z|_{L_i}= \mathrm{I}_{L_i}(t)\big\}.
\end{equation}
Thus, by \eqref{7e4} applied to ``$y=\mathrm{I}_{L_i}(t)$'', we obtain
\begin{equation}\label{7e6}
\dens_{\Omega_t}(A)= \dens\big(\{w\in [2]^{\overline{L}_i}: \big( \mathrm{I}_{L_i}(t),w\big)\in A\cap [2]^{l_i}\}\big)\meg 3\delta/4.
\end{equation}
Finally, by Lemma \ref{5l5}, we have
\begin{equation}\label{7e7}
\dens_{\Omega_t}(A)= \dens_{\{t\}\times X_L}(B)= \dens_{X_L}(B_t).
\end{equation}
Combining \eqref{7e6} and \eqref{7e7} the result follows.
\end{proof}
For the next step of the proof of Proposition \ref{7p1} we need to introduce some terminology. Specifically, let $W$ be an $m$-dimensional
Carlson--Simpson tree of $[2]^{<\nn}$ and $(c,w_0,...,w_{m-1})$ be its generating sequence. Also let $t,t'\in W$.  We say that $t'$ is a
\textit{successor} of $t$ in $W$ if there exist $i,j\in\{0,...,m-1\}$ with $i\mik j$ as well as $a_i,...,a_j\in [2]$ such that 
$t'=t^{\con}w_i(a_i)^{\con}...^{\con}w_j(a_j)$. If, in addition, we have $a_i=1$, then we say that $t'$ is a \textit{left successor} of $t$ in $W$.
\begin{lem} \label{7l3}
Let $L$ and $\{B_t: t\in [2]^{<|L|}\}$  be as in Lemma \ref{7l2}. Then there exists a Carlson--Simpson subtree $W$ of $[2]^{<|L|}$ with
$\dim(W)=\lceil 17\delta^{-2}\rceil$ and such that
\begin{equation} \label{7e8}
\dens(B_t\cap B_{t'})\meg \delta^2/16
\end{equation}
for every $t,t'\in W$ with $t'$ left successor of $t$ in $W$.
\end{lem}
\begin{proof}
For every Carlson--Simpson line $S$ of $[2]^{<\nn}$ let us denote by $(c_S, w_S)$ its generating sequence. We set
\begin{equation} \label{7e9}
\mathcal{L}=\Big\{ S\in\subtr_1\big([2]^{<|L|}\big): \dens(B_{c_S}\cap B_{c_S^{\con}w_S(1)})\meg \delta^2/16\Big\}.
\end{equation}
By Theorem \ref{4t1} and \eqref{7e3}, there exists a Carlson--Simpson subtree $W$ of $[2]^{<|L|}$ with $\dim(W)=\lceil 17 \delta^{-2}\rceil$ such
that either $\subtr_1(W)\subseteq \mathcal{L}$ or $\subtr_1(W)\cap \mathcal{L}=\varnothing$. Observe that for every $t,t'\in W$ we have that $t'$
is a left successor of $t$ in $W$ if and only if there exists a Carlson--Simpson line $S$ of $W$ such that $t=c_S$ and $t'=c_S^\con w_S(1)$.
Therefore, the proof will be completed once we show that $\subtr_1(W)\cap \mathcal{L}\neq\varnothing$. To this end we argue as follows. Let
$d=\dim(W)=\lceil 17 \delta^{-2}\rceil$ and $(c,w_0,...,w_{d-1})$ be the generating sequence of $W$. We set $t_0=c$ and
$t_i=c^\con w_0(1)^\con...^\con w_{i-1}(1)$ for every $i\in [d]$. By our assumptions, we have $\dens(B_{t_i})\meg 3\delta/4$ for every
$i\in \{0,...,d\}$. Hence, by Lemma \ref{2l7} applied for ``$\ee=3\delta/4$'' and ``$\theta=\delta/4$'', there exist $i,j\in \{0,...,d\}$
with $i<j$ and such that $\dens(B_{t_i}\cap B_{t_j})\meg \delta^2/16$. If $R$ is the unique Carlson--Simpson line of $W$ with $c_R=t_i$
and $w_R=w_i^{\con}...^{\con}w_{j-1}$, then the previous discussion implies that $R\in\mathcal{L}$, as desired.
\end{proof}
The following lemma is the last step towards the proof of Proposition \ref{7p1}.
\begin{lem} \label{7l4}
Let $W$ be the Carlson--Simpson tree obtained by Lemma \ref{7l3}. Then $W$ contains a Carlson--Simpson line $S$ such that
\begin{equation} \label{7e10}
\bigcap_{t\in S} B_t\neq \varnothing.
\end{equation}
\end{lem}
\begin{proof}
As in Lemma \ref{7l3}, let $d=\dim(W)=\lceil 17 \delta^{-2}\rceil$ and $(c,w_0,...,w_{d-1})$ be the generating sequence of $W$.
For every $i\in \{0,...,d-2\}$ we set
\begin{equation} \label{7e11}
t_i=c^{\con}w_0(2)^{\con}...^{\con}w_i(2) \text{ and } s_i=t_i^{\con}w_{i+1}(1)^{\con}...^{\con}w_{d-1}(1).
\end{equation}
Observe that $s_i$ is a left successor of $t_i$ in $W$. Therefore, by Lemma \ref{7l3}, setting $C_i=B_{t_i}\cap B_{s_i}$ we have $\dens(C_i)\meg\delta^2/16$
for every $i\in \{0,...,d-2\}$. Also let $s_{d-1}=c^{\con}w_0(2)^{\con}...^{\con}w_{d-1}(2)$ and $C_{d-1}=B_{s_{d-1}}$ and notice that, by Lemma \ref{7l2}, 
we have $\dens(C_{d-1})\meg 3\delta/4\meg \delta^2/16$. Since $d>16/\delta^2$ there exist $0\mik i<j\mik d-1$ such that $C_i\cap C_j\neq \varnothing$. We define
\begin{equation} \label{7e12}
c'=t_i \text{ and } w'=w_{i+1}^{\con}...^{\con}w_j^{\con}y
\end{equation}
where $y=w_{j+1}(1)^{\con}...^{\con}w_{d-1}(1)$ if $j<d-1$ and $y=\varnothing$ otherwise. Let $S$ be the Carlson--Simpson line of $W$ generated
by the sequence $(c',w')$ and observe that $S=\{t_i\}\cup\{s_i,s_j\}$. Hence,
\begin{equation} \label{7e13}
\bigcap_{t\in S} B_t \supseteq C_i\cap C_j\neq \varnothing
\end{equation}
and the proof is completed.
\end{proof}
We are ready to proceed to the proof of Proposition \ref{7p1}.
\begin{proof}[Proof of Proposition \ref{7p1}]
Let $S$ be the Carlson--Simpson line obtained by Lemma \ref{7l4}. We select $x_0\in X_L$ such that $x_0\in B_t$ for every $t\in S$ and we set
\begin{equation} \label{7e14}
R=\big\{\cv_L(t,x_0):t\in S\big\}.
\end{equation}
By Lemma \ref{5l6}, we have that $R$ is a Carlson--Simpson line of $[2]^{<\nn}$. Next, recall that $B=\cv_L^{-1}(A)$. Since $(t,x_0)\in B$
for every $t\in S$, we conclude that $R$ is contained in $A$, as desired.
\end{proof}
\noindent 7.2. \textbf{Estimating the numbers $\dcs(k,m+1,\delta)$.} Let $k,m\in\nn$ with $k\meg 2$ and $m\meg 1$ and assume that for every
$0<\beta\mik 1$ the number $\dcs(k,m,\beta)$ has been defined. This assumption, of course, implies that for every $\ell\in [m]$ and every
$0<\beta\mik 1$ the number $\dcs(k,\ell,\beta)$ has been defined. Therefore, for every $\ell\in [m]$ and every $0<\delta\mik 1$ we may set
\begin{equation} \label{7e15}
\Lambda(k,\ell,\delta)=\lceil \delta^{-1}\dcs(k,\ell,\delta) \rceil
\end{equation}
and
\begin{equation} \label{7e16}
\Theta(k,\ell,\delta)=\frac{2\delta}{|\subtr_{\ell}\big( [k]^{<\Lambda(k,\ell,\delta)}\big)|}.
\end{equation}
Moreover, let
\begin{equation} \label{7e17}
\Lambda_0=\Lambda_0(k,\delta)=\Lambda(k,1,\delta^2/16) \text{ and } \Theta_0=\Theta_0(k,\delta)=\Theta(k,1,\delta^2/16)
\end{equation}
and define $h_{\delta}:\nn\to\nn$ by the rule
\begin{equation} \label{7e18}
h_{\delta}(n)= \Lambda_0+ \lceil 2\Theta_0^{-1}n\rceil.
\end{equation}
It is, of course, clear that for the definition of $\Lambda_0$, $\Theta_0$ and $h_{\delta}$ we only need to have the number $\dcs(k,1,\delta^2/16)$
at our disposal. The main result of this section is the following dichotomy.
\begin{prop} \label{7p5}
Let $k\in\nn$ with $k\meg 2$ and assume that for every $0<\beta\mik 1$ the number $\dcs(k,1,\beta)$ has been defined.

Let $0<\delta\mik 1$ and define $\Lambda_0$ and $\Theta_0$ as in \eqref{7e17}. Also let $L$ be a nonempty finite subset of $\nn$ and
$A\subseteq [k]^{<\nn}$ such that $\dens_{[k]^\ell}(A)\meg \delta$ for every $\ell\in L$. Finally, let $n\in\nn$ with $n\meg 1$ and assume
that $|L|\meg h_{\delta}(n)$ where $h_{\delta}$ is as in \eqref{7e18}. Then, setting $L_0$ to be the set of the first $\Lambda_0$
elements of $L$, we have that either
\begin{enumerate}
\item[(i)] there exist a subset $L'$ of $L\setminus L_0$ with $|L'|\meg n$ and a word $t_0\in [k]^{\ell_0}$ for some $\ell_0\in L_0$ such that
\begin{equation} \label{7e19}
\dens_{[k]^{\ell-\ell_0}}\big( \{s\in[k]^{<\nn}: t_0^{\con}s\in A\}\big) \meg \delta+\delta^2/8
\end{equation}
for every $\ell\in L'$, or
\item[(ii)] there exist a subset $L''$ of $L\setminus L_0$ with $|L''|\meg n$ and a Carlson--Simpson line $V$ of $[k]^{<\nn}$ contained in $A$ with
$L(V)\subseteq L_0$ and such that, setting $\ell_1$ to be the unique integer with $V(1)\subseteq [k]^{\ell_1}$, for every $\ell\in L''$ we have
\begin{equation} \label{7e20}
\dens_{[k]^{\ell-\ell_1}}\big( \{s\in[k]^{<\nn}: t^{\con}s\in A \text{ for every } t\in V(1)\}\big) \meg \Theta_0/2.
\end{equation}
\end{enumerate}
\end{prop}
Proposition \ref{7p5} can be used to estimate the numbers $\dcs(k,m+1,\delta)$ via a standard iteration. In particular, we have the following
corollary.
\begin{cor} \label{7c6}
Let $k,m\in\nn$ with $k\meg 2$ and $m\meg 1$ and assume that for every $0<\beta\mik 1$ the number $\dcs(k,m,\beta)$ has been defined.
Then, for every $0<\delta\mik 1$
\begin{equation} \label{7e21}
\dcs(k,m+1,\delta)\mik h_\delta^{(\lceil8\delta^{-2}\rceil)}\big( \dcs(k,m,\Theta_0/2)\big).
\end{equation}
\end{cor}
\begin{proof}
We fix $0<\delta\mik 1$. For notional convenience we set $N_0=\dcs(k,m,\Theta_0/2)$. Let $L$ be an arbitrary finite subset of $\nn$ with
$|L|\meg h_{\delta}^{(\lceil8\delta^{-2}\rceil)}(N_0)$ and $A$ be a subset of $[k]^{<\nn}$ such that $\dens_{[k]^\ell}(A)\meg\delta$ for every
$\ell\in L$. By our assumptions on the size of the set $L$ and repeated applications of Proposition \ref{7p5}, it is possible to find a subset
$L''$ of $L$ with $|L''|\meg N_0$ and a Carlson--Simpson line $V$ of $[k]^{<\nn}$ contained in $A$ such that, setting $\ell_1$ to be the unique
integer with $V(1)\subseteq [k]^{\ell_1}$, we have that $\ell_1<\min(L'')$ and
\begin{equation} \label{7e22}
\dens_{[k]^{\ell-\ell_1}}\big( \{s\in[k]^{<\nn}: t^{\con}s\in A \text{ for every } t\in V(1)\}\big) \meg \Theta_0/2
\end{equation}
for every $\ell\in L''$. By the choice of $N_0$ and \eqref{7e22}, there exists an $m$-dimensional Carlson--Simpson tree $U$ of $[k]^{<\nn}$
such that
\begin{equation} \label{7e23}
U\subseteq \{s\in[k]^{<\nn}: t^{\con}s\in A \text{ for every } t\in V(1)\}.
\end{equation}
Therefore, setting
\begin{equation} \label{7e24}
S =V(0) \cup \bigcup_{t\in V(1)} \{t^{\con}u:u\in U\},
\end{equation}
we see that $S$ is a Carlson--Simpson tree of $[k]^{<\nn}$ of dimension $m+1$ which is contained in $A$. The proof is thus completed.
\end{proof}
For the proof of Proposition \ref{7p5} we need to do some preparatory work. Specifically, let $L$ be a finite subset of $\nn$ with
$|L|\meg 2$ and consider the convolution operation $\cv_L:[k]^{<|L|}\times X_L\to [k]^{<\nn}$ associated to $L$. For every $x\in X_L$ let
$R_x=\{ \cv_{L}(t,x): t\in [k]^{<|L|}\}$ and recall that, by Lemma \ref{5l6}, the set $R_x$ is a Carlson--Simpson tree of $[k]^{<\nn}$
of dimension $|L|-1$. Therefore, we may consider the Furstenberg--Weiss measure $\mathrm{d}^{R_x}_{\mathrm{FW}}$ associated to $R_x$
defined in \S 2.7. On the other hand, we may also consider the generalized Furstenberg--Weiss measure $\mathrm{d}_L$ associated to the set $L$.
The following lemma relates these classes of measures.
\begin{lem} \label{7l7}
Let $k\in\nn$ with $k\meg 2$ and $L$ be a finite subset of $\nn$ with $|L|\meg 2$. Then for every subset $A$ of $[k]^{<\nn}$ we have
\begin{equation} \label{7e25}
\ave_{x\in X_L} \mathrm{d}^{R_x}_{\mathrm{FW}}(A)=\mathrm{d}_L(A).
\end{equation}
In particular, there exists a Carlson--Simpson tree $W$ of $[k]^{<\nn}$ of dimension $|L|-1$ with $L(W)=L$ and such that
$\mathrm{d}^{W}_{\mathrm{FW}}(A) \meg \mathrm{d}_L(A)$.
\end{lem}
Lemma \ref{7l7} follows by Lemma \ref{5l8}. More precisely, notice that equality \eqref{7e25} follows from \eqref{5e22} by averaging over all
$i\in \{0,...,|L|-1\}$.

The next fact is straightforward.
\begin{fact} \label{7f8}
Let $k,m\in\nn$ with $k\meg 2$ and $m\meg 1$. Also let $0<\eta\mik 1/2$ and assume that the number $\dcs(k,m,\eta)$ has been defined. Finally,
let $W$ be a Carlson--Simpson tree of $[k]^{<\nn}$ with $\mathrm{dim}(W)\meg \Lambda(k,m,\eta)-1$ where $\Lambda(k,m,\eta)$ is as in \eqref{7e15}.
Then every subset $B$ of $[k]^{<\nn}$ with $\mathrm{d}^W_{\mathrm{FW}}(B)\meg 2\eta$ contains a Carlson--Simpson tree of dimension $m$.
\end{fact}
The final ingredient of the proof of Proposition \ref{7p5} is the following lemma.
\begin{lem} \label{7l9}
Let $k,m\in\nn$ with $k\meg 2$ and $m\meg 1$ and assume that for every $0<\beta\mik 1$ the number $\dcs(k,m,\beta)$ has been defined.

Let $0< \rho,\gamma\mik 1$ and $L$ be a finite subset of $\nn$ with $|L|\meg \Lambda(k,m,\rho\gamma/4)$ where $\Lambda(k,m,\rho\gamma/4)$
is as in \eqref{7e15}.
Also let $B$ be a subset of $[k]^{<\nn}$ such that $\mathrm{d}_L(B)\meg \rho$. If $\{A_t:t\in B\}$ is a family of measurable events in a probability
space $(\Omega,\Sigma,\mu)$ satisfying $\mu(A_t)\meg \gamma$ for every $t\in B$, then there exists an $m$-dimensional Carlson--Simpson tree $V$
of $[k]^{<\nn}$ which is contained in $B$ and such that
\begin{equation} \label{7e26}
\mu\Big( \bigcap_{t\in V}A_t\Big) \meg \Theta(k,m,\rho\gamma/4)
\end{equation}
where $\Theta(k,m,\rho\gamma/4)$ is as in \eqref{7e16}.
\end{lem}
\begin{proof}
We set $\eta=\rho\gamma/4$ and $\Lambda=\Lambda(k,m,\eta)$. Notice that, by passing to an appropriate subset of $L$ if necessary, we may assume
that $|L|=\Lambda$. By Lemma \ref{7l7}, there exists a Carlson--Simpson tree $W$ of $[k]^{<\nn}$ of dimension $\Lambda-1$ with $L(W)=L$ and such that
$\mathrm{d}_{\mathrm{FW}}^W(B) \meg \mathrm{d}_L(B)\meg \rho$. For every $\omega\in \Omega$ let $B_\omega=\{t\in B\cap W :\omega\in A_t\}$ and set
\begin{equation} \label{7e27}
Y=\{\omega\in \Omega:\mathrm{d}_{\mathrm{FW}}^W(B_\omega)\meg \rho\gamma/2\}.
\end{equation}
Since $\mathrm{d}_{\mathrm{FW}}^W(B)\meg \rho$ and $\mu(A_t)\meg \gamma$ for every $t\in B$, we see that $\mu(Y)\meg \rho\gamma/2$.
By Fact \ref{7f8}, for every $\omega\in Y$ there exists an $m$-dimensional Carlson--Simpson tree $V_\omega$ of $[k]^{<\nn}$ such that
$V_\omega\subseteq B_\omega$. Hence, there exist $V\in\subtr_m(W)$ and $G\in\Sigma$ with $V_\omega=V$ for every $\omega\in G$ and such that
\begin{equation} \label{7e28}
\mu(G)\meg \frac{\mu(Y)}{|\subtr_m(W)|}\meg \frac{2\eta}{|\subtr_m([k]^{< \Lambda})|}\stackrel{\eqref{7e16}}{=}\Theta(k,m,\eta).
\end{equation}
It is easy to see that $V$ satisfies the conclusion of the lemma.
\end{proof}
We are now ready to proceed to the proof of Proposition \ref{7p5}.
\begin{proof}[Proof of Proposition \ref{7p5}]
Let $M=L\setminus L_0$ and set $M_\ell=\{m-\ell:m\in M\}$ for every $\ell\in L_0$. Moreover, for every $\ell\in L_0$ and every $t\in [k]^\ell$ let
\begin{equation} \label{7e29}
A_t=\{s\in [k]^{<\nn}: t^\con s\in A\}.
\end{equation}
By \eqref{7e16} and \eqref{7e17}, we see that $\Theta_0\mik \delta^2/8$. Hence, for every $\ell\in L_0$ we have
\begin{equation} \label{7e30}
|M_\ell|=|M|\meg 2\Theta_0^{-1}n\meg 8\delta^{-2}n.
\end{equation}
Also observe that for every $\ell\in L_0$ we have
\begin{equation} \label{7e31}
\ave_{t\in [k]^{\ell}} \mathrm{d}_{M_\ell}(A_t)= \mathrm{d}_M(A).
\end{equation}
Next recall that $\dens_{[k]^\ell}(A)\meg \delta$ for every $\ell\in L$. Therefore,
\begin{equation} \label{7e32}
\mathrm{d}_{L_0}(A)\meg \delta \text{ and } \mathrm{d}_M(A)\meg \delta.
\end{equation}
We consider the following cases.
\medskip

\noindent \textsc{Case 1:} \textit{there exist $\ell_0\in L_0$ and $t_0\in [k]^{\ell_0}$ such that $\mathrm{d}_{M_{\ell_0}}(A_{t_0})\meg
\delta+\delta^2/4$.} In this case we have that
\begin{equation} \label{7e33}
| \{ m\in M_{\ell_0}: \dens_{[k]^m}(A_{t_0})\meg \delta+ \delta^2/8\}| \meg (\delta^2/8) |M_{\ell_0}| \stackrel{\eqref{7e30}}{\meg} n.
\end{equation}
We set $L'=\{m\in M: \dens_{[k]^{m-\ell_0}}(A_{t_0})\meg \delta+\delta^2/8 \}$. By \eqref{7e33}, we see that with this choice the first
alternative of Proposition \ref{7p5} holds true.
\medskip

\noindent \textsc{Case 2:} \textit{for every $\ell\in L_0$ and every $t\in [k]^\ell$ we have $\mathrm{d}_{M_{\ell}}(A_t)< \delta+\delta^2/4$.}
Combining \eqref{7e31}, \eqref{7e32} and taking into account our assumptions, in this case we see that for every $\ell\in L_0$ we have
\begin{equation} \label{7e34}
|\{t\in [k]^\ell : \mathrm{d}_{M_\ell}(A_t)\meg \delta/2\}|\meg (1-\delta/2) k^\ell.
\end{equation}
We set
\begin{equation} \label{7e35}
B=\bigcup_{\ell\in L_0} \{t\in A\cap [k]^\ell: \mathrm{d}_{M_\ell}(A_t)\meg \delta/2\}.
\end{equation}
By \eqref{7e32} and \eqref{7e34}, we get that $\mathrm{d}_{L_0}(B)\meg \delta/2$. Let
\begin{equation} \label{7e36}
(\Omega,\mu)=\prod_{\ell\in L_0} \big([k]^{<\nn}, \mathrm{d}_{M_\ell}\big)
\end{equation}
be the product of the discrete probability spaces $\big([k]^{<\nn},\mathrm{d}_{M_\ell}\big)$. For every $t\in B$ we define a measurable event
$\widetilde{A}_t$ of $\Omega$ as follows. We set
\begin{equation} \label{7e37}
\widetilde{A}_t=\prod_{\ell\in L_0} X_t^\ell
\end{equation}
where $X_t^\ell=A_t$ if $\ell=|t|$ and $X_t^\ell=[k]^{<\nn}$ otherwise. Notice that for every $\ell\in L_0$ and every $t\in B\cap[k]^{\ell}$
we have
\begin{equation} \label{7e38}
\mu(\widetilde{A}_t)=\mathrm{d}_{M_\ell}(A_t) \stackrel{\eqref{7e35}}{\meg}\delta/2.
\end{equation}
Recall that $|L_0|=\Lambda_0\stackrel{\eqref{7e17}}{=}\Lambda(k,1,\delta^2/16)$. Since $\mathrm{d}_{L_0}(B)\meg \delta/2$, by Lemma \ref{7l9}
applied for ``$\rho=\gamma=\delta/2$'', there exists a Carlson--Simpson line $V$ of $[k]^{<\nn}$ which is contained in $B$ and such that
\begin{equation} \label{7e39}
\mu\Big( \bigcap_{t\in V} \widetilde{A}_t\Big)\meg \Theta(k,1,\delta^2/16)\stackrel{\eqref{7e17}}{=}\Theta_0.
\end{equation}
Notice, in particular, that the level set $L(V)$ of $V$ is contained in $L_0$. Let $\ell_1$ to be the unique integer with
$V(1)\subseteq [k]^{\ell_1}$ and observe that $\ell_1\in L_0$. By the definition of the events $\{\widetilde{A}_t:t\in B\}$ in
\eqref{7e37} and \eqref{7e39}, we get that
\begin{equation} \label{7e40}
\mathrm{d}_{M_{\ell_1}}\Big( \bigcap_{t\in V(1)} A_t\Big)= \mu\Big( \bigcap_{t\in V(1)} \widetilde{A}_t\Big) \meg\Theta_0.
\end{equation}
Let
\begin{equation} \label{7e41}
M_{\ell_1}'= \Big\{m\in M_{\ell_1}: \dens_{[k]^m}\Big( \bigcap_{t\in V(1)}A_t\Big)\meg\Theta_0/2\Big\}.
\end{equation}
By \eqref{7e40}, we have
\begin{equation} \label{7e42}
|M_{\ell_1}'| \meg (\Theta_0/2)|M_{\ell_1}|\stackrel{\eqref{7e30}}{\meg} n.
\end{equation}
We set
\begin{equation} \label{7e43}
L''=\{\ell_1+m:  m\in M_{\ell_1}'\}.
\end{equation}
It is clear that $|L''|\meg n$. Moreover, $L''$ is contained in $L\setminus L_0$ and so $\ell_1<\min(L'')$. Finally, notice that for every
$\ell\in L''$ we have that $\ell-\ell_1\in M'_{\ell_1}$. Therefore, by \eqref{7e29} and \eqref{7e41}, we conclude that
\begin{equation} \label{7e44}
\dens_{[k]^{\ell-\ell_1}}\big( \{s\in[k]^{<\nn}: t^{\con}s\in A \text{ for every } t\in V(1)\}\big)\meg\Theta_0/2.
\end{equation}
That is, the second alternative of Proposition \ref{7p5} is satisfied. The above cases are exhaustive and so the proof is completed.
\end{proof}
\noindent 7.3. \textbf{Consequences.} Let $k,m\in\nn$ with $k\meg 2$ and $m\meg 1$ and assume that for every $0<\beta\mik 1$ the number
$\mathrm{DCS}(k,m,\beta)$ has been defined. For every $0<\gamma\mik 1$ we set
\begin{equation} \label{7e45}
\theta(k,m,\gamma)=\Theta(k,m,\gamma/4)
\end{equation}
where $\Theta(k,m,\gamma/4)$ is as in \eqref{7e16}. Recall that for every Carlson--Simpson tree $W$ of $[k]^{<\nn}$ and every $k'\in\{2,...,k\}$
by $W\upharpoonright k'$ we denote the $k'$-restriction of $W$ defined in \eqref{2e14}. We have the following corollary.
\begin{cor} \label{7c10}
Let $k,m\in\nn$ with $k\meg 2$ and $m\meg 1$ and assume that for every $0<\beta\mik 1$ the number $\mathrm{DCS}(k,m,\beta)$ has been defined.

Let $0<\gamma\mik 1$ and $d\in\nn$ with $d\meg \Lambda(k,m,\gamma/4)-1$ where $\Lambda(k,m,\gamma/4)$ is as in \eqref{7e15}.
Also let $V$ be a Carlson--Simpson tree of $[k+1]^{<\nn}$ with
\begin{equation} \label{7e46}
\dim(V)\meg \mathrm{CS}(k+1,d,m,2).
\end{equation}
If $\{A_t: t\in V\}$ is a family of measurable events in a probability space $(\Omega,\Sigma,\mu)$ satisfying
$\mu(A_t)\meg \gamma$ for every $t\in V$, then there exists $W\in\mathrm{Subtr}_d(V)$ such that for every $U\in \subtr_m(W)$ we have
\begin{equation} \label{7e47}
\mu\Big( \bigcap_{v\in U\upharpoonright k} A_v\Big) \meg \theta(k,m,\gamma)
\end{equation}
where $\theta(k,m,\gamma)$ is as in \eqref{7e45}.
\end{cor}
\begin{proof}
We set $\theta=\theta(k,m,\gamma)$ and we define
\begin{equation} \label{7e48}
\mathcal{U}=\Big\{ U\in\subtr_m(V): \mu\Big( \bigcap_{v\in U\upharpoonright k}A_v\Big) \meg \theta \Big\}.
\end{equation}
By \eqref{7e46} and Theorem \ref{4t1}, it is possible to select $W\in\mathrm{Subtr}_d(V)$ such that either $\subtr_m(W)\subseteq \mathcal{U}$
or $\subtr_m(W)\cap \mathcal{U}=\varnothing$. Therefore, it is enough to show that $\subtr_m(W)\cap \mathcal{U}\neq\varnothing$.

To this end we argue as follows. Let $\mathrm{I}_W:[k+1]^{<d+1}\to W$ be the canonical isomorphism associated to $W$ and for every $t\in [k]^{<d+1}$
set $A'_t=A_{\mathrm{I}_W(t)}$. By Lemma \ref{7l9}, there exists an $m$-dimensional Carlson--Simpson subtree $R$ of $[k]^{<d+1}$ such that
\begin{equation} \label{7e49}
\mu\Big( \bigcap_{t\in V} A'_t\Big) \meg \theta.
\end{equation}
Let $S$ be the unique element of $\subtr_m(W)$ such that $S\upharpoonright k=\mathrm{I}_W(R)$. Then, by \eqref{7e49},
we conclude that $S\in\mathcal{U}$ and the proof is completed.
\end{proof}
The final result of this section is the following corollary.
\begin{cor} \label{7c11}
Let $k,m\in\nn$ with $k\meg 2$ and $m\meg 1$ and assume that for every $0<\beta\mik 1$ the number $\dcs(k,m,\beta)$ has been defined.

Let $0<\gamma\mik 1$ and $d\in\nn$ with $d\meg \Lambda(k,m,\gamma/4)-1$ where $\Lambda(k,m,\gamma/4)$ is as in \eqref{7e15}.
Also let $V$ be a Carlson--Simpson tree of $[k+1]^{<\nn}$ with
\begin{equation} \label{7e50}
\dim(V)\meg \mathrm{CS}(k+1,d,m,2).
\end{equation}
Finally let $\{A_t: t\in V\}$ be a family of measurable events in a probability space $(\Omega,\Sigma,\mu)$ satisfying $\mu(A_t)\meg \gamma$
for every $t\in V$.  Assume, in addition, that there exists $r\in [k]$ such that $A_t=A_{t'}$ for every $t,t'\in V$ which are $(r,k+1)$-equivalent
(see \S 2.6). Then there exists $W\in\mathrm{Subtr}_d(V)$ such that for every $U\in \subtr_m(W)$ we have
\begin{equation} \label{7e51}
\mu\Big( \bigcap_{v\in U} A_v\Big) \meg \theta(k,m,\gamma)
\end{equation}
where $\theta(k,m,\gamma)$ is as in \eqref{7e45}.
\end{cor}
\begin{proof}
Since $A_t=A_{t'}$ for every $t,t'\in V$ which are $(r,k+1)$-equivalent, for every $\ell\in [\dim(V)]$ and every $R\in\subtr_\ell(V)$ we have
\begin{equation} \label{7e52}
\bigcap_{t\in R} A_t= \bigcap_{t\in R\upharpoonright k} A_t.
\end{equation}
Using this observation, the result follows by Corollary \ref{7c10}.
\end{proof}


\section{A probabilistic version of Theorem B}

\numberwithin{equation}{section}

\noindent 8.1. \textbf{Overview.} In this subsection we give an outline of the proof of Theorem B. Very briefly, and oversimplifying
dramatically, the proof proceeds by induction on $k$ and is based on a density increment strategy.

The first step is given in Corollary \ref{7c6}. Indeed, by Corollary \ref{7c6}, the proof of Theorem B reduces to the task of estimating
the numbers $\dcs(k,1,\delta)$. To achieve this goal we follow an inductive scheme that can be described as follows:
\begin{equation} \label{8e1}
\dcs(k,m,\beta) \text{ for every } m \text{ and } \beta \ \Rightarrow \ \dcs(k+1,1,\delta).
\end{equation}
Precisely, in order to estimate the numbers $\dcs(k+1,1,\delta)$ we need to have at our disposal the numbers $\dcs(k,m,\beta)$ for every
integer $m\meg 1$ and every $0<\beta\mik 1$. The base case -- that is, the estimation of the numbers $\dcs(2,1,\delta)$ -- is, of course,
the content of Proposition \ref{7p1}.

At this point it is useful to recall the philosophy of the density increment method. One starts with a subset $A$ of a ``structured'' set
$\mathcal{S}$ of density $\delta$ and assumes that $A$ does not contain a subset of a certain kind. The goal is then to find a sufficiently
large ``substructure'' $\mathcal{S}'$ of $\mathcal{S}$ such that the density of $A$ inside $\mathcal{S}'$ is at least $\delta+\gamma$, where
$\gamma$ is a positive constant that depends only on $\delta$. Usually this task is rather difficult to achieve at once, and so, one first
tries to increase the density of $A$ inside a relatively ``simple'' subset of $\mathcal{S}$. We refer to the essay \cite{Gowers} of W. T.
Gowers for a thorough exposition of this method. 

The proof of the inductive scheme described in \eqref{8e1} follows the strategy just mentioned above. Specifically, fix the parameters $k$
and $\delta$ and let $A$ be a subset of $[k+1]^{<\nn}$ not containing a Carlson--Simpson line such that $\dens_{[k+1]^n}(A)\meg \delta$ for
sufficiently many $n\in\nn$. What we find is a Carlson--Simpson tree $W$ of $[k+1]^{<\nn}$ such that the density of the set $A$ has been
significantly increased in sufficiently many levels of $W$. This is done in two steps. Firstly we show that there exists a Carlson--Simpson
tree $V$ of $[k+1]^{<\nn}$ and a subset $D$ of $V$ which is the intersection of relatively few insensitive sets and correlates with the set
$A$ more than expected in many levels of $V$ (it is useful to view $D$ as a ``simple'' subset of $V$). This is the content of Corollary
\ref{8c6} below. In the second step we use this information to achieve the density increment. We will not comment at this point on the second
step, since we will do so in \S 9.1; here we simply mention that the statement of main interest is Corollary \ref{9c15}.

We will, however, discuss in detail the proof of the first step which is analogous to the first part of the polymath proof of the density
Hales--Jewett Theorem. In fact, this is more than an analogy since we are using a beautiful argument from the polymath proof
(see \cite[\S 7.2]{Pol}) to reduce the proof of Corollary \ref{8c6} -- the main result of this step -- to a ``probabilistic'' version
of Theorem B. The analogy, however, with the polymath proof breaks down at this point and the main bulk of the argument is quite different.

The aforementioned ``probabilistic'' version of Theorem B refers to the question whether a dense subset of $[k+1]^{<\nn}$ not only will contain
a Carlson--Simpson line but, actually, a non-trivial portion of them. Results of this type figure prominently in Ramsey Theory and have found
significant applications (see, e.g., \cite{CG,Sch} and the references therein). A classical result in this direction is the ``probabilistic''
version of Szemer\'{e}di's Theorem, essentially due to P. Varnavides \cite{Va}, asserting that for every integer $k\meg 3$ and every
$0<\delta\mik 1$ there exists a constant $c(k,\delta)>0$ such that every subset $A$ of $[n]$ with $|A|\meg \delta n$ contains at least
$c(k,\delta)n^2$ arithmetic progressions of length $k$, as long as $n$ is sufficiently large.

However, some density results do not admit a ``probabilistic'' version of the form stated above. The most well-known example is the density
Hales--Jewett Theorem. Indeed, if $n$ is large enough, then one can find a highly dense subset of, say, $[2]^n$ containing just a tiny portion
of combinatorial lines; see \cite[\S 3.1]{Pol}. This example can modified, in a straightforward way, to show that Theorem B also fails to admit 
a naive ``probabilistic'' version. 

This phenomenon appears to be quite discouraging, but it can be bypassed. So far there has been only one method in the literature dealing
with this problem. It was introduced by the participants of the polymath project and was based on the technique of changing the measure.
Part of the novelty of the present paper is the development of a new method which not only is conceptually easy to grasp but also appears
to be quite robust (the clearest sign for this is that it can be combined with the arguments in \cite{Pol} to give a very simple proof \cite{DKT3}
of the density Hales--Jewett Theorem). The idea is to avoid the pathological behavior by passing to an appropriate ``substructure''. This is done
applying the following three basic steps. We will describe them in abstract setting since we feel that no clarity will be gained by restricting
our discussion to the specifics of Theorem B.
\medskip

\noindent \textit{Step 1.} By an application of Szemer\'{e}di's regularity method \cite{Sz2}, we show that a given dense set $A$ of our 
``structured'' set $\mathcal{S}$ is sufficiently pseudorandom. This enables us to model the set $A$ as a family of measurable events
$\{B_t:t\in\mathcal{R}\}$ in a probability space $(\Omega,\Sigma,\mu)$ indexed by a Ramsey space $\mathcal{R}$ closely related, of course,
with $\mathcal{S}$. The measure of the events is controlled by the density of $A$.
\medskip

\noindent \textit{Step 2.} We apply coloring arguments and our basic density result to show that there exists a ``substructure'' $\mathcal{R}'$
of $\mathcal{R}$ such that the events in the subfamily $\{B_t:t\in \mathcal{R}'\}$ are highly correlated. The reasoning can be traced in an old
paper of P. Erd\H{o}s and A. Hajnal \cite{EH}. Also we notice that it is precisely in this step that we need to pass to a ``substructure''. 
As can be seen from the examples mentioned above, this is a necessity rather than a coincidence.
\medskip

\noindent \textit{Step 3.} We use a double counting argument to locate a ``substructure'' $\mathcal{S}'$ of $\mathcal{S}$ such that the
set $A$ contains a non-trivial portion of subsets of $\mathcal{S}'$ of the desired kind (combinatorial lines, Carlson--Simpson lines, etc.).
\medskip

Some final comments on the computational effectiveness of the method. In all cases of interest known to the authors, the tools used in the
three steps described above have primitive recursive (and fairly reasonable) bounds. However, the argument yields very poor lower bounds for
the correlation of the events $\{B_t:t\in\mathcal{R}'\}$ in the second step. These lower bounds are partly responsible for the Ackermannian
behavior of the numbers $\dcs(k,m,\delta)$.
\medskip

\noindent 8.2. \textbf{The main dichotomy.}  Let $k,m\in\nn$ with $k\meg 2$ and $m\meg 1$ and assume that for every $0<\beta\mik 1$
the number $\dcs(k,m,\beta)$ has been defined. Hence, for every $0<\delta\mik 1$ we may set
\begin{equation} \label{8e2}
\vartheta=\vartheta(k,m,\delta)=\Theta(k,m,\delta/8) \text{ and } \eta=\eta(k,m,\delta)=\frac{\delta\vartheta}{30k}
\end{equation}
where $\Theta(k,m,\delta/8)$ is as in \eqref{7e16}. Moreover let
\begin{equation} \label{8e3}
\Lambda'=\Lambda(k,m,\delta/8)\stackrel{\eqref{7e15}}{=}\lceil8\delta^{-1}\mathrm{DCS}(k,m,\delta/8)\rceil
\end{equation}
and for every $n\in\nn$ set
\begin{equation} \label{8e4}
\ell(n,m)=\mathrm{CS}(k+1, n+\Lambda', m,2)+1.
\end{equation}
We define the map $G:\nn\times\nn\times (0,1]\to \nn$ by the rule
\begin{equation} \label{8e5}
G(n,m,\ee)=\mathrm{Reg}(k+1,\ell(n,m) ,1,\ee).
\end{equation}
Also for every Carlson--Simpson tree $V$ of $[k]^{<\nn}$ and every $1\mik m\mik i\mik \dim(V)$ let
\begin{equation} \label{8e6}
\subtr^0_m(V,i)=\big\{ R\in\subtr_m(V): R(0)=V(0) \text{ and } R(m)\subseteq V(i)\big\}.
\end{equation}
As we have pointed out in \S 2.5, if $W$ is a Carlson--Simpson tree of $[k+1]^{<\nn}$, then its $k$-restriction $W\upharpoonright k$ can be
identified as a Carlson--Simpson tree of $[k]^{<\nn}$. Hence, we may also consider the set $\subtr^0_m(W\upharpoonright k,i)$ whenever
$W$ is a Carlson--Simpson tree of $[k+1]^{<\nn}$.

We are ready to state the main result of this section.
\begin{prop} \label{8p1}
Let $k,m\in\nn$ with $k\meg 2$ and $m\meg 1$ and assume that for every $0<\beta\mik 1$ the number $\mathrm{DCS}(k,m,\beta)$ has been defined.

Let $0<\delta\mik 1$ and define $\vartheta$, $\eta$ and $\Lambda'$ as in \eqref{8e2} and \eqref{8e3} respectively. Also let $n\in\nn$ with
$n\meg 1$ and $N$ be a finite subset of $\nn$ such that
\begin{equation}\label{8e7}
|N|\meg G\big( \lceil \eta^{-4}n\rceil, m, \eta^2/2\big)
\end{equation}
where $G$ is as in \eqref{8e5}. If $A\subseteq [k+1]^{<\nn}$ satisfies $|A\cap [k+1]^l|\meg \delta (k+1)^l$ for every $l\in N$, then there exist
a Carlson--Simpson tree $W$ of $[k+1]^{<\nn}$ with $\dim(W)=\lceil \eta^{-4}n\rceil+\Lambda'$ and $I\subseteq \{m,...,\dim(W)\}$ with $|I|\meg n$
such that either
\begin{enumerate}
\item[(i)] for every $i\in I$ we have $\dens_{W(i)}(A)\meg \delta+\eta^2/2$, or
\item[(ii)] for every $i\in I$ we have $\dens_{W(i)}(A)\meg \delta-2\eta$ and moreover
\begin{equation} \label{8e8}
\dens \big( \{V\in\subtr_m^0(W\upharpoonright k,i): V\subseteq A\} \big) \meg \vartheta/2.
\end{equation}
\end{enumerate}
\end{prop}
For the proof of Proposition \ref{8p1} we need to do some preparatory work. We start with the following lemma.

\begin{lem} \label{8l2}
Let $k,m,\delta, \vartheta, \eta, \Lambda', n$ and $N$ be as in Proposition \ref{8p1}. If $A$ is a subset of $[k+1]^{<\nn}$ satisfying 
$|A\cap [k+1]^l|\meg \delta (k+1)^l$ for every $l\in N$, then there exist a subset $L$ of $N$ with $|L|=\mathrm{CS}(k+1,\lceil
\eta^{-4}n \rceil+\Lambda',m,2)+1$ and a Carlson--Simpson subtree $W$ of $[k+1]^{<|L|}$ of dimension $\lceil\eta^{-4}n\rceil+ \Lambda'$
such that, setting $\cv_L:[k+1]^{<|L|}\times X_L\to [k+1]^{<\nn}$ to be the convolution operation associated to $L$ and $B=\cv_L^{-1}(A)$,
the following properties hold.
\begin{enumerate}
\item[(i)] For every $t\in [k+1]^{<|L|}$ we have $\dens(B_t)\meg \delta-\eta^2/2$ where $B_t$ is the section of $B$ at $t$.
\item[(ii)] For every $U\in \subtr_m(W\upharpoonright k)$ we have 
\begin{equation} \label{8e9}
\dens\Big( \bigcap_{t\in U} B_t\Big) \meg \vartheta.
\end{equation}
\end{enumerate}
\end{lem}
\begin{proof}
By \eqref{8e7} and the definition of the function $G$ in \eqref{8e5}, we may apply Lemma \ref{3l2} and we get $L\subseteq N$ with
$|L|=\mathrm{CS}(k+1,\lceil \eta^{-4}n\rceil+\Lambda',m,2)+1$ and such that the family $\mathcal{F}:=\{A\}$ is $(\eta^2/2,L)$-regular.
Using this information and arguing as in the proof of Lemma \ref{7l2} we see that the first part of the lemma is satisfied. 

For part (ii), set $V=[k+1]^{<|L|}$. Also let $\gamma=\delta/2$ and $d=\lceil \eta^{-4}n\rceil+\Lambda'$. Notice that 
$d\meg \Lambda'-1=\Lambda(k,m,\gamma/4)-1$. Moreover, by part (i), we have 
\begin{equation} \label{8e10}
\dim(V)=|L|-1=\mathrm{CS}(k+1,d,m,2)
\end{equation}
and $\dens(B_t)\meg \delta-\eta^2/2\meg \gamma$ for every $t\in [k+1]^{<|L|}$. Hence, by Corollary \ref{7c10}, we get a $d$-dimensional
Carlson--Simpson subtree $W$ of $[k+1]^{<|L|}$ such that 
\begin{equation} \label{8e11}
\dens \Big( \bigcap_{t\in U\upharpoonright k}B_t\Big)\meg \theta(k,m,\gamma)
\end{equation}
for every $U\in \subtr_m(W)$. Since $\subtr_m(W\upharpoonright k)\subseteq\{U\upharpoonright k: U\in\subtr_m(W)\}$ and
$\theta(k,m,\gamma)=\Theta(k,m,\gamma/4)=\vartheta$ the result follows.
\end{proof}
To state the next result towards the proof Proposition \ref{8p1} we need to recall some notation introduced in \S 5. Specifically,
let $L$ be a nonempty finite subset of $\nn$ and consider the convolution operation $\cv_L: [k+1]^{<|L|}\times X_L\to [k+1]^{<\nn}$ 
associated to $L$. As in \eqref{5e14}, for every Carlson--Simpson subtree $W$ of $[k+1]^{<|L|}$ and every $x\in X_L$ we set
\begin{equation} \label{8e12}
W_x=\big\{ \cv_L(w,x): w\in W\big\}.
\end{equation}
Recall that, by Lemma \ref{5l6}, $W_x$ is a Carlson--Simpson tree of $[k+1]^{<\nn}$ with $\dim(W_x)=\dim(W)$. We have the following lemma.
\begin{lem} \label{8l3}
Let $k\in\nn$ with $k\meg 2$, $L$ be a nonempty finite subset of $\nn$ and consider the convolution operation 
$\cv_L: [k+1]^{<|L|}\times X_L \to [k+1]^{<\nn}$ associated to the set $L$. Also let $W$ be a Carlson--Simpson subtree of $[k+1]^{<|L|}$
and $A\subseteq [k+1]^{<\nn}$. We set $B=\cv_{L}^{-1}(A)$ and for every  $t\in [k+1]^{<|L|}$ let $B_t$ be the section of $B$ at $t$. 
Moreover, for every $x\in X_L$ let $W_x$ be as
in \eqref{8e12}. Then the following hold.
\begin{enumerate}
\item[(i)] For every $i\in\{0,...,\dim(W)\}$ we have
\begin{equation} \label{8e13}
\ave_{x\in X_L}\dens_{W_x(i)}(A)=\ave_{t\in W(i)}\dens(B_t).
\end{equation}
\item[(ii)] For every $1\mik m\mik i\mik\dim (W)$ we have
\[\ave_{x\in X_L}\dens\big(\{V\!\in\!\subtr_m^0(W_x\upharpoonright k,i): V\!\subseteq\! A\}\big)  = \ave_{U\in\subtr_m^0 (W\upharpoonright k,i)}
\dens\Big(\bigcap_{t\in U} B_t\Big).\]
\end{enumerate}
\end{lem}
\begin{proof}
(i) Fix $i\in \{0,...,\dim(W)\}$. By Lemma \ref{5l8}, we have
\begin{equation} \label{8e14}
\ave_{x\in X_L}\dens_{W_x(i)}(A)=\dens_{\cv_L(W(i)\times X_L)}(A).
\end{equation}
Also notice that
\begin{equation} \label{8e15}
\cv_L(W(i)\times X_L)= \bigcup_{t\in W(i)}\cv_L(\{t\}\times X_L)\stackrel{\eqref{5e6}}{=} \bigcup_{t\in W(i)}\Omega_t.
\end{equation}
Next observe that $W(i)\subseteq [k+1]^l$ for some $l\in\{0,...,|L|-1\}$. Therefore, by Fact \ref{5f2}, we see that
$|\Omega_t|=|\Omega_{t'}|$ for every $t, t'\in W(i)$. Hence,
\begin{equation} \label{8e16}
\dens_{\cv_L(W(i)\times X_L)}(A)=\ave_{t\in W(i)}\dens_{\Omega_t}(A).
\end{equation}
Finally, by Lemma \ref{5l5}, for every $t\in W(i)$ we have
\begin{equation} \label{8e17}
\dens_{\Omega_t}(A)=\dens_{\{t\}\times X_L}(B)=\dens(B_t).
\end{equation}
Combining (\ref{8e14}), (\ref{8e16}) and (\ref{8e17}) we conclude that \eqref{8e13} is satisfied.
\medskip

\noindent (ii) We fix $1\mik m\mik i\mik \dim(W)$ and we set
\begin{equation}\label{8e18}
\mathbb{A}=\big\{ (U, x)\in \subtr_m^0(W\upharpoonright k,i)\times X_L: U_x\subseteq A \big\}
\end{equation}
where, as in \eqref{8e12}, $U_x=\{\cv_L(u,x):u\in U\}$. Also for every $U\in \subtr_m^0(W\upharpoonright k,i)$ and every $x\in X_L$ let
\begin{equation} \label{8e19}
\mathbb{A}_U = \{ x\in X_L:(U,x)\in \mathbb{A}\} 
\end{equation}
and
\begin{equation} \label{8e20}
\mathbb{A}_x= \{U\in \subtr_m^0(W\upharpoonright k,i): (U,x)\in \mathbb{A}\}
\end{equation}
be the sections of $\mathbb{A}$ at $U$ and $x$ respectively. Notice that
\begin{equation} \label{8e21}
x\in \mathbb{A}_U\Leftrightarrow(U,x)\in \mathbb{A}\Leftrightarrow U_x\subseteq A\Leftrightarrow x\in \bigcap_{t\in U} B_t.
\end{equation}
This, of course, implies that for every $U\in\subtr_m^0(W\upharpoonright k,i)$ we have
\begin{equation} \label{8e22}
\mathbb{A}_U= \bigcap_{t\in U} B_t.
\end{equation}
Moreover, it is easy to see that for every $x\in X_L$ the map
\begin{equation} \label{8e23}
\subtr_m^0(W\upharpoonright k,i)\ni U \mapsto U_x\in\subtr_m^0(W_x\upharpoonright k,i)
\end{equation}
is a bijection. Therefore, for every $x\in X_L$ we have
\begin{eqnarray*}
\dens\big( \{V\in \subtr_m^0(W_x\upharpoonright k,i): V\subseteq A\}\big) & = & 
\frac{|\{V\in \subtr_m^0(W_x\upharpoonright k,i): V\subseteq A\}|}{|\subtr_m^0(W_x\upharpoonright k,i)|} \\
& = & \frac{|\{U\in \subtr_m^0(W\upharpoonright k,i): U_x\subseteq A\}|}{|\subtr_m^0(W\upharpoonright k,i)|} \\ 
& = & \dens(\mathbb{A}_x).
\end{eqnarray*}
Taking into account \eqref{8e22} and the above equalities we conclude that
\begin{eqnarray*}
\ave_{x\in X_L}\dens(\{V\!\in\!\subtr_m^0(W_x\upharpoonright k,i): V\!\subseteq\! A\}) & = & \ave_{x\in X_L} \dens(\mathbb{A}_x) \\
& = & \ave_{U\in\subtr_m^0(W\upharpoonright k,i)}\dens(\mathbb{A}_U) \\
& = & \ave_{U\in\subtr_m^0(W\upharpoonright k,i)}\dens\Big(\bigcap_{t\in U}B_t\Big)
\end{eqnarray*}
and the proof is completed.
\end{proof}
We are ready to proceed to the proof of Proposition \ref{8p1}.
\begin{proof}[Proof of Proposition \ref{8p1}]
We fix $A\subseteq [k+1]^{<\nn}$ such that $|A\cap [k+1]^l|\meg \delta (k+1)^l$ for every $l\in N$. For notational convenience, we set
$d=\lceil \eta^{-4}n\rceil+\Lambda'$. Let $L$ and $W$ be as in Lemma \ref{8l2} when applied to the fixed set $A$. Notice that $\dim(W)=d$.
Invoking the first parts of Lemmas \ref{8l2} and \ref{8l3}, for every $i\in \{0,...,d\}$ we have
\begin{equation}\label{8e24}
\ave_{x\in X_L}\dens_{W_x(i)}(A)\meg \delta-\eta^2/2.
\end{equation}
On the other hand, by the second parts of the aforementioned lemmas, we see that
\begin{equation}\label{8e25}
\ave_{x\in X_L}\dens\big(\{V\in \subtr_m^0(W_x\upharpoonright k,i): V\subseteq A\}\big)\meg \vartheta
\end{equation}
for every $i\in \{m,...,d\}$.

Let $J=\{m,..., d\}$ and notice that 
\begin{equation}\label{8e26}
|J|=d-m+1\meg \lceil \eta^{-4}n\rceil.
\end{equation}
Also for every $i\in J$ set
\begin{equation}\label{8e27}
X^0_i=\{x\in X_L : \dens_{W_x(i)}(A)\meg \delta+\eta^2/2\},
\end{equation}
\begin{equation}\label{8e28}
X^1_i=\{x\in X_L : \dens_{W_x(i)}(A)\meg \delta-2\eta\}
\end{equation}
and
\begin{equation}\label{8e29}
X^2_i=\Big\{x\in X_L : \dens\big( \{V\in \subtr_m^0(W_x\upharpoonright k,i): V\subseteq A\}\big) \meg \vartheta/2 \Big\}.
\end{equation}
Finally let $J_0=\{i\in J:\dens(X_i^0)\meg \eta^3\}$. We distinguish the following cases.
\medskip

\noindent \textsc{Case 1:} \textit{we have $|J_0|\meg |J|/2$}. By Lemma \ref{2l5}, there exists $x_0\in X_L$ such that
\begin{equation}\label{8e30}
|\{i\in J_0: x_0\in X_i^0\}| \meg \frac{\eta^3 |J_0|}{2} \meg \frac{\eta^3|J|}{4} \stackrel{\eqref{8e26}}{\meg} 
\frac{\eta^3\lceil \eta^{-4}n\rceil}{4} \meg n.
\end{equation}
We set ``$I=\{i\in J: x_0\in X_i^0\}$'' and ``$W=W_{x_0}$''. With these choices it is easy to see that the first part of the proposition
is satisfied. 
\medskip

\noindent \textsc{Case 2:} \textit{we have $|J_0|<|J|/2$}. In this case we set $\bar{J_0}=J\setminus J_0$. Let $i\in \bar{J_0}$ be arbitrary
and notice that 
\begin{equation} \label{8e31}
\dens(X_i^0)=\dens\big( \{x\in X_L: \dens_{W_x(i)}(A)\meg \delta+\eta^2/2\}\big )< \eta^3.
\end{equation}
Combining \eqref{8e24} and \eqref{8e31} we see that $\dens(X^1_i)\meg 1-\eta$. On the other hand, by \eqref{8e25}, we have 
$\dens(X_i^2)\meg\vartheta/2$. Therefore, by the choice of $\eta$ in \eqref{8e2}, we conclude that $\dens(X^1_i\cap X^2_i)\meg \vartheta/4$
for every $i\in \bar{J_0}$. By a second application of Lemma \ref{2l5}, we find $x_1\in X_L$ such that
\begin{equation} \label{8e32}
|\{i\in \bar{J_0}: x_1\in X_i^1\cap X_i^2\}|\meg \frac{\vartheta |\bar{J_0}|}{8}\meg \frac{\vartheta |J|}{16} \stackrel{\eqref{8e26}}{\meg}
\frac{\vartheta\lceil \eta^{-4}n\rceil}{16} \stackrel{\eqref{8e2}}{\meg} n.
\end{equation}
We set ``$I=\{i\in \bar{J_0}: x_1\in X_i^1\cap X_i^2\}$'' and ``$W=W_{x_1}$'' and we notice that with these choices the second part
of the proposition is satisfied. The above cases are exhaustive and the proof is completed.
\end{proof}
\noindent 8.3. \textbf{Obtaining insensitive sets.} Let $k\in\nn$ with $k\meg 2$ and assume that for every $0<\beta\mik 1$ the number 
$\dcs(k,1,\beta)$ has been defined. For every $0<\delta\mik 1$ let 
\begin{equation} \label{8e33}
\vartheta_1=\vartheta_1(k,\delta)=\vartheta(k,1,\delta) \text{ and } \eta_1=\eta_1(k,\delta)=\eta(k,1,\delta)
\end{equation}
where $\vartheta(k,1,\delta)$ and $\eta(k,1,\delta)$ are as in \eqref{8e2}. Also define $G_1:\nn\times (0,1]\to\nn$ by
\begin{equation} \label{8e34}
G_1(n,\ee)=G(n,1,\ee)
\end{equation}
where $G$ is as in \eqref{8e5}. We have the following lemma.
\begin{lem} \label{8l4}
Let $k\in\nn$ with $k\meg 2$ and assume that for every $0<\beta\mik 1$ the number $\dcs(k,1,\beta)$ has been defined.

Let $0<\delta\mik 1$ and define $\vartheta_1$ and $\eta_1$ as in \eqref{8e33}. Also let $n\in\nn$ with $n\meg 1$ and $N$ be a finite subset of
$\nn$ such that
\begin{equation} \label{8e35}
|N|\meg G_1\big( \lceil \eta_1^{-4}(k+1)n \rceil, \eta_1^2/2 \big)
\end{equation}
where $G_1$ is as in \eqref{8e34}. Finally let $A\subseteq [k+1]^{<\nn}$ such that $|A\cap [k+1]^l|\meg \delta (k+1)^l$ for every $l\in N$
and assume that $A$ contains no Carlson--Simpson line of $[k+1]^{<\nn}$. Assume, moreover, that for every Carlson--Simpson tree $W$ of
$[k+1]^{<\nn}$ we have
\begin{equation} \label{8e36}
|\big\{i\in\{0,...,\dim(W)\}: \dens_{W(i)} (A)\meg \delta + \eta_1^2/2\big\}|< n.
\end{equation}
Then there exist a Carlson--Simpson tree $V$ of $[k+1]^{<\nn}$, a subset $C$ of $V$ and a subset $J$ of $\{0,...,\dim(V)\}$ 
with the following properties.
\begin{enumerate}
\item[(i)] We have $|J|\meg n$.
\item[(ii)] We have $C=\bigcap_{r=1}^k C_r$ where the set $C_r$ is $(r,k+1)$-insensitive in $V$ for every $r\in [k]$. 
Moreover, $\dens_{V(j)}(C)\meg\vartheta_1/2$ for every $j\in J$.
\item[(iii)] The sets $A$ and $C$ are disjoint.
\item[(iv)] For every $j\in J$ we have $\dens_{V(j)}(A) \meg \delta-5k\eta_1$.
\end{enumerate}
\end{lem}
Lemma \ref{8l4} will be reduced to Proposition \ref{8p1}. The reduction will be achieved using essentially the same arguments as in
\cite[\S 7.2]{Pol}. 

First we need to introduce some pieces of notation. Specifically, let $W$ be a Carlson--Simpson tree of $[k+1]^{<\nn}$ and set $d=\dim(W)$.
Consider the canonical isomorphism $\mathrm{I}_W: [k+1]^{<d+1}\to W$ associated to $W$ (see \S 2.5) and for every $r\in [k+1]$ let
\begin{equation} \label{8e37}
W[r]=\big\{ \mathrm{I}_W(s): s\in [k+1]^{<d+1} \text{ with } |s|\meg 1 \text{ and } s(0)=r\big\}.
\end{equation}
Notice that if $\dim(W)\meg 2$, then $W[r]$ is a Carlson--Simpson subtree of $W$ with $\dim(W[r])=\dim(W)-1$. On the other hand, if $W$ is a
Carlson--Simpson line, then $W[r]$ is the singleton $\{\mathrm{I}_W(r)\}$; we will identify in this case $W[r]$ with $\mathrm{I}_W(r)$. Next
observe that for every Carlson--Simpson subtree $U$ of the $k$-restriction $W\upharpoonright k$ of $W$ there exists a unique Carlson--Simpson
tree $R$ of $[k+1]^{<\nn}$ such that $R\upharpoonright k=U$. We will call this unique Carlson--Simpson tree $R$ as the \textit{extension} of
$U$ and we will denote it by $\bar{U}$. Notice that the extension of $U$ is a Carlson--Simpson subtree of $W$. Also we will need the following
elementary fact.
\begin{fact} \label{8f5}
Let $W$ be a Carlson--Simpson tree of $[k+1]^{<\nn}$ with $\dim(W)\meg 2$ and set $V=W[k+1]$. If $j\in \{0,...,\dim(V)\}$ and
$U\in\subtr^0_1(W\upharpoonright k,j+1)$, then $\bar{U}[k+1]\in V(j)$. Moreover, the map
\begin{equation} \label{8e38}
\subtr^0_1(W\upharpoonright k,j+1)\ni U \mapsto \bar{U}[k+1]\in V(j)
\end{equation}
is a bijection.
\end{fact}
We are ready to give the proof of Lemma \ref{8l4}.
\begin{proof}[Proof of Lemma \ref{8l4}]
By our assumptions we may apply Proposition \ref{8p1} for $m=1$ and we get a Carlson--Simpson tree $W$ of $[k+1]^{<\nn}$ and
$I\subseteq \{1,...,\dim (W)\}$ with $|I|\meg (k+1)n$ and such that $\dens_{W(i)}(A)\meg \delta-2\eta_1$ and
\begin{equation}\label{8e39}
\dens\big( \{U\in\subtr_1^0(W\upharpoonright k,i): U\subseteq A\} \big)\meg \vartheta_1/2
\end{equation}
for every $i\in I$. For every $r\in [k+1]$ let $V_r=W[r]$. We set
\begin{equation} \label{8e40}
V=V_{k+1},
\end{equation}
\begin{equation} \label{8e41}
C= \bigcup_{i\in I}\big\{ \bar{U}[k+1]: U\in\subtr_1^0(W\upharpoonright k,i) \text{ with } U\subseteq A\big\}
\end{equation}
and 
\begin{equation} \label{8e42}
J=\big\{ j\in \{0,...,\dim(V)\}: \dens_{V(j)}(A)\meg \delta-5k\eta_1 \text{ and } j+1\in I \big\}.
\end{equation}
We claim that $V$, $C$ and $J$ are as desired. First we argue to show that $|J|\meg n$. Let $J_0=\{j\in \{0,...,\dim(V)\}: j+1\in I\}$.
Observe that $J\subseteq J_0$ and 
\begin{equation} \label{8e43}
|J_0|\meg (k+1)n.
\end{equation}
Let $j\in J_0\setminus J$ be arbitrary and notice that $\dens_{V(j)}(A)< \delta-5k\eta_1$. On the other hand, $j+1\in I$ and so
\begin{equation} \label{8e44}
\ave_{r\in[k+1]}\dens_{V_r(j)}(A)=\dens_{W(j+1)}(A)\meg \delta-2\eta_1.
\end{equation}
Thus, there exists $r_j\in [k]$ such that $\dens_{V_{r_j}(j)}(A)\meg \delta+\eta_1$. Therefore, by the classical
pigeonhole principle, there exists $r_0\in [k]$ such that
\begin{equation} \label{8e45}
|\{j\in J_0\setminus J: \dens_{V_{r_0}(j)}(A)\meg \delta+\eta_1\}|\meg \frac{|J_0\setminus J|}{k}\stackrel{\eqref{8e43}}{\meg} n+\frac{n-|J|}{k}.
\end{equation}
Moreover, by \eqref{8e36}, we have
\begin{equation} \label{8e46}
|\{j\in J_0\setminus J: \dens_{V_{r_0}(j)}(A)\meg \delta+\eta_1\}|<n.
\end{equation}
Combining \eqref{8e45} and \eqref{8e46} we conclude that $|J|\meg n$.

We continue with  the proof of part (ii). Let $r\in [k]$ be arbitrary. For every $l\in\nn$ and every $s\in [k+1]^{l}$ let $s^{{k+1}\to r}$
be the unique element of $[k]^l$ obtained by replacing all appearances of $k+1$ in $s$ by $r$. We set
\begin{equation} \label{8e47}
C_r=\Big\{ \mathrm{I}_V(s): s\in \bigcup_{i\in I} [k+1]^{i-1} \text{ and } s^{{k+1}\to r}\in \mathrm{I}_{V_r}^{-1}(A) \Big\}
\end{equation} 
where $\mathrm{I}_V$ and $\mathrm{I}_{V_r}$ are the canonical isomorphisms associated to $V$ and $V_r$ respectively. It is easy to check
that $C_r$ is $(r,k+1)$-insensitive in $V$. Next we argue to show that $C$ coincides with $C_1\cap ... \cap C_k$. First we notice that
$C\subseteq C_1\cap ... \cap C_k$. To see the other inclusion, let $t\in C_1\cap...\cap C_k$ be arbitrary and set $s=\mathrm{I}_V^{-1}(t)$.
Let $i$ be the unique element of $I$ such that $s\in [k+1]^{i-1}$ and set
\begin{equation} \label{8e48}
U_t= \{W(0)\}\cup \big\{ \mathrm{I}_{V_r}(s^{{k+1}\to r}) : r\in [k]\big\}.
\end{equation}
Notice that $U_t\in\subtr_1^0(W\upharpoonright k,i)$. Next observe that, by \eqref{8e39}, we have $W(0)\in A$. On the other hand, we have
$\mathrm{I}_{V_r}(s^{{k+1}\to r})\in A$ for every $r\in [k]$ since $t\in C_1\cap...\cap C_k$. This shows that $U_t\subseteq A$. Observing
that $t=\bar{U}_t[k+1]$ we conclude that $t\in C$. Finally let $j\in J$. Notice that
\begin{equation} \label{8e49}
\dens_{V(j)}(C) = \dens_{V(j)}\big( \{\bar{U}[k+1]: U\in\subtr_1^0(W\upharpoonright k,j+1) \text{ with } U\subseteq A\}\big)
\end{equation}
and recall that $j+1\in I$. Hence, by Fact \ref{8f5}, we have
\begin{equation} \label{8e50}
\dens_{V(j)}(C) = \dens\big( \{U\in\subtr_1^0(W\upharpoonright k,j+1): U\subseteq A\} \big) \stackrel{\eqref{8e39}}{\meg} \vartheta_1/2
\end{equation}
as desired.

The fact that $A$ and $C$ are disjoint follows by our assumption that $A$ contains no Carlson--Simpson line of $[k+1]^{<\nn}$ and the definition
of the set $C$. Finally, part (iv) is an immediate consequence of \eqref{8e42}. The proof is completed.
\end{proof}
\noindent 8.4. \textbf{Consequences.} In this subsection we summarize what we have achieved in Proposition \ref{8p1} and Lemma \ref{8l4}.
We remark that the resulting statement is the first main step towards the proof of Theorem B.
\begin{cor} \label{8c6}
Let $k\in\nn$ with $k\meg 2$ and assume that for every $0<\beta\mik 1$ the number $\dcs(k,1,\beta)$ has been defined.

Let $0<\delta\mik 1$ and define $\eta_1$ as in \eqref{8e33}. Also let $n\in\nn$ with $n\meg 1$ and $N$ be a finite subset
of $\nn$ such that
\begin{equation} \label{8e51}
|N|\meg G_1\big( \lceil \eta_1^{-4} (k+1)kn\rceil, \eta_1^2/2 \big)
\end{equation}
where $G_1$ is as in \eqref{8e34}. Finally let $A\subseteq [k+1]^{<\nn}$ such that $|A\cap [k+1]^l|\meg \delta (k+1)^l$ for every $l\in N$
and assume that $A$ contains no Carlson--Simpson line of $[k+1]^{<\nn}$. Then there exist a Carlson--Simpson tree $V$ of $[k+1]^{<\nn}$,
a subset $D$ of $V$ and a subset $I$ of $\{0,...,\dim(V)\}$ with the following properties.
\begin{enumerate}
\item[(i)]  We have $|I|\meg n$.
\item[(ii)] We have $D=\bigcap_{r=1}^k D_r$ where the set $D_r$ is $(r,k+1)$-insensitive in $V$ for every $r\in [k]$. 
\item[(iii)] For every $i\in I$ we have $\dens_{V(i)}(A\cap D)\meg(\delta+\eta_1^2/2) \dens_{V(i)}(D)$ and 
$\mathrm{dens}_{V(i)}(D)\meg \eta_1^2/2$.
\end{enumerate}
\end{cor}
\begin{proof}
First assume that there exists a Carlson--Simpson tree $W$ of $[k+1]^{<\nn}$ such that
\begin{equation} \label{8e52}
|\big\{i\in \{0,...,\dim(W)\}:\dens_{W(i)}(A)\meg \delta+\eta_1^2/2\big\}|\meg kn.
\end{equation}
In this case we set ``$V= W$'', ``$I=\{i\in \{0,...,\dim (W)\}: \dens_{W(i)}(A)\meg \delta + \eta_1^2/2\}$'' and ``$D_r=V$'' for every $r\in[k]$.
It is clear that with these choices the result follows. 

Otherwise, by Lemma \ref{8l4}, there exist a Carlson--Simpson tree $V$ of $[k+1]^{<\nn}$, a subset $J$ of $\{0,...,\dim(V)\}$ with
$|J|\meg kn$ and a set $C=C_1\cap ... \cap C_k$, where $C_r$ is $(r,k+1)$-insensitive in $V$ for every $r\in [k]$, such that
\begin{enumerate}
\item[(a)] $A\cap C=\varnothing$,
\item[(b)] $\dens_{V(j)}(C)\meg\vartheta_1/2$ for every $j\in J$ and
\item[(c)] $\dens_{V(j)}(A)\meg \delta-5k\eta_1$ for every $j\in J$
\end{enumerate}
where $\vartheta_1$ and $\eta_1$ are as in \eqref{8e33}. In particular, for every $j\in J$ we have
\begin{equation} \label{8e53}
\frac{\dens_{V(j)}(A)}{\dens_{V(j)}(V\setminus C)} \meg \frac{\delta-5k\eta_1}{1-\vartheta_1/2} \meg (\delta-5k\eta_1)(1+\vartheta_1/2)
\meg \delta+7k\eta_1.
\end{equation}
We set $Q_1=V\setminus C_1$ and $Q_r=(V\setminus C_r)\cap C_1\cap...\cap C_{r-1}$ if $r\in \{2,...,k\}$. Clearly the family $(Q_r)_{r=1}^k$
forms a partition of $V\setminus C$. Let $j\in J$ be arbitrary. Applying Lemma \ref{2l6} for ``$\ee=k\eta_1$'' we see that there exists
$r_j\in [k]$ such that
\begin{equation} \label{8e54}
\dens_{V(j)}(A\cap Q_{r_j})\meg (\delta+6k\eta_1)\dens_{V(j)}(Q_{r_j})
\end{equation}
and
\begin{equation} \label{8e55}
\dens_{V(j)}(Q_{r_j})\meg (\delta-5k\eta_1)\eta_1/4.
\end{equation}
Hence, there exist $r_0\in [k]$ and a subset $I$ of $J$ with $|I|\meg |J|/k\meg n$ such that $r_i=r_0$ for every $i\in I$. We set ``$D=Q_{r_0}$''.
Also let ``$D_r=C_r$'' if $r<r_0$, ``$D_{r_0}=V\setminus C_{r_0}$''  and ``$D_r=V$'' if $r>r_0$. Clearly $D_r$ is $(r,k+1)$-insensitive in $V$ for
every $r\in [k]$ and  $D_1\cap ...\cap D_k=D$. Moreover, by the choice of $\eta_1$, for every $i\in I$ we have
\begin{equation} \label{8e56}
\dens_{V(i)}(A\cap D) \stackrel{\eqref{8e54}}{\meg} (\delta+6k\eta_1) \dens_{V(i)}(D) \meg (\delta+\eta_1^2/2)\dens_{V(i)}(D)
\end{equation}
and
\begin{equation} \label{8e57}
\dens_{V(i)}(D) \stackrel{\eqref{8e55}}{\meg} (\delta-5k\eta_1)\eta_1/4 \meg \eta_1^2/2.
\end{equation}
The proof is thus completed.
\end{proof}


\section{An exhaustion procedure: achieving the density increment}

\numberwithin{equation}{section}

\noindent 9.1. \textbf{Motivation.} As we have seen in Corollary \ref{8c6} if a dense subset $A$ of $[k+1]^{<\nn}$ fails to contain a 
Carlson--Simpson line, then there exist a Carlson--Simpson tree $V$ of $[k+1]^{<\nn}$ and a structured subset $D$ of $V$ (recall that $D$
is the intersection of relatively few insensitive sets) that correlates with the set $A$ more than expected in many levels of $V$. Our goal
in this section is to use this information to achieve density increment for the set $A$. A natural strategy for doing so -- initiated by
M. Ajtai and E. Szemer\'{e}di in \cite{AS} -- is to produce an ``almost tiling" of the set $D$, that is, to construct a collection $\mathcal{V}$
of pairwise disjoint Carlson--Simpson trees of sufficiently large dimension which are all contained in $D$ and are such that the set
$D\setminus\cup\mathcal{V}$ is essentially negligible. Once this is done, one then expects to be able to find a Carlson--Simpson tree $W$
belonging to the ``almost tiling" $\mathcal{V}$ such that the density of the set $A$ has been significantly increased in sufficiently many
levels of $W$. However, as is shown below, this is not possible in general.
\begin{examp} \label{9ex1}
Let $m\in\nn$ with $m\meg 1$ and $0<\ee\mik 1$ be arbitrary. Also let $q,\ell\in \nn$ with $q\meg \ell\meg 1$ and such that 
$2^{\ell}(2^{\ell}-1)2^{-q}\mik \ee$. With these choices it is possible to select a family $\{x_t\in [2]^{q-\ell}:t\in[2]^{<\ell}\}$ such that
for every $t,t'\in [2]^{<\ell}$ with $t\neq t'$ we have that $x_t\neq x_{t'}$. For every $i\in\{0,...,\ell-1\}$ and every $t\in [2]^i$ we set
$z^1_t=t^{\con}(1^{\ell-i})^{\con}x_t$ and $z^2_t=t^{\con}(2^{\ell-i})^{\con}x_t$ and we define
\begin{equation} \label{tl1-e1}
\mathcal{F}=\big\{z^j_t:t\in [2]^{<\ell} \text{ and } j\in [2]\big\}.
\end{equation}
Notice that $\mathcal{F}\subseteq [2]^q$ and $\dens(\mathcal{F})\mik \ee$. Also observe that $\{z^1_t,z^2_t\}\cap \{z^1_{t'},z^2_{t'}\}=\varnothing$
provided that $t\neq t'$. We set $D=[2]^{<\ell}\cup [2]^q\cup ...\cup [2]^{q+m-1}$ and
\begin{equation} \label{tl1-e2}
A=[2]^{<\ell}\cup \big\{y^{\con}s: y\in [2]^q\setminus \mathcal{F} \text{ and } s\in[2]^{<m}\big\}.
\end{equation}
It is clear that $D$ is a highly structured subset of $[2]^{<\nn}$ -- it is the union of certain levels of $[2]^{<\nn}$ --  and $A$ is a subset
of $D$ of relative density at least $1-\ee$. Next for every $t\in [2]^{<\ell}$ let
\begin{equation} \label{tl1-e3}
V_t=\{t\}\cup \big\{z_t^{j\con}s: j\in[2] \text{ and } s\in [2]^{<m}\big\}.
\end{equation}
Observe that $\mathcal{V}=\{V_t:t\in [2]^{<\ell}\}$ is a family of pairwise disjoint $m$-dimensional Carlson--Simpson trees which are all
contained in $D$. Also notice that, no matter how large $\ell$ is, $\mathcal{V}$ is maximal, that is, the set $D\setminus\cup\mathcal{V}$
contains no Carlson--Simpson tree of dimension $m$. However, $V_t\cap A$ is the singleton $\{t\}$ for every $t\in [2]^{<\ell}$.
\end{examp}
The above example shows that, in our context, the problem of achieving the density increment cannot be solved by merely producing an arbitrary
``almost tiling" of the structured set $D$. To overcome this obstacle we devise a refined exhaustion procedure that can be roughly described
as follows. At each step of the process we are given a subset $D'$ of $D$ and we produce a collection $\mathcal{E}$ of Carlson--Simpson trees
of sufficiently large dimension which are all contained in $D'$. These Carlson--Simpson trees are \textit{not} pairwise disjoint since we are
not aiming at producing a tiling. Instead, what we are really interested in is whether a sufficient portion of them behaves ``as expected''.
If this is the case, then we can easily achieve the density increment. Otherwise, using coloring arguments, we can show that for ``almost every''
Carlson--Simpson tree $V$ of the collection $\mathcal{E}$, the restriction of our set $A$ in $V$ is quite ``thin" and in a very canonical way.
We then remove from $D'$ an appropriately chosen subset of $\cup\mathcal{E}$ and we repeat the argument for the resulting set. The above process
is shown that it will eventually terminate, completing thus the proof of this step.

At a technical level, in order to execute the steps described above we need to represent any subset of $[k+1]^{<\nn}$ as a family of measurable
events indexed by an appropriately chosen Carlson--Simpson tree of $[k+1]^{<\nn}$. The philosophy is identical to that in \S 8.2. However, due
to the recursive nature of the process, we need to work with iterated convolutions. In particular, the reader is advised to review the material
in \S 6 before studying this section.
\medskip

\noindent 9.2. \textbf{The main result.} Let $k\in\nn$ with $k\meg 2$ and assume that for every integer $l\meg 1$ and every $0<\beta\mik 1$
the number $\dcs(k,l,\beta)$ has been defined. This assumption permits us to introduce some numerical invariants. Specifically, for every
integer $m\meg 1$ and every $0<\gamma\mik 1$ we set
\begin{equation} \label{9e4}
\bar{m}= \bar{m}(m,\gamma)=\Big\lceil\frac{512}{\gamma^3} m\Big\rceil
\end{equation}
and
\begin{equation} \label{9e5}
M=\Lambda(k,\bar{m},\gamma^2/32) \stackrel{\eqref{7e15}}{=} \Big\lceil\frac{32}{\gamma^2}\dcs(k,\bar{m},\gamma^2/32)\Big\rceil.
\end{equation}
Also let
\begin{equation} \label{9e6}
\alpha=\alpha(k,m,\gamma)=\theta(k,\bar{m},\gamma^2/8) \text{ and } p_0=p_0(k,m,\gamma)=\lfloor \alpha^{-1} \rfloor
\end{equation}
where $\theta(k,\bar{m},\gamma^2/8)$ is as in \eqref{7e45}. Finally we define, recursively, three sequences $(n^1_p)$, $(n^2_p)$ and $(N_p)$
in $\nn$ -- also depending on the parameters $m$ and $\gamma$ -- by the rule $n^1_0=n^2_0=N_0=0$ and
\begin{equation} \label{9e7}
\left\{ \begin{array} {l} n^1_{p+1}=(N_p+1)\bar{m}+N_p, \\
n^2_{p+1}= \mathrm{CS}(k+1, n^1_{p+1}, \bar{m}, \bar{m}+1), \\
N_{p+1}=\mathrm{CS}(k+1, \max\{n^2_{p+1},M\}, \bar{m}, 2). \end{array}  \right.
\end{equation}
We are mainly interested in the sequence $(N_p)$. The sequences $(n^1_p)$ and $(n^2_p)$ are auxiliary ones which will be used in the proof
of the following lemma.
\begin{lem} \label{9l1}
Let $k\in\nn$ with $k\meg 2$ and assume that for every integer $l\meg 1$ and every $0<\beta\mik 1$ the number $\dcs(k,l,\beta)$ has been defined.

Let $0<\gamma,\delta\mik 1$ and $r\in [k]$. Also let $V$ be a Carlson--Simpson tree of $[k+1]^{<\nn}$ and $I$ be a nonempty subset
of $\{0,...,\dim(V)\}$. Assume that we are given subsets  $A,D_r,...,D_k$ of $[k+1]^{<\nn}$ with the following properties.
\begin{enumerate}
\item[(a)] The set $D_r$ is $(r,k+1)$-insensitive in $V$.
\item[(b)] We have $\dens_{V(i)}(D_r\cap... \cap D_k\cap A)\meg(\delta+2\gamma)\dens_{V(i)}(D_r\cap ...\cap D_k)$
and $\dens_{V(i)}(D_r\cap ...\cap D_k)\meg 2\gamma$ for every $i\in I$.
\end{enumerate}
Finally, let $m\in\nn$ with $m\meg 1$ and suppose that
\begin{equation} \label{9e8}
|I|\meg \reg(k+1,N_{p_0}+1,2,\gamma^2/2)
\end{equation}
where $p_0$ and $N_{p_0}$ are defined in \eqref{9e6} and \eqref{9e7} respectively for the parameters $m$ and $\gamma$. Then there exist a
Carlson--Simpson subtree $W$ of $V$ and a subset $I'$ of $\{0,...,\dim(W)\}$ of cardinality $m$ with the following properties. If $r<k$, then
\begin{equation} \label{9e9}
\dens_{W(i)}(D_{r+1}\cap...\cap D_k\cap A) \meg (\delta+\gamma/2) \dens_{W(i)}(D_{r+1}\cap...\cap D_k)
\end{equation}
and
\begin{equation} \label{9e10}
\dens_{W(i)}(D_{r+1}\cap...\cap D_k)\meg\frac{\gamma^3}{256}
\end{equation}
for every $i\in I'$. On the other hand if $r=k$, then
\begin{equation} \label{9e11}
\dens_{W(i)}(A) \meg \delta+\gamma/2
\end{equation}
for every $i\in I'$.
\end{lem}
The proof of Lemma \ref{9l1} will be given in \S 9.3. As the reader might have already guessed, Lemma \ref{9l1} is the main result of this
section and incorporates the exhaustion procedure outlined in \S 9.1. It will be used in \S 9.4 where we shall achieve the density increment.
\medskip

\noindent 9.3. \textbf{Proof of Lemma \ref{9l1}.} The first step of the proof relies on an application of the regularity lemma presented
in \S 3. Precisely, let $B=D_{r+1}\cap ...\cap D_k$ if $r<k$; otherwise, let $B=V$. Identifying the Carlson--Simpson tree $V$ with
$[k+1]^{<\dim(V)+1}$ via the canonical isomorphism $\mathrm{I}_V$ (see \S 2.5), we may apply Lemma \ref{3l2} to the family
$\mathcal{F}=\{D_r\cap B, D_r\cap B\cap A\}$ and we get a subset $L$ of $I$ of cardinality $N_{p_0}+1$ such that $\mathcal{F}$
is $(\gamma^2/2,L)$-regular. We set
\begin{equation} \label{9e12}
A^0=\cv_{L,V}^{-1}(A), \ B^0=\cv_{L,V}^{-1}(B) \text{ and } D^0=\cv_{L,V}^{-1}(D_r).
\end{equation}
The fact that $\mathcal{F}$ is $(\gamma^2/2,L)$-regular and conditions (a) and (b) in the statement of the lemma have some consequences
which are isolated in the following fact. Its proof is similar to the proof of Lemma \ref{7l2} and is left to the reader.
\begin{fact} \label{9f2}
For every $t\in [k+1]^{<|L|}$ let $A^0_t, B^0_t$ and $D^0_t$ be the sections at $t$ of $A^0, B^0$ and $D^0$ respectively. Then
for every $t,t'\in[k+1]^{<|L|}$ the following hold.
\begin{enumerate}
\item[(i)] If $t,t'$ are $(r,k+1)$-equivalent (see \S 2.6), then $D^0_t$ and $D^0_{t'}$ coincide.
\item[(ii)] We have $\dens_{X_L}(D^0_t\cap B^0_t\cap A^0_t)\meg(\delta+\gamma)\dens_{X_L}(D^0_t\cap B^0_t)$.
\item[(iii)] We have $\dens_{X_L}(D^0_t\cap B^0_t)\meg\gamma$.
\end{enumerate}
\end{fact}
We are ready to proceed to the main part of the proof. We will argue by contradiction. In particular, assuming that
the lemma is not satisfied, we shall determine an integer $d\in [p_0]$ and we shall construct
\begin{enumerate}
\item[(1)] a $(k+1)$-compatible pair $\big((L_n)_{n=0}^d,(V_n)_{n=0}^d\big)$ with $L_0=L$ and $V_0=V$, and
\item[(2)] for every $p\in [d]$, every $\ell\in [p]$ and every $s\in [k+1]^{<|L_p|}$ a family $\mathcal{Q}^{\ell,p}_s$ of subsets
of $X_{\bl_p}$, where $\bl_p=(L_n)_{n=0}^p$ and $\bv_p=(V_n)_{n=0}^p$.
\end{enumerate}
The construction is done recursively so that, setting
\begin{equation} \label{9e13}
A^p=\cv_{\bl_p,\bv_p}^{-1}(A), \ B^p=\cv_{\bl_p,\bv_p}^{-1}(B) \text{ and } D^p=\cv_{\bl_p,\bv_p}^{-1}(D_r),
\end{equation}
for every $p\in [d]$ the following conditions are satisfied.
\begin{enumerate}
\item[(C1)] The set $L_p$ has cardinality $N_{p_0-p}+1$.
\item[(C2)] For every $\ell\in[p]$ and every $s\in[k+1]^{<|L_p|}$ the family $\mathcal{Q}_s^{\ell,p}$ consists
of pairwise disjoint subsets of the section $D^p_s$ of $D^p$ at $s$.
\item[(C3)] For every $s\in[k+1]^{<|L_p|}$ the sets $\cup\mathcal{Q}^{1,p}_s,...,\cup\mathcal{Q}^{p,p}_s$ are pairwise disjoint.
\item[(C4)] For every $\ell\in [p]$, every pair $s,s'\in[k+1]^{<|L_p|}$ with the same length and every $Q\in\mathcal{Q}_s^{\ell,p}$
and $Q'\in\mathcal{Q}_{s'}^{\ell,p}$ we have $\dens_{X_{\bl_p}}(Q)=\dens_{X_{\bl_p}}(Q')$.
\item[(C5)] For every $\ell\in [p]$ and every $s\in[k+1]^{<|L_p|}$ we say that an element $Q$ of $\mathcal{Q}^{\ell,p}_s$ is \textit{good}
provided that $\dens_{Q}(D^p_s\cap B^p_s\cap A^p_s)\meg (\delta+\gamma/2) \dens_{Q}(D^p_s\cap B^p_s)$ and $\dens_{Q}(D^p_s\cap B^p_s)\meg
\gamma^3/256$. Then, setting
\begin{equation} \label{9e14}
\mathcal{G}_s^{\ell,p} = \{ Q\in\mathcal{Q}_s^{\ell,p} : Q \text{ is good}\}
\end{equation}
we have
\begin{equation}\label{9e15}
\frac{|\mathcal{G}_s^{\ell,p}|}{|\mathcal{Q}_s^{\ell,p}|}<\frac{\gamma^3}{256}.
\end{equation}
\item[(C6)] For every $\ell\in[p]$ and every $s\in [k+1]^{<|L_p|}$ we have $\dens_{X_{\bl_p}}(\cup\mathcal{Q}_s^{\ell,p})\meg\alpha$
where $\alpha$ is as in \eqref{9e6}.
\item[(C7)] For every $\ell\in [p]$ and every $s,s'\in [k+1]^{<|L_p|}$ we have $\mathcal{Q}_s^{\ell,p}= \mathcal{Q}_{s'}^{\ell,p}$
if $s$ and $s'$ are $(r,k+1)$-equivalent.
\item[(C8)] If $p=d$, then there exists $s_0\in[k+1]^{<|L_d|}$ such that
\begin{equation} \label{9e16}
\dens_{X_{\bl_d}}\Big( D^d_{s_0}\setminus \bigcup_{\ell=1}^d\cup\mathcal{Q}^{\ell,d}_{s_0} \Big)<\gamma^2/8.
\end{equation}
\end{enumerate}

Assuming that the above construction has been carried out, let us derive the contradiction. Let $s_0$ be as in (C8).
By Corollary \ref{6c8}, Corollary \ref{6c9} and Fact \ref{9f2}, we see that
\begin{equation}\label{9e17}
\dens_{X_{\bl_d}}(D^d_{s_0}\cap B^d_{s_0}\cap A^d_{s_0})\meg(\delta+\gamma) \dens_{X_{\bl_d}}(D^d_{s_0}\cap B^d_{s_0})
\end{equation}
and
\begin{equation}\label{9e18}
\dens_{X_{\bl_d}}(D^d_{s_0}\cap B^d_{s_0})\meg\gamma.
\end{equation}
For every $\ell\in[d]$ we set $C_\ell=\cup\mathcal{Q}^{\ell,d}_{s_0}$. By (\ref{9e16}), the family $\{C_\ell:\ell\in[d]\}$ is an ``almost cover"
of $D^d_{s_0}\cap B^d_{s_0}$. Hence, invoking (\ref{9e17}) and (\ref{9e18}) and applying Lemma \ref{2l6} for ``$\ee=\gamma/4$", we may find
$\ell_0\in[d]$ such that
\begin{equation}\label{9e19}
\dens_{C_{\ell_0}}(D^d_{s_0}\cap B^d_{s_0}\cap A^d_{s_0})\meg(\delta+3\gamma/4) \dens_{C_{\ell_0}}(D^d_{s_0}\cap B^d_{s_0})
\end{equation}
and
\begin{equation}\label{9e20}
\dens_{C_{\ell_0}}(D^d_{s_0}\cap B^d_{s_0})\meg\gamma^2/16.
\end{equation}
Next observe that, by conditions (C2) and (C4), the family $\mathcal{Q}^{\ell_0,d}_{s_0}$ is a partition of $C_{\ell_0}$ into sets of equal size.
Taking into account this observation and the estimates in (\ref{9e19}) and (\ref{9e20}), by a second application of Lemma \ref{2l6}
for ``$\ee=\gamma/4$",  we conclude that
\begin{equation}\label{9e21}
\frac{|\mathcal{G}_{s_0}^{\ell_0,d}|}{|\mathcal{Q}_{s_0}^{\ell_0,d}|} \meg \frac{\gamma^3}{256}.
\end{equation}
This contradicts (\ref{9e15}), as desired.

The rest of the proof is devoted to the description of the recursive construction. For ``$p=0$" we set $L_0=L$ and $V_0=V$.
Let $p\in \{0,...,p_0\}$ and assume that the construction has been carried out up to $p$ so that conditions (C1)-(C8) are satisfied.
We distinguish the following cases.
\medskip

\noindent \textsc{Case 1: }\textit{$p=p_0$.} Notice first that, by (C1), the set $L_p$ is a singleton. Therefore, the set
$[k+1]^{<|L_p|}$ is the singleton $\{\varnothing\}$. We set $s_0=\varnothing$ and $d=p_0$. With these choices, the recursive construction
will be completed once we show that the estimate in (\ref{9e16}) is satisfied. This is, however, an immediate consequence of conditions
(C3) and (C6) and the choice of $\alpha$ and $p_0$ in (\ref{9e6}).
\medskip

\noindent \textsc{Case 2: }\textit{we have that $p\meg 1$ and there exists $s_0\in [k+1]^{<|L_p|}$ such that }
\begin{equation} \label{9e22}
\dens_{X_{\bl_p}}\Big( D^p_{s_0}\setminus \bigcup_{\ell=1}^p\cup\mathcal{Q}^{\ell,p}_{s_0} \Big) <\gamma^2/8.
\end{equation}
In this case we set $d=p$ and we terminate the construction.
\medskip

\noindent \textsc{Case 3: }\textit{we have that $p<p_0$ and either $p=0$, or $p\meg 1$ and 
\begin{equation} \label{9e23}
\dens_{X_{\bl_p}}\Big( D^p_s\setminus \bigcup_{\ell=1}^p\cup\mathcal{Q}^{\ell,p}_s \Big)\meg\gamma^2/8
\end{equation}
for every $s\in [k+1]^{<|L_p|}$.} If $p\meg 1$, then for every $s\in [k+1]^{<|L_p|}$ we set
\begin{equation} \label{9e24}
\Gamma_s= D^p_s\setminus \bigcup_{\ell=1}^p\cup\mathcal{Q}^{\ell,p}_s.
\end{equation}
Otherwise, let $\Gamma_s=D^0_s$. The following fact follows by \eqref{9e23}, condition (a) in the statement of the lemma and condition (C7)
if $p\meg 1$, and by Fact \ref{9f2} if $p=0$.
\begin{fact} \label{9f3}
For every $s\in [k+1]^{<|L_p|}$ we have $\dens_{X_{\bl_p}}(\Gamma_s)\meg \gamma^2/8$. Moreover, if $s,s'\in [k+1]^{<|L_p|}$ are
$(r,k+1)$-equivalent, then $D^p_s=D^p_{s'}$ and $\Gamma_s=\Gamma_{s'}$.
\end{fact}
By Fact \ref{9f3}, condition (C1), the choice of the sequence $(N_p)$ in \eqref{9e7} and the choice of $\alpha$ in \eqref{9e6}, we may apply
Corollary \ref{7c11} to get a Carlson--Simpson subtree $S$ of $[k+1]^{<|L_p|}$ with $\dim(S)=n^2_{p_0-p}$ such that for every $\bar{m}$-dimensional
Carlson--Simpson subtree $U$ of $S$, setting
\begin{equation}\label{9e25}
\Gamma_U=\bigcap_{s\in U}\Gamma_s,
\end{equation}
we have $\dens_{X_{\bl_p}}(\Gamma_U)\meg\alpha$.

For every $i\in\{0,...,\bar{m}\}$ and every Carlson--Simpson subtree $U$ of $S$ of dimension $\bar{m}$ let $G_{i,U}$ be the set of all
$\bx\in\Gamma_U$ satisfying
\begin{enumerate}
\item[(P1)] $\dens_{U(i)\times\{\bx\}}(D^p\cap B^p\cap A^p)\meg(\delta+\gamma/2)\dens_{U(i)\times\{\bx\}}(D^p\cap B^p)$ and
\item[(P2)] $\dens_{U(i)\times\{\bx\}}(D^p\cap B^p)\meg\gamma^3/256$.
\end{enumerate}
Our assumption that the lemma is not satisfied reduces to the following property of the sets $G_{i,U}$.
\begin{claim} \label{9c4}
For every $\bar{m}$-dimensional Carlson--Simpson subtree $U$ of $S$ there exists $i\in\{0,...,\bar{m}\}$ such that $\dens_{\Gamma_U}(G_{i,U})<\gamma^3/256$.
\end{claim}
\begin{proof}
We will argue by contradiction. So, assume that there exists a Carlson--Simpson subtree $U$ of $S$ of dimension $\bar{m}$ such that
for every $i\in\{0,...,\bar{m}\}$ we have $\dens_{\Gamma_U}(G_{i,U})\meg\gamma^3/256$. For every $\bx\in \Gamma_U$ let
\begin{equation} \label{9e26}
I_\bx=\{i\in\{0,...,\bar{m}\}:\bx\in G_{i,U}\}.
\end{equation}
By Lemma \ref{2l5}, setting $G=\{\bx\in\Gamma_U:|I_{\bx}|\meg(\bar{m}+1)\gamma^3/512\}$, we have that $\dens_{\Gamma_U}(G)
\meg \gamma^3/512$. This implies, in particular, that the set $G$ is nonempty. We select $\bx\in G$. By the choice of
$\bar{m}$ in (\ref{9e4}) we have that $|I_\bx|\meg m$. We define $W=\{\cv_{\bl_p,\bv_p}(s,\bx):s\in U\}$. By Lemma \ref{6l4},
$W$ is a Carlson--Simpson subtree of $V$ of dimension $\bar{m}$. Applying Lemma \ref{6l5} twice for the sets
$D_r\cap B\cap A$ and $D_r\cap B$, for every $i\in \{0,...,\bar{m}\}$ we have
\begin{equation}\label{9e27}
\dens_{W(i)}(D_r\cap B\cap A)=\dens_{U(i)\times \{\bx\}}(D^p\cap B^p\cap A^p)
\end{equation}
and
\begin{equation}\label{9e28}
\dens_{W(i)}(D_r\cap B)=\dens_{U(i)\times \{\bx\}}(D^p\cap B^p).
\end{equation}
The above equalities and the fact that $\bx\in G_{i,U}$ for every $i\in I_{\bx}$ yield that
\begin{equation} \label{9e29}
\dens_{W(i)}(D_{r}\cap B\cap A) \meg (\delta+\gamma/2) \dens_{W(i)}(D_{r}\cap B)
\end{equation}
and
\begin{equation} \label{9e30}
\dens_{W(i)}(D_{r}\cap B)\meg\frac{\gamma^3}{256}
\end{equation}
for every $i\in I_{\bx}$. Finally, observe that $W$ is a subset of $D_r$ since $\bx\in \Gamma_U$. It is then clear that $W$ and $I_{\bx}$
satisfy the conclusion of the lemma in contradiction with our assumption.
\end{proof}
We are now in the position to start the process of selecting the new objects of the recursive construction.
\medskip

\noindent \textit{Step 1: selection of $V_{p+1}$ and $L_{p+1}$.} Firstly, we will use a coloring argument to control the integer $i$ obtained
by Claim \ref{9c4}. Specifically, by the choice of the sequence $(n_p^2)$ in (\ref{9e7}) and Claim \ref{9c4}, we may apply Theorem \ref{4t1}
to obtain $i_0\in\{0,...,\bar{m}\}$ and a Carlson--Simpson subtree $T$ of $S$ of dimension $n_{p_0-p}^1$ such that for every $\bar{m}$-dimensional
Carlson--Simpson subtree $U$ of $T$ we have $\dens_{\Gamma_{U}}(G_{i_0,U})<\gamma^3/256$. We define
\begin{equation} \label{9e31}
V_{p+1}=T \text{ and } L_{p+1}=\big\{ i_0+j(i_0+1): j\in\{0,..., N_{p_0-(p+1)}\}\big\}.
\end{equation}
Notice that the pair $\big((L_n)_{n=0}^{p+1},(V_n)_{n=0}^{p+1}\big)$ is $(k+1)$-compatible. This follows by our inductive assumptions and the
choice of the sequence $(n_p^1)$ in (\ref{9e7}). Moreover, the cardinality of the set $L_{p+1}$ is $N_{p_0-(p+1)}+1$. Hence, with these choices,
condition (C1) is satisfied. For notational simplicity, in what follows by $\qv_{p+1}$ we shall denote the quotient map associated to the pair
$(\bl_{p+1},\bv_{p+1})$ defined in (\ref{6e5}).
\medskip

\noindent \textit{Step 2: selection of the families $\mathcal{Q}_t^{p+1,p+1}$.} This is the most important part of the recursive selection.
The members of the families $\mathcal{Q}_t^{p+1,p+1}$ are, essentially, the sets $\Gamma_U$ where $U$ ranges over all $\bar{m}$-dimensional
Carlson--Simpson subtrees of $V_{p+1}$. However, in order to carry out the construction, we have to group them in a canonical way. We proceed
to the details.

Consider the canonical isomorphism $\mathrm{I}_{V_{p+1}}:[k+1]^{<\dim(V_{p+1})+1}\to V_{p+1}$ defined in \S 2.5;
for convenience it will be denoted by $\mathrm{I}$. For every $j\in \{0,...,|L_{p+1}|-1\}$ and every $t\in [k+1]^j$ we set
\begin{equation} \label{9e32}
\Omega_t=\big\{ \cv_{L_{p+1},V_{p+1}}(t,x): x\in X_{L_{p+1}}\big\}\subseteq V_{p+1}\big(i_0+j(i_0+1)\big).
\end{equation}
If $j\meg 1$, then let $t_*$ be the unique initial segment of $t$ of length $j-1$ and set
\begin{equation} \label{9e33}
K_t=\big\{\mathrm{I}^{-1}(w)^{\con}t(j-1):w\in\Omega_{t_*}\big\}\subseteq [k+1]^{j(i_0+1)};
\end{equation}
otherwise, let $K_\varnothing=\{\varnothing\}$. Notice that $K_t=\{ \cv_{L_{p+1}}(t_*,x)^{\con}t(j-1):x\in X_{L_{p+1}}\}$.
Finally for every $s\in K_t$ let $C_s=s^\con[k+1]^{<\bar{m}+1}$ and set
\begin{equation} \label{9e34}
\mathcal{P}_t=\big\{\mathrm{I}(C_s):s\in K_t\big\}.
\end{equation}
Before we analyze the above definitions, let us give a specific example. For concreteness take $k=3$ and assume, for notational simplicity,
that $V_{p+1}$ is of the form $[4]^{<n}$ where $n$ is large enough compared to $i_0$ (hence, the map $\mathrm{I}$ is the identity). Consider
the sequence $t=(1,2,1)$ and observe that $t_*=(1,2)$. Notice that $\Omega_t$ is the subset of $[4]^{4i_0+3}$ consisting of all finite sequences
$x$ such that $x(i_0)=t(0)=1$, $x(2i_0+1)=t(1)=2$ and $x(3i_0+2)=t(2)=1$. On the other hand, the set $K_t$ is the subset of $[4]^{3i_0+3}$
consisting of all sequences $x$ such that $x(i_0)=t(0)=1$, $x(2i_0+1)=t(1)=2$ and $x(3i_0+2)=t(2)=1$. It is then easy to see that in this
specific case the family $\{ U(i_0): U\in\mathcal{P}_t\}$ forms a partition of the set $\Omega_t$. This is, actually, a general property as is
shown in the following fact.
\begin{fact} \label{9f5}
Let $t\in[k+1]^{<|L_{p+1}|}$ be arbitrary. Then the family $\mathcal{P}_t$ consists of pairwise disjoint $\bar{m}$-dimensional Carlson--Simpson
subtrees of $V_{p+1}$. Moreover,
\begin{equation} \label{9e35}
\Omega_t=\bigcup_{U\in\mathcal{P}_t} U(i_0).
\end{equation}
\end{fact}
\begin{proof}
It is clear that the family $\mathcal{P}_t$ consists of pairwise disjoint Carlson--Simpson trees. Moreover, by the choice of $L_{p+1}$ in
(\ref{9e31}), we have
\begin{eqnarray} \label{9e36}
\mathrm{I}^{-1}(\Omega_t) & = & \big\{s^\con y:s\in K_t\text{ and } y\in[k+1]^{i_0}\big\} \\
& = & \bigcup_{s\in K_t} C_s(i_0)= \bigcup_{U\in\mathcal{P}_t}\mathrm{I}^{-1}(U)(i_0) \nonumber
\end{eqnarray}
and the proof is completed.
\end{proof}
We record, for future use, another property of the family $\mathcal{P}_t$.
\begin{fact} \label{9f6}
Let $t,t'\in[k+1]^{<|L_{p+1}|}$ with the same length. Then for every $U\in\mathcal{P}_t$ and every $U'\in\mathcal{P}_{t'}$  we have
$\dens_{\Omega_t}\big(U(i_0)\big) = \dens_{\Omega_{t'}}\big(U'(i_0)\big)$.
\end{fact}
\begin{proof}
By Fact \ref{9f5}, we have $U(i_0)\subseteq\Omega_{t}$ and $U'(i_0)\subseteq\Omega_{t'}$ while, by Fact \ref{5f2}, we have
$|\Omega_t|=|\Omega_{t'}|$. Noticing that $|U(i_0)|=|U'(i_0)|=(k+1)^{i_0}$ the result follows.
\end{proof}
We are ready to define the families $\mathcal{Q}_t^{p+1,p+1}$. Specifically, fix $t\in[k+1]^{<|L_{p+1}|}$. For every $U\in\mathcal{P}_t$
and every $\bx\in\Gamma_U$ let
\begin{equation} \label{9e37}
Q_t^{\bx,U}=\{\bx\}\times \{x\in X_{L_{p+1}}:\cv_{L_{p+1},V_{p+1}}(t,x)\in U(i_0)\}\subseteq X_{\bl_{p+1}}.
\end{equation}
By Fact \ref{9f5}, the set $U(i_0)$ is contained in $\Omega_t$. Hence,
\begin{equation} \label{9e38}
\{t\}\times Q_t^{\bx,U} = \qv_{p+1}^{-1}\big(U(i_0)\times\{\bx\}\big).
\end{equation}
We define
\begin{equation}\label{9e39}
\mathcal{Q}_t^{p+1,p+1}=\{Q_t^{\bx,U}:U\in\mathcal{P}_t\text{ and }\bx\in\Gamma_U\}.
\end{equation}
This completes the second step of the recursive selection.
\medskip

\noindent\textit{Step 3: selection of the families $\mathcal{Q}^{\ell,p+1}_t$ for every $\ell\in[p]$.} In this step we will not introduce
something new but rather ``copy'' in the space $X_{\bl_{p+1}}$ what we have constructed so far. In particular, this step is meaningful
only if $p\meg1$.

Specifically, let $p\meg1$ and fix $t\in[k+1]^{<|L_{p+1}|}$ and $\ell\in[p]$. For every $s\in\Omega_t$ and every $Q\in\mathcal{Q}^{\ell,p}_s$ let
\begin{equation}\label{9e40}
C_t^{s,\ell,Q}=Q\times\{x\in X_{L_{p+1}}:\cv_{L_{p+1},V_{p+1}}(t,x)=s\}\subseteq X_{\bl_{p+1}}.
\end{equation}
Notice that
\begin{equation} \label{9e41}
\{t\}\times C_t^{s,\ell,Q} = \qv_{p+1}^{-1}\big(\{s\}\times Q\big).
\end{equation}
We define
\begin{equation}\label{9e42}
\mathcal{Q}_t^{\ell,p+1}=\{C_t^{s,\ell,Q}:s\in\Omega_t\text{ and }Q\in\mathcal{Q}_s^{\ell,p}\}.
\end{equation}
This completes the third, and final, step of the recursive selection.
\medskip

\noindent\textit{Step 4: verification of the inductive assumptions.}  Recall that condition (C1) has already been verified in Step 1.
Also observe that condition (C8) is meaningless in this case. Conditions (C2) up to (C7) will be verified in the following claims.
\begin{claim}\label{9c7}
For every $\ell\in[p+1]$ and every $t\in[k+1]^{<|L_{p+1}|}$ the family $\mathcal{Q}_t^{\ell,p+1}$ consists of pairwise disjoint subsets of
$D^{p+1}_t$. That is, condition \emph{(C2)} is satisfied.
\end{claim}
\begin{proof}
Fix $t\in[k+1]^{<|L_{p+1}|}$. First assume that $\ell\in[p]$. Invoking (\ref{9e40}), (\ref{9e42}) and our inductive assumptions,
we see that the family $\mathcal{Q}_t^{\ell,p+1}$ consists of pairwise disjoint sets. Fix $C_t^{s,\ell,Q}\in \mathcal{Q}_t^{\ell,p+1}$
for some $s\in\Omega_t$ and $Q\in\mathcal{Q}^{\ell,p}_s$. Using our inductive assumptions once again, we see that $Q\subseteq D^p_s$
or equivalently $\{s\}\times Q\subseteq D^p$. Hence, by \eqref{9e41} and Fact \ref{6f3},
\begin{equation} \label{9e43}
\{t\}\times C_t^{s,\ell,Q} \subseteq \qv_{p+1}^{-1}(D^p)=D^{p+1}.
\end{equation}

Now we treat the case ``$\ell=p+1$''. Let $U,U'\in \mathcal{P}_t$, $\bx\in\Gamma_U$ and $\bx'\in\Gamma_{U'}$. We will show that the sets $Q_t^{\bx,U}$
and $Q_t^{\bx',U'}$ are disjoint provided that $(U,\bx)\neq (U',\bx')$. Indeed, if $U=U'$, then necessarily $\bx\neq\bx'$ which implies that the
sets $Q_t^{\bx,U}$ and $Q_t^{\bx',U'}$ are disjoint. Otherwise, by Fact \ref{9f5}, the Carlson--Simpson trees $U$ and $U'$ are disjoint.
In particular, the sets $U(i_0)$ and $U'(i_0)$ are disjoint yielding that $Q_t^{\bx,U}\cap Q_t^{\bx',U'}=\varnothing$. What remains is to check
that $Q_t^{\bx,U}\subseteq D^{p+1}_t$ for every $U\in\mathcal{P}_t$ and every $\bx\in \Gamma_U$. So, fix $U\in\mathcal{P}_t$ and $\bx\in \Gamma_U$.
By \eqref{9e38} and Fact \ref{6f3}, we conclude that
\begin{equation} \label{9e44}
\{t\}\times Q_t^{\bx,U} \subseteq \qv_{p+1}^{-1}(D^p)=D^{p+1}
\end{equation}
and the proof is completed.
\end{proof}
\begin{claim}\label{9c8}
For every $t\in[k+1]^{<|L_{p+1}|}$ the sets $\cup\mathcal{Q}_t^{1,p+1}, ..., \cup\mathcal{Q}_t^{p+1,p+1}$ are pairwise disjoint.
That is, condition \emph{(C3)} is satisfied.
\end{claim}
\begin{proof}
Fix $t\in[k+1]^{<|L_{p+1}|}$ and $\ell\in[p]$. Let $\ell'\in[p+1]$ with $\ell'\neq\ell$. We need to show that the sets
$\cup\mathcal{Q}^{\ell,p+1}_t$ and $\cup\mathcal{Q}^{\ell',p+1}_t$ are disjoint. If $\ell' \mik p$, then this follows immediately from
(\ref{9e40}) and our inductive assumptions. So, assume that $\ell'=p+1$ and let $U\in\mathcal{P}_t$ and $\bx\in \Gamma_U$. By (\ref{9e24})
and (\ref{9e25}), we have that $\bx\not\in \cup\mathcal{Q}^{\ell,p}_s$ for every $s\in U$. Using this observation, the result follows from
(\ref{9e38}) and (\ref{9e41}).
\end{proof}
\begin{claim}\label{9c9}
For every $\ell\in [p+1]$, every $t,t'\in[k+1]^{<|L_{p+1}|}$ with the same length and every $\Theta\in\mathcal{Q}_t^{\ell,p+1}$
and $\Theta'\in\mathcal{Q}_{t'}^{\ell,p+1}$ we have $\dens_{X_{\bl_{p+1}}}(\Theta)=\dens_{X_{\bl_{p+1}}}(\Theta')$.
That is, condition \emph{(C4)} is satisfied.
\end{claim}
\begin{proof}
If $\ell\in[p]$, then by (\ref{9e41}) there exist $s\in\Omega_t$ and $s'\in\Omega_{t'}$ as well as $Q\in\mathcal{Q}_s^{\ell,p}$ and
$Q'\in\mathcal{Q}_{s'}^{\ell,p}$ such that $\{t\}\times\Theta=\qv_{p+1}^{-1}\big(\{s\}\times Q\big)$ and
$\{t'\}\times\Theta'=\qv_{p+1}^{-1}\big(\{s'\}\times Q'\big)$. By Fact \ref{5f2}, we have that $s$ and $s'$ have the same length and 
$|\Omega_t|=|\Omega_{t'}|$. Therefore, by our inductive assumptions,
\begin{eqnarray} \label{9e45}
\dens_{\Omega_t\times X_{\bl_p}}\big(\{s\}\times Q\big)
& = & \frac{1}{|\Omega_t|}\dens_{\{s\}\times X_{\bl_p}}\big(\{s\}\times Q\big)\\
& = & \frac{1}{|\Omega_t|}\dens_{X_{\bl_p}}(Q)
 =  \frac{1}{|\Omega_{t'}|}\dens_{X_{\bl_p}}(Q') \nonumber \\
& = & \frac{1}{|\Omega_{t'}|}\dens_{\{s'\}\times X_{\bl_p}}\big(\{s'\}\times Q'\big) \nonumber \\
& = & \dens_{\Omega_{t'}\times X_{\bl_p}}\big(\{s'\}\times Q'\big). \nonumber
\end{eqnarray}
Applying Lemma \ref{6l6} we conclude that
\begin{eqnarray} \label{9e46}
\dens_{X_{\bl_{p+1}}}(\Theta)
& = & \dens_{\{t\}\times X_{\bl_{p+1}}}\big(\{t\}\times\Theta\big) \\
& = & \dens_{\Omega_t\times X_{\bl_p}}\big(\{s\}\times Q\big) \nonumber \\
& \stackrel{\eqref{9e45}}{=} &  \dens_{\Omega_{t'}\times X_{\bl_p}}\big(\{s'\}\times Q'\big) \nonumber \\
& = & \dens_{\{t'\}\times X_{\bl_{p+1}}}\big(\{t'\}\times\Theta'\big) = \dens_{X_{\bl_{p+1}}}(\Theta'). \nonumber
\end{eqnarray}

If $\ell=p+1$, then by (\ref{9e38}) there exist $U\in\mathcal{P}_t$ and $U'\in\mathcal{P}_{t'}$ as well as $\bx\in\Gamma_U$ and
$\bx'\in\Gamma_{U'}$ such that $\{t\}\times \Theta = \qv_{p+1}^{-1}\big(U(i_0)\times\{\bx\}\big)$ and
$\{t'\}\times \Theta = \qv_{p+1}^{-1}\big(U'(i_0)\times\{\bx'\}\big)$. By Fact \ref{9f6}, we have
\begin{eqnarray} \label{9e47}
\dens_{\Omega_t\times X_{\bl_p}}\big(U(i_0)\times \{\bx\}\big)
& = & \frac{1}{|X_{\bl_p}|}\dens_{\Omega_t}\big(U(i_0)\big) \\
& = & \frac{1}{|X_{\bl_p}|}\dens_{\Omega_{t'}}\big(U'(i_0)\big) \nonumber\\
& = & \dens_{\Omega_{t'}\times X_{\bl_p}}\big(U'(i_0)\times \{\bx'\}\big). \nonumber
\end{eqnarray}
Using (\ref{9e47}), Lemma \ref{6l6} and arguing precisely as in the previous case, we see that
$\dens_{X_{\bl_{p+1}}}(\Theta)=\dens_{X_{\bl_{p+1}}}(\Theta')$ and the proof is completed.
\end{proof}
\begin{claim} \label{9c10}
For every $\ell\in [p+1]$ and every $t\in[k+1]^{<|L_{p+1}|}$ we have
\begin{equation}\label{9e48}
\frac{|\mathcal{G}_t^{\ell,p+1}|}{|\mathcal{Q}_t^{\ell,p+1}|}<\frac{\gamma^3}{256}.
\end{equation}
That is, condition \emph{(C5)} is satisfied.
\end{claim}
\begin{proof}
Fix $t\in[k+1]^{<|L_{p+1}|}$. Assume first that $\ell\in[p]$. For every $s\in \Omega_t$ let
\begin{equation} \label{9e49}
\mathcal{Q}_s=\big\{C_t^{s,\ell,Q}:Q\in\mathcal{Q}_s^{\ell,p}\big\} \text{ and } \mathcal{G}_s= \mathcal{G}_t^{\ell,p+1}\cap \mathcal{Q}_s.
\end{equation}
By (\ref{9e42}), the family $\{\mathcal{Q}_s:s\in\Omega_t\}$ is a partition of $\mathcal{Q}_t^{\ell,p+1}$. The family
$\{\mathcal{G}_s:s\in\Omega_t\}$ is the induced partition of $\mathcal{G}_t^{\ell,p+1}$. Moreover, for every $s\in\Omega_t$ let
\begin{equation} \label{9e50}
\mathcal{B}_s=\{C_t^{s,\ell,Q}:Q\in\mathcal{G}_s^{\ell,p}\}.
\end{equation}
\begin{subclaim} \label{9sc11}
For every $s\in\Omega_t$ we have $\mathcal{G}_s=\mathcal{B}_s$.
\end{subclaim}
\begin{proof}[Proof of Subclaim \ref{9sc11}]
Fix $s\in\Omega_t$ and let $\Theta\in\mathcal{Q}_s$ be arbitrary. By (\ref{9e41}), the map $\mathcal{Q}_s^{\ell,p}\ni Q\mapsto
C_t^{s,\ell,Q}\in\mathcal{Q}_s$ is a bijection. Hence, there exists a unique $Q\in\mathcal{Q}_s^{\ell,p}$ such that
$\Theta=C_t^{s,\ell,Q}$. Using (\ref{9e41}) once again, we see that $\{t\}\times\Theta=\qv_{p+1}^{-1}\big(\{s\}\times Q\big)$.
Since $s\in\Omega_t$, by Corollary \ref{6c10}, we get
\begin{eqnarray} \label{9e51}
\dens_{\Theta}(D^{p+1}_t\cap B^{p+1}_t) & = & \dens_{\{t\}\times\Theta}(D^{p+1}\cap B^{p+1}) \\
& = & \dens_{\qv_{p+1}^{-1}(\{s\}\times Q)}\big( \qv_{p+1}^{-1}(D^p\cap B^p)\big) \nonumber \\
& = & \dens_{\{s\}\times Q}(D^p\cap B^p) \nonumber \\
& = & \dens_Q(D^p_s\cap B^p_s). \nonumber
\end{eqnarray}
Similarly,
\begin{equation} \label{9e52}
\dens_{\Theta}(D^{p+1}_t\cap B^{p+1}_t\cap A^{p+1}_t) = \dens_Q(D^p_s\cap B^p_s\cap A^p_s).
\end{equation}
Using (\ref{9e51}) and (\ref{9e52}) and invoking the definition of a good set described in condition (C5), we conclude that
$\Theta\in\mathcal{G}_s$ if and only  if $Q\in\mathcal{G}_s^{\ell,p}$. This is equivalent to say that $\mathcal{G}_s=\mathcal{B}_s$.
\end{proof}
We are ready to complete the proof for the case ``$\ell\in[p]$''. Indeed, we have already pointed out that the map
$\mathcal{Q}_s^{\ell,p}\ni Q\mapsto C_t^{s,\ell,Q}\in\mathcal{Q}_s$ is a bijection. Therefore, using our inductive assumptions, we see that
\begin{eqnarray} \label{9e53}
\frac{|\mathcal{G}_t^{\ell,p+1}|}{|\mathcal{Q}_t^{\ell,p+1}|} & = &
\sum_{s\in\Omega_t} \frac{|\mathcal{G}_s|}{|\mathcal{Q}_s|} \cdot \frac{|\mathcal{Q}_s|}{|\mathcal{Q}_t^{\ell,p+1}|} =
\sum_{s\in\Omega_t} \frac{|\mathcal{B}_s|}{|\mathcal{Q}_s|} \cdot \frac{|\mathcal{Q}_s|}{|\mathcal{Q}_t^{\ell,p+1}|} \\
& = & \sum_{s\in\Omega_t} \frac{|\mathcal{G}^{\ell,p}_s|}{|\mathcal{Q}^{\ell,p}_s|} \cdot
\frac{|\mathcal{Q}_s|}{|\mathcal{Q}_t^{\ell,p+1}|} < \frac{\gamma^3}{256} \Big( \sum_{s\in\Omega_t}
\frac{|\mathcal{Q}_s|}{|\mathcal{Q}_t^{\ell,p+1}|} \Big) = \frac{\gamma^3}{256}. \nonumber
\end{eqnarray}

Now we treat the case $\ell=p+1$. The argument is similar. Specifically, for every $U\in \mathcal{P}_t$ let
\begin{equation} \label{9e54}
\mathcal{Q}_U=\big\{ Q_t^{\bx,U}: \bx\in\Gamma_U\} \text{ and } \mathcal{G}_U= \mathcal{G}_t^{\ell,p+1}\cap \mathcal{Q}_U.
\end{equation}
By (\ref{9e39}), the family $\{\mathcal{Q}_U:U\in\mathcal{P}_t\}$ is a partition of $\mathcal{Q}_t^{\ell,p+1}$ and the family
$\{\mathcal{G}_U: U\in\mathcal{P}_t\}$ is the induced partition of $\mathcal{G}_t^{\ell,p+1}$. Also, for every $U\in\mathcal{P}_t$ let
\begin{equation} \label{9e55}
\mathcal{B}_U=\{ Q_t^{\bx,U}: \bx\in G_{i_0,U}\}.
\end{equation}
Recall that $G_{i_0,U}$ is the set of all $\bx\in\Gamma_U$ satisfying properties (P1) and (P2). Moreover, by the choice of
$i_0$ in Step 1, we have that $\dens_{\Gamma_U}(G_{i_0,U})<\gamma^3/256$. We have the following analogue of Subclaim \ref{9sc11}.
\begin{subclaim} \label{9sc12}
For every $U\in\mathcal{P}_t$ we have $\mathcal{G}_U=\mathcal{B}_U$.
\end{subclaim}
\begin{proof}[Proof of Subclaim \ref{9sc12}]
Fix $U\in\mathcal{P}_t$ and let $\Theta\in\mathcal{Q}_U$ be arbitrary. By (\ref{9e38}), the map $\Gamma_U\ni \bx \mapsto
Q_t^{\bx,U}\in\mathcal{Q}_U$ is a bijection. Hence, there exists a unique $\bx\in\Gamma_U$ such that $Q_t^{\bx,U}=\Theta$. Invoking (\ref{9e38})
once again, $\{t\}\times Q_t^{\bx,U}=\qv_{p+1}^{-1}(U(i_0)\times\{\bx\})$. Moreover, by Fact \ref{9f5}, we have $U(i_0)\subseteq \Omega_t$. By the
previous remarks and Corollary \ref{6c10}, and arguing precisely as in the proof of Subclaim \ref{9sc11}, we see that $\Theta\in\mathcal{G}_U$
if and only if $\bx\in G_{i_0,U}$.
\end{proof}

With Subclaim \ref{9sc12} at our disposal, we are ready complete the proof for the case ``$\ell=p+1$''. We have already pointed out that
the map $\Gamma_U\ni \bx \mapsto Q_t^{\bx,U}\in\mathcal{Q}_U$ is a bijection. Therefore, using our inductive assumptions, we conclude
\begin{eqnarray} \label{9e56}
\frac{|\mathcal{G}_t^{p+1,p+1}|}{|\mathcal{Q}_t^{p+1,p+1}|} & = &
\sum_{U\in\mathcal{P}_t} \frac{|\mathcal{G}_U|}{|\mathcal{Q}_U|} \cdot \frac{|\mathcal{Q}_U|}{|\mathcal{Q}_t^{p+1,p+1}|} =
\sum_{U\in\mathcal{P}_t} \frac{|\mathcal{B}_U|}{|\mathcal{Q}_U|} \cdot \frac{|\mathcal{Q}_U|}{|\mathcal{Q}_t^{p+1,p+1}|} \\
& = & \sum_{U\in\mathcal{P}_t} \frac{|G_{i_0,U}|}{|\Gamma_U|} \cdot
\frac{|\mathcal{Q}_U|}{|\mathcal{Q}_t^{p+1,p+1}|} \nonumber \\
& < & \frac{\gamma^3}{256} \Big( \sum_{U\in\mathcal{P}_t} \frac{|\mathcal{Q}_U|}{|\mathcal{Q}_t^{p+1,p+1}|} \Big) =
\frac{\gamma^3}{256}. \nonumber
\end{eqnarray}
The proof of Claim \ref{9c10} is completed.
\end{proof}
\begin{claim} \label{9c13}
For every $\ell\in[p+1]$ and every $t\in [k+1]^{<|L_{p+1}|}$ we have that $\dens_{X_{\bl_{p+1}}}(\cup\mathcal{Q}_t^{\ell,p+1})\meg\alpha$.
That is, condition \emph{(C6)} is satisfied.
\end{claim}
\begin{proof}
Let $t\in [k+1]^{<|L_{p+1}|}$ be arbitrary. Assume, first, that $\ell\in[p]$. By the inductive assumptions, we have
\begin{equation}\label{9e57}
\dens_{\Omega_t\times X_{\bl_p}}\Big(\bigcup_{s\in\Omega_t}\{s\}\times\cup\mathcal{Q}_s^{\ell,p}\Big)\meg\alpha.
\end{equation}
On the other hand,
\begin{eqnarray} \label{9e58}
\qv^{-1}_{p+1}\Big(\bigcup_{s\in\Omega_t}\{s\}\times\cup\mathcal{Q}_s^{\ell,p}\Big)
& =& \qv^{-1}_{p+1}\Big(\bigcup_{s\in\Omega_t}\bigcup_{Q\in \mathcal{Q}_s^{\ell,p}} \{s\}\times Q\Big) \\
& =& \bigcup_{s\in\Omega_t}\bigcup_{Q\in \mathcal{Q}_s^{\ell,p}} \qv^{-1}_{p+1}\big(\{s\}\times Q\big) \nonumber \\
& \stackrel{(\ref{9e41})}{=}& \bigcup_{s\in\Omega_t}\bigcup_{Q\in \mathcal{Q}_s^{\ell,p}} \{t\}\times C_t^{s,\ell,Q} \nonumber \\
& \stackrel{(\ref{9e42})}{=} & \{t\}\times \cup\mathcal{Q}_t^{\ell,p+1}.\nonumber
\end{eqnarray}
By (\ref{9e57}), (\ref{9e58}) and Lemma \ref{6l6}, we conclude that
\begin{equation}\label{9e59}
\dens_{X_{\bl_{p+1}}}(\cup\mathcal{Q}_t^{\ell,p+1})= \dens_{\{t\}\times X_{\bl_{p+1}}}\big(\{t\}\times \cup\mathcal{Q}_t^{\ell,p+1}\big)\meg\alpha.
\end{equation}

Next assume that $\ell=p+1$. By Fact \ref{9f5}, the family $\{U(i_0):U\in\mathcal{P}_t\}$ forms a partition of the set $\Omega_t$. Moreover,
as we have already pointed out immediately after \eqref{9e25}, we have $\dens_{X_{\bl_{p}}}(\Gamma_U)\meg\alpha$. Hence,
\begin{equation} \label{9e60}
\dens_{\Omega_t\times X_{\bl_{p}}}\Big(\bigcup_{U\in\mathcal{P}_t}U(i_0)\times\Gamma_U\Big)\meg\alpha.
\end{equation}
Notice that
\begin{eqnarray} \label{9e61}
\qv^{-1}_{p+1} \Big(\bigcup_{U\in\mathcal{P}_t}U(i_0)\times\Gamma_U\Big)
& = & \qv^{-1}_{p+1} \Big(\bigcup_{U\in\mathcal{P}_t}\bigcup_{\bx\in\Gamma_U}U(i_0)\times\{\bx\}\Big) \\
& = & \bigcup_{U\in\mathcal{P}_t}\bigcup_{\bx\in\Gamma_U}\qv^{-1}_{p+1} \big(U(i_0)\times\{\bx\}\big) \nonumber \\
& \stackrel{(\ref{9e38})}{=} & \bigcup_{U\in\mathcal{P}_t}\bigcup_{\bx\in\Gamma_U} \{t\}\times Q^{\bx,U}_t \nonumber\\
& \stackrel{(\ref{9e39})}{=} & \{t\}\times\cup\mathcal{Q}_t^{p+1,p+1}. \nonumber
\end{eqnarray}
Combining \eqref{9e60}, \eqref{9e61} and applying Lemma \ref{6l6} we obtain
\begin{equation} \label{9e62}
\dens_{X_{\bl_{p+1}}}(\cup\mathcal{Q}_t^{p+1,p+1}) =
\dens_{\{t\}\times X_{\bl_{p+1}}}\big( \{t\}\times \cup\mathcal{Q}_t^{p+1,p+1}\big)\meg\alpha
\end{equation}
and the proof is completed.
\end{proof}
\begin{claim} \label{9c14}
For every $\ell\in [p+1]$ and every $t,t'\in [k+1]^{<|L_{p+1}|}$ if $t$ and $t'$ are $(r,k+1)$-equivalent, then
$\mathcal{Q}_t^{\ell,p+1}= \mathcal{Q}_{t'}^{\ell,p+1}$. That is, condition \emph{(C7)} is satisfied.
\end{claim}
\begin{proof}
We fix $t,t'\in [k+1]^{<|L_{p+1}|}$ which are $(r,k+1)$-equivalent. For every $s\in\Omega_t$ let
$Y^t_s=\{x\in X_{L_{p+1}}:\cv_{L_{p+1},V_{p+1}}(t,x)=s\}$. Respectively, for every $s'\in\Omega_{t'}$ let
$Y^{t'}_{s'}=\{x\in X_{L_{p+1}}:\cv_{L_{p+1},V_{p+1}}(t',x)=s'\}$. Let $g_{t,t'}:\Omega_t\to\Omega_{t'}$ be the bijection
obtained by Fact \ref{5f4} and recall that for every $s\in\Omega_t$ we have
\begin{equation} \label{9e63}
Y_s^t=Y^{t'}_{g_{t,t'}(s)}.
\end{equation}
Since $t$ and $t'$ are $(r,k+1)$-equivalent, by Fact \ref{5f4} again, we see that $s$ and $g_{t,t'}(s)$ are also $(r,k+1)$-equivalent
for every $s\in\Omega_t$.

After this preliminary discussion, we are ready for the main argument. First assume that $\ell\in[p]$. For every $s\in\Omega_t$
and every $s'\in\Omega_{t'}$ let
\begin{equation} \label{9e64}
\mathcal{Q}^t_s=\{C_t^{s,\ell,Q}:Q\in\mathcal{Q}_s^{\ell,p}\} \text{ and }
\mathcal{Q}^{t'}_{s'}=\{C_{t'}^{s',\ell,Q'}:Q'\in\mathcal{Q}_{s'}^{\ell,p}\}.
\end{equation}
By (\ref{9e42}), the families $\{\mathcal{Q}^t_s:s\in\Omega_t\}$ and $\{\mathcal{Q}^{t'}_{s'}:s'\in\Omega_{t'}\}$ form partitions
of $\mathcal{Q}_t^{\ell,p+1}$ and $\mathcal{Q}_{t'}^{\ell,p+1}$ respectively. By (\ref{9e40}), for every $s\in\Omega_{t}$
and every $Q\in\mathcal{Q}_s^{\ell,p}$ we have
\begin{equation}\label{9e65}
C_t^{s,\ell,Q} = Q\times Y^t_s.
\end{equation}
Of course, we have the same equality for $t'$, that is, for every $s'\in\Omega_{t'}$ and every $Q'\in\mathcal{Q}_{s'}^{\ell,p}$ it holds that
\begin{equation}\label{9e66}
C_{t'}^{s',\ell,Q'} = Q'\times Y^{t'}_{s'}.
\end{equation}
Moreover, invoking our inductive assumptions, for every $s\in\Omega_t$ we have
\begin{equation}\label{9e67}
\mathcal{Q}^{\ell,p}_s=\mathcal{Q}^{\ell,p}_{g_{t,t'}(s)}.
\end{equation}
Hence,
\begin{eqnarray} \label{9e68}
\mathcal{Q}_s^t & = &  \{C_t^{s,\ell,Q}:Q\in\mathcal{Q}_s^{\ell,p}\} \stackrel{(\ref{9e65})}{=} \{Q\times Y^t_s:Q\in\mathcal{Q}_s^{\ell,p}\} \\
& \stackrel{(\ref{9e63})}{=} & \{Q\times Y^{t'}_{g_{t,t'}(s)}:Q\in\mathcal{Q}_s^{\ell,p}\} \nonumber \\
& \stackrel{(\ref{9e67})}{=} & \{Q\times Y^{t'}_{g_{t,t'}(s)}:Q\in\mathcal{Q}_{g_{t,t'}(s)}^{\ell,p}\}
\stackrel{(\ref{9e66})}{=}\mathcal{Q}^{t'}_{g_{t,t'}(s)}. \nonumber
\end{eqnarray}
Since $g_{t,t'}$ is a bijection we conclude that $\mathcal{Q}^{\ell,p+1}_t=\mathcal{Q}^{\ell,p+1}_{t'}$.

Before we proceed we need, first, to introduce some terminology. Specifically, let $U$ and $U'$ be two Carlson--Simpson trees of $[k+1]^{<\nn}$
of the same dimension and consider the canonical isomorphism $\mathrm{I}_{U,U'}$ associated to the pair $U, U'$ described in \S 2.5. We say that
$U$ and $U'$ are \textit{$(r,k+1)$-equivalent} if for every $s\in U$ we have that $s$ and $\mathrm{I}_{U,U'}(s)$ are $(r,k+1)$-equivalent.

Now assume that $\ell=p+1$ and let $K_t$ and $K_{t'}$ be as in (\ref{9e33}). Our first goal is to define a map $h:K_t\to K_{t'}$ with the
following properties.
\begin{enumerate}
\item[(a)] The map $h$ is a bijection.
\item[(b)] The map $h$ preserves the lexicographical order.
\item[(c)] For every $s\in K_t$ we have that $s$ and $h(s)$ are $(r,k+1)$-equivalent.
\end{enumerate}
If $t=\varnothing$, then $h$ is the identity. Assume that $|t|=|t'|=j$ for some $j\meg1$ and recall that $t_*$ and $t'_*$ are the initial segments
of $t$ and $t'$ respectively of length $j-1$. Let $g_{t_*, t'_*}:\Omega_{t_*}\to\Omega_{t'_*}$ be the map obtained by Fact \ref{5f4}. We define
\begin{equation} \label{9e69}
h\big(\mathrm{I}^{-1}(w)^\con t(j-1)\big)=\mathrm{I}^{-1}(g_{t_*,t'_*}(w))^\con t'(j-1)
\end{equation}
for every $w\in\Omega_{t_*}$. With this choice the aforementioned properties of $h$ follow readily from the properties of $g_{t_*,t'_*}$.

The map $h$ induces a function $f:\mathcal{P}_t\to\mathcal{P}_{t'}$ defined by the rule
\begin{equation} \label{9e70}
f\big(\mathrm{I}(C_s)\big)=\mathrm{I}(C_{h(s)}).
\end{equation}
We isolate, for future use, the following properties of $f$. Their verification is straightforward.
\begin{enumerate}
\item[(d)] The function $f$ is a bijection.
\item[(e)] For every $U,U'\in\mathcal{P}_t$ if $U(0)\lex U'(0)$, then $f(U)(0)\lex f(U')(0)$.
\item[(f)] For every $U\in\mathcal{P}_t$ we have that $U$ and $f(U)$ are $(r,k+1)$-equivalent.
\end{enumerate}
Also observe that for every $U\in\mathcal{P}_t$ we have
\begin{equation}\label{9e71}
\Gamma_U=\Gamma_{f(U)}.
\end{equation}
This follows by Fact \ref{9f3} and property (f) above. More important, however, is the relation of the function $f$ with the map
$g_{t,t'}$. Specifically, for every $U\in\mathcal{P}_t$ we have
\begin{equation}\label{9e72}
\{g_{t,t'}(s): s\in U(i_0)\}=f(U)(i_0).
\end{equation}
To see this notice, first, that for every $s\in K_t$ the set $C_s(i_0)$ is an interval, in the lexicographical order, of $[k+1]^{l}$ for
some $l\in\{0,...,\dim(V_{p+1})\}$ depending only on the length of $t$ (precisely, $l=i_0+(i_0+1)|t|$). Hence, by (\ref{9e34}), for every
$U\in\mathcal{P}_t$ the set $U(i_0)$ is an interval of $V_{p+1}(l)$ for the same $l$. Therefore, equality (\ref{9e71}) follows by
Fact \ref{9f5} and property (e) isolated above.

For every $U\in\mathcal{P}_{t}$ and every $U'\in\mathcal{P}_{t'}$ we set
\begin{equation} \label{9e73}
\mathcal{Q}^t_U=\{Q_t^{\bx,U}:\bx\in \Gamma_U\} \text{ and } \mathcal{Q}^{t'}_{U'}=\{Q_{t'}^{\bx',U'}:\bx'\in \Gamma_{U'}\}.
\end{equation}
By (\ref{9e39}), the families $\{\mathcal{Q}^t_U:U\in\mathcal{P}_t\}$ and $\{\mathcal{Q}^{t'}_{U'}:U'\in\mathcal{P}_{t'}\}$
form partitions of $\mathcal{Q}_t^{p+1,p+1}$ and $\mathcal{Q}_{t'}^{p+1,p+1}$ respectively. By (\ref{9e37}), for every $U\in\mathcal{P}_t$
and $U'\in\mathcal{P}_{t'}$ and every $\bx\in\Gamma_U$ and $\bx'\in\Gamma_{U'}$ we have
\begin{equation} \label{9e74}
Q_t^{\bx,U}=\{\bx\}\times \bigcup_{s\in U(i_0)}Y^t_s \text{ and } Q_{t'}^{\bx',U'}=\{\bx'\}\times \bigcup_{s'\in U'(i_0)}Y^{t'}_{s'}.
\end{equation}
Thus, for every $U\in\mathcal{P}_t$,
\begin{eqnarray} \label{9e75}
\mathcal{Q}^{t}_{U} & \stackrel{(\ref{9e73})}{=} & \{Q_t^{\bx,U}:\bx\in \Gamma_U\}
\stackrel{(\ref{9e74})}{=} \Big\{\{\bx\}\times \bigcup_{s\in U(i_0)}Y^t_s:\bx\in \Gamma_U\Big\} \\
& \stackrel{(\ref{9e71})}{=} & \Big\{\{\bx'\}\times \bigcup_{s\in U(i_0)}Y^t_s:\bx'\in \Gamma_{f(U)}\Big\} \nonumber \\
& \stackrel{(\ref{9e63})}{=} & \Big\{\{\bx'\}\times \bigcup_{s\in U(i_0)}Y^{t'}_{g_{t,t'}(s)}:\bx'\in \Gamma_{f(U)}\Big\} \nonumber \\
& \stackrel{(\ref{9e72})}{=} & \Big\{\{\bx'\}\times \bigcup_{s'\in f(U)(i_0)}Y^{t'}_{s'}:\bx'\in \Gamma_{f(U)}\Big\}
=\mathcal{Q}^{t'}_{f(U)}. \nonumber
\end{eqnarray}
Since $f$ is a bijection we conclude that $\mathcal{Q}_t^{p+1,p+1}=\mathcal{Q}_{t'}^{p+1,p+1}$. The proof of Claim \ref{9c14} is thus completed.
\end{proof}
By Claims \ref{9c7} up to \ref{9c14}, the pair $(V_{p+1},L_{p+1})$ and the families $\mathcal{Q}_t^{\ell,p+1}$ constructed in Steps 1, 2 and 3,
satisfy all required conditions. This completes the recursive selection, and as we have already indicated, the proof of Lemma \ref{9l1}
is also completed.
\medskip

\noindent 9.4. \textbf{Consequences.} In this subsection we will isolate what we get by iterating Lemma \ref{9l1}. The resulting statement
together with Corollary \ref{8c6} form the basis of the proof of Theorem B. We proceed to the details.

Let $k\in\nn$ with $k\meg 2$ and assume that for every integer $l\meg 1$ and every $0<\beta\mik 1$ the number $\dcs(k,l,\beta)$ has been
defined. We define $H:\nn\times (0,1]\to \nn$ by $H(0,\gamma)=0$ and
\begin{equation} \label{9e76}
H(m,\gamma) = \reg(k+1,N_{p_0}+1,2,\gamma^2/2)
\end{equation}
if $m\meg 1$, where $p_0$ and $N_{p_0}$ are defined in \eqref{9e6} and \eqref{9e7} respectively for the parameters $m$ and $\gamma$.
Next for every $n\in\{0,...,k\}$ we define $H^{(n)}:\nn\times (0,1]\to \nn$ recursively by the rule $H^{(0)}(m,\gamma)=m$ and
\begin{equation} \label{9e77}
H^{(n+1)}(m,\gamma) = H\big( H^{(n)}(m,\gamma),\gamma\big).
\end{equation}
Finally, for every $0<\gamma\mik 1$ let
\begin{equation} \label{9e78}
\xi=\xi(\gamma)=\frac{\gamma^{3^k}}{\big(2^{1/2}\cdot 32\big)^{3^k-1}}.
\end{equation}
We have the following corollary. It is an immediate consequence of Lemma \ref{9l1}.
\begin{cor} \label{9c15}
Let $k\in\nn$ with $k\meg 2$ and assume that for every integer $l\meg 1$ and every $0<\beta\mik 1$ the number $\dcs(k,l,\beta)$ has been defined.

Let $0<\gamma,\delta\mik 1$. Also let $V$ be a Carlson--Simpson tree of $[k+1]^{<\nn}$ and $I$ be a nonempty subset of $\{0,...,\dim(V)\}$.
Assume that we are given subsets  $A,D_1,...,D_k$ of $[k+1]^{<\nn}$ with the following properties.
\begin{enumerate}
\item[(a)] For every $r\in [k]$ the set $D_r$ is $(r,k+1)$-insensitive in $V$.
\item[(b)] We have $\dens_{V(i)}(D_1\cap... \cap D_k\cap A)\meg(\delta+\gamma)\dens_{V(i)}(D_1\cap ...\cap D_k)$
and $\dens_{V(i)}(D_1\cap ...\cap D_k)\meg \gamma$ for every $i\in I$.
\end{enumerate}
Finally, let $m\in\nn$ with $m\meg 1$ and suppose that
\begin{equation} \label{9e79}
|I|\meg H^{(k)}(m,\xi)
\end{equation}
where $H^{(k)}$ and $\xi$ are defined in \eqref{9e77} and \eqref{9e78} respectively for the parameters $m$ and $\gamma$.
Then there exist a Carlson--Simpson subtree $W$ of $V$ and a subset $I'$ of $\{0,...,\dim(W)\}$ of cardinality $m$ such that
\begin{equation} \label{9e80}
\dens_{W(i)}(A) \meg \delta+\xi
\end{equation}
for every $i\in I'$.
\end{cor}


\section{Proof of Theorem B}

\numberwithin{equation}{section}

In this section we will complete the proof of Theorem B following the inductive scheme outlined in \S 8.1. Notice first that the numbers
$\dcs(2,1,\delta)$ are estimated in Proposition \ref{7p1}. It is then easy to see that, by induction on $m$ and Corollary \ref{7c6},
we may also estimate the numbers $\dcs(2,m,\delta)$.

Now we argue for the general inductive step. So let $k\in\nn$ with $k\meg 2$ and assume that for every integer $l\meg 1$ and every
$0<\beta\mik 1$ the number $\dcs(k,l,\beta)$ has been defined. We fix $0<\delta\mik 1$. Let $\eta_1$ be as in \eqref{8e33}. Recall that
\begin{equation} \label{10e1}
\eta_1= \frac{\delta^2}{120k\cdot |\subtr_1\big([k]^{<\Lambda}\big)|}
\end{equation}
where $\Lambda=\lceil 8\delta^{-1}\dcs(k,1,\delta/8) \rceil$. We set
\begin{equation} \label{10e2}
\varrho=\xi(\eta_1^2/2)\stackrel{\eqref{9e78}}{=}\frac{(\eta_1^2/2)^{3^k}}{\big(2^{1/2}\cdot 32\big)^{3^k-1}}
\end{equation}
and we define $F_{\delta}:\nn\to\nn$ by the rule
\begin{equation} \label{10e3}
F_{\delta}(m)= G_1\Big( \big\lceil \eta_1^{-4}(k+1)k \cdot H^{(k)}(m,\varrho) \big\rceil, \eta_1^2/2 \Big)
\end{equation}
where $G_1$ and $H^{(k)}(m,\varrho)$ are as in \eqref{8e34} and \eqref{9e77} respectively. The following proposition
is the heart of the density increment strategy. It follows immediately by Corollary \ref{8c6} and Corollary \ref{9c15}.
\begin{prop} \label{10p1}
Let $k\in\nn$ with $k\meg 2$ and assume that for every integer $l\meg 1$ and every $0<\beta\mik 1$ the number $\dcs(k,l,\beta)$ has been defined.

Let $0<\delta\mik 1$ and $L$ be a nonempty finite subset of $\nn$. Also let $A\subseteq [k+1]^{<\nn}$ such that 
$|A\cap [k+1]^l|\meg \delta (k+1)^l$ for every $l\in L$ and assume that $A$ contains no Carlson--Simpson line of $[k+1]^{<\nn}$.
Finally, let $m\in\nn$ with $m\meg 1$ and suppose that $|L|\meg F_{\delta}(m)$ where $F_{\delta}$ is as in \eqref{10e3}. Then there exist
a Carlson--Simpson tree $W$ of $[k+1]^{<\nn}$ and a subset $I$ of $\{0,...,\dim(W)\}$ of cardinality $m$ such that
$\dens_{W(i)}(A) \meg \delta+\varrho$ for every $i\in I$ where $\varrho$ is as in \eqref{10e2}.
\end{prop}
Using Proposition \ref{10p1} the numbers $\dcs(k+1,1,\delta)$ can, of course, be estimated easily. In particular, we have the following corollary.
\begin{cor} \label{10c2}
Let $k\in\nn$ with $k\meg 2$ and assume that for every integer $l\meg 1$ and every $0<\beta\mik 1$ the number $\dcs(k,l,\beta)$ has been defined.
Then for every $0<\delta\mik 1$ we have
\begin{equation} \label{10e4}
\dcs(k+1,1,\delta)\mik F_\delta^{(\lceil \varrho^{-1}\rceil)}(1).
\end{equation}
\end{cor}
Finally, just as in the case ``$k=2$'', the numbers $\dcs(k+1,m,\delta)$ can be estimated using Corollary \ref{10c2} and
Corollary \ref{7c6}. This completes the proof of the general inductive step, and so, the entire proof of Theorem B is completed.


\section{Consequences}

\numberwithin{equation}{section}

Our goal in this section is to prove several consequences of Theorem B. These include Theorem A and Theorem C stated in the introduction,
as well as, an appropriate finite version of Theorem C. To state this finite version we need, first, to introduce some terminology.

Recall that, given two sequences $(p_n)$ and $(w_n)$ of variable words over $k$, we say that $(w_n)$ is of \textit{pattern} $(p_n)$ if $p_n$
is an initial segment of $w_n$ for every $n\in\nn$. This notion can, of course, be extended to finite sequences of the same length.
Specifically, given two finite sequences $(p_n)_{n=0}^{m-1}$ and $(w_n)_{n=0}^{m-1}$ of variable words over $k$, we say that $(w_n)_{n=0}^{m-1}$
is of \textit{pattern} $(p_n)_{n=0}^{m-1}$ if $p_n$ is an initial segment of $w_n$ for every $n\in\{0,...,m-1\}$. In particular, if $p$ and $w$
are variable words over $k$, then $w$ is of pattern $p$ if $p$ is an initial segment of $w$. We have the following theorem.
\begin{thm} \label{11t1}
For every integer $k\meg 2$, every nonempty finite sequence $(\tau_n)_{n=0}^{m-1}$ of positive integers and every $0<\delta\mik 1$ there
exists an integer $N$ with the following property. If $(p_n)_{n=0}^{m-1}$ is a finite sequence of variable words over $k$ such that the length
of $p_n$ is $\tau_n$ for every $n\in\{0,...,m-1\}$, $L$ is a finite subset of $\nn$ of cardinality at least $N$ and $A$ is a subset of
$[k]^{<\nn}$ satisfying $\dens_{[k]^\ell}(A)\meg \delta$ for every $\ell\in L$, then there exist a word $c$ over $k$ and a finite sequence
$(w_n)_{n=0}^{m-1}$ of variable words over $k$ of pattern $(p_n)_{n=0}^{m-1}$ such that the set
\begin{equation} \label{11e1}
\{c\}\cup \big\{c^\con w_0(a_0)^\con...^\con w_n(a_n): n\in\{0,...,m-1\}\;\text{and}\;a_0,...,a_n\in[k]\big\}
\end{equation}
is contained in $A$. The least integer $N$ with the above property will be denoted by $\mathrm{DP}(k,(\tau_n)_{n=0}^{m-1},\delta)$.
\end{thm}
The proof of Theorem \ref{11t1} will be given in \S 11.4. All necessary tools (beside, of course, Theorem B) are developed in the previous
subsections. The corresponding infinite versions -- that is, Theorem A and Theorem C -- will be proved in \S 11.5. Finally, in \S 11.6 we
discuss how one can derive the density Hales--Jewett Theorem and the density Halpern--L\"{a}uchli Theorem from Theorem B and Theorem A respectively.
\medskip

\noindent 11.1. \textbf{Sparse sets and regularity.} We begin with the following definition.
\begin{defn} \label{11d2}
Let $\tau\in\nn$ with $\tau\meg1$ and $L$ be a nonempty subset of $\nn$. We say that $L$ is \emph{$\tau$-sparse} if for every $l,l'\in L$ with
$l\neq l'$ we have $|l'-l|\meg \tau$. If $L$ is a $\tau$-sparse subset of $\nn$, then we define the \emph{$\tau$-extension} of $L$ to be the set
\begin{equation} \label{11e2}
(L)_{\tau}=L+\{0,...,\tau-1\}=\{l+n:l\in L\;\text{and}\;0\mik n\mik \tau-1\}.
\end{equation}
\end{defn}
Every nonempty subset of $\nn$ is $1$-sparse and coincides with its $1$-extension. Also notice that every subset of $\nn$ of cardinality
at least $\tau (m-1)+1$ contains a $\tau$-sparse subset of  cardinality $m$. Finally observe that the notion of $\tau$-sparseness is hereditary,
that is, every nonempty subset of a $\tau$-sparse set is $\tau$-sparse.

Much of our interest in sparse sets is related to the following generalized version of Definition \ref{3d1}.
\begin{defn} \label{11d3}
Let $k\in\nn$ with $k\meg2$ and $\mathcal{F}$ be a family of subsets of $[k]^{<\nn}$. Also let $0<\ee\mik1$, $\tau\in\nn$ with $\tau\meg1$ and
$L$ be a $\tau$-sparse finite subset of $\nn$. The family $\mathcal{F}$ will be called \emph{$(\ee,\tau,L)$-regular} if for every $n\in L$, every
$I\subseteq\{l\in L:l<n\}$ and every $y\in [k]^{(I)_{\tau}}$ we have
\begin{equation}\label{11e3}
|\dens\big(\{z\in[k]^{\{m\in\nn:m<n\}\setminus (I)_{\tau}}:(y,z)\in A\}\big)-\dens(A\cap[k]^n)|\mik\ee.
\end{equation}
\end{defn}
We have the following analogue of Lemma \ref{3l2}. It is the main result of this subsection.
\begin{lem} \label{11l4}
Let $0<\ee\mik 1$ and $k,\ell,q,\tau\in\nn$ with $k\meg 2$ and $\ell,q,\tau\meg 1$. Then there exists an integer $n$ with the following property.
If $N$ is a finite $\tau$-sparse subset of $\nn$ with $|N|\meg n$ and $\mathcal{F}$ is a family of subsets of $[k]^{<\nn}$ with $|\mathcal{F}|=q$,
then there exists a subset $L$ of $N$ with $|L|=\ell$ such that $\mathcal{F}$ is $(\ee,\tau,L)$-regular. The least integer $n$ with this property
will be denoted by $\reg_\tau(k,\ell,q,\ee)$.
\end{lem}
The proof of Lemma \ref{11l4} is similar to the proof of Lemma \ref{3l2} and is based on the following consequence of Sublemma \ref{3sbl7}.
\begin{cor} \label{11c5}
Let $k,m,\tau,q\in\nn$ with $k\meg2$ and $\tau,q\meg1$. Let $0<\ee< k^{-\tau(m+1)}$ and $N$ be a finite $\tau$-sparse subset of $\nn$ with
\begin{equation}\label{11e4}
|N|\meg (q\lfloor16\ee^{-4}\rfloor+1)(m+1)+1.
\end{equation}
Finally, let $\mathcal{F}$ be a family of subsets of $[k]^{\max(N)}$ with $|\mathcal{F}|=q$. Then there exists a subinterval $M$ of
$N\setminus\{\max(N)\}$ with $|M|=m$ such that for every $A\in\mathcal{F}$, every subset $I$ of $(M)_\tau$ and every $y\in[k]^I$ we have
$|\dens(A_y)-\dens(A)|\mik \ee$.
\end{cor}
\begin{proof}
We set $N'=N\setminus\{\max(N)\}$. Observe that
\begin{equation} \label{11e5}
|(N')_\tau\cup\{\max(N)\}|=\tau(N-1)+1\meg \tau(m+1)(q\lfloor16\ee^{-4}\rfloor+1)+1.
\end{equation}
By Sublemma \ref{3sbl7}, there exists a subinterval $M'$ of $(N')_\tau$ with $|M'|=\tau(m+1)$ such that for every
$A\in\mathcal{F}$, every subset $I$ of $M'$ and every $y\in[k]^I$ we have
 \begin{equation}\label{11e6}
|\dens(A_y)-\dens(A)|\mik \ee.
\end{equation}
Since $M'$ is subinterval of $(N')_\tau$ of cardinality $\tau(m+1)$, it is possible to select a subinterval $M$ of $N'$ of cardinality $m$
such that $(M)_\tau\subseteq M'$. Clearly $M$ is as desired.
\end{proof}
We are ready to proceed to the proof of Lemma \ref{11l4}.
\begin{proof}[Proof of Lemma \ref{11l4}]
We set $\varrho=\min\{\ee,k^{-\tau(\ell+1)}/2\}$ and we define $\tilde{F}:\nn\to \nn$ by
\begin{equation} \label{11e7}
\tilde{F}(m)=(q\lfloor16\varrho^{-4}\rfloor+1)(m+1)+1.
\end{equation}
Arguing precisely as in the proof of Lemma \ref{3l2} and using Corollary \ref{11c5} instead of Sublemma \ref{3sbl7}, we see that
\begin{equation} \label{1138}
\reg_\tau(k,\ell,q,\ee)\mik \tilde{F}^{(\ell)}(0)
\end{equation}
and the proof is completed.
\end{proof}
\noindent 11.2. \textbf{The $(p,L)$-restriction of $[k]^{<\nn}$.} We are about to introduce a family of subsets of $[k]^{<\nn}$
which are the analogues of Carlson--Simpson trees in the context of variable words of a fixed pattern $p$.
\begin{defn} \label{11d6}
Let $k,\tau\in\nn$ with $k\meg2$ and $\tau\meg 1$. Let $p$ be a variable word over $k$ of length $\tau$ and $L=\{l_0<...<l_{|L|-1}\}$
a $\tau$-sparse finite subset of $\nn$. Recursively, for every $i\in\{0,...,|L|-1\}$ we define a subset $R_{p,L}(i)$ of $[k]^{l_i}$
as follows. We set $R_{p,L}(0)=[k]^{l_0}$. Assume that $R_{p,L}(i)$ has been defined for some $i\in\{0,...,|L|-2\}$. Then we set
\begin{equation}\label{11e9}
R_{p,L}(i+1)=\big\{ x^{\con}p(a)^{\con} y: x\in R_{p,L}(i), a\in [k] \text{ and } y\in[k]^{l_{i+1}-l_i-\tau}\big\}.
\end{equation}
We define the \emph{$(p,L)$-restriction} of $[k]^{<\nn}$ to be the set
\begin{equation}\label{11e10}
R_{p,L}=\bigcup_{i=0}^{|L|-1} R_{p,L}(i).
\end{equation}
\end{defn}
Notice that the $(p,L)$-restriction $R_{p,L}$ is a rather ``thin'' subset of $[k]^{<\nn}$. So, if we are given a subset $A$ of $[k]^{<\nn}$
it is likely that the density of $A$ inside $R_{p,L}$ will be negligible. This phenomenon, however, does not occur as long as $A$ is sufficiently
regular. In particular, we have the following lemma. Its proof is a straightforward consequence of the relevant definitions.
\begin{lem} \label{11l7}
Let $k,\tau\in\nn$ with $k\meg2$ and $\tau\meg 1$, $p$ be a variable word over $k$ of length $\tau$ and $L=\{l_0<...<l_{|L|-1}\}$ a $\tau$-sparse
finite subset of $\nn$. Let  $0<\ee\mik 1$ and $\mathcal{F}$ be a family of subsets of $[k]^{<\nn}$ which is $(\ee,\tau,L)$-regular. Then
\begin{equation} \label{11e11}
| \dens_{R_{p,L}(i)}(A)-\dens_{[k]^{l_i}}(A)|\mik \ee
\end{equation}
for every $i\in\{0,...,|L|-1\}$ and every $A\in\mathcal{F}$.
\end{lem}
We will need to parameterize the $(p,L)$-restriction $R_{p,L}$ of $[k]^{<\nn}$ in a ``canonical'' way. This is, essentially, the content of the following definition.
\begin{defn} \label{11d8}
Let $k,\tau\in\nn$ with $k\meg2$ and $\tau\meg 1$. Let $p$ be a variable word over $k$ of length $\tau$ and $L=\{l_0<...<l_{|L|-1}\}$
a $\tau$-sparse finite subset of $\nn$. For every $i\in\{0,...,|L|-1\}$ we set $l_i'=l_i-i\tau+i$ and we define
\begin{equation} \label{11e12}
L^{(\tau)}=\big\{l'_i: i=0,...,|L|-1\big\}.
\end{equation}
Recursively we define a bijection
\begin{equation} \label{11e13}
\Phi_{p,L}:\bigcup_{l\in L^{(\tau)}}[k]^l\to R_{p,L}
\end{equation}
as follows. For every $x\in [k]^{l_0'}=[k]^{l_0}$ we set $\Phi_{p,L}(x)=x$. Let $i\in \{0,...,|L|-2\}$ and assume that
$\Phi_{p,L}(y)$ has been defined for every $y\in [k]^{l_i'}$. Then for every $x\in [k]^{l'_{i+1}}$ we define
\begin{equation} \label{11e14}
\Phi_{p,L}(x)=\Phi_{p,L}(x_1)^\con p(a_x)^\con x_2
\end{equation}
where $x_1=\big(x(0),..., x(l'_i-1)\big)$, $a_x=x(l_{i}')$ and $x_2=\big(x(l'_i+1),...,x(l'_{i+1}-1)\big)$.
\end{defn}
We isolate, for future use, some elementary properties of the map $\Phi_{p,L}$.
\begin{lem} \label{11l9}
Let $k,\tau\in\nn$ with $k\meg2$ and $\tau\meg 1$. Let $p$ be a variable word over $k$ of length $\tau$ and $L=\{l_0<...<l_{|L|-1}\}$
a $\tau$-sparse finite subset of $\nn$. Let $L^{(r)}=\{l'_i: i=0,...,|L|-1\}$ be as in \eqref{11e12}. Then the following are satisfied.
\begin{enumerate}
\item[(i)] For every $i\in\{0,...,|L|-1\}$ we have $\Phi_{p,L}\big([k]^{l_i'}\big)=R_{p,L}(i)$.
\item[(ii)] For every Carlson--Simpson line $W$ of $[k]^{<\nn}$ with $L(W)\subseteq L^{(\tau)}$ its image $\Phi_{p,L}(W)$ under the map
$\Phi_{p,L}$ is of the form $\{c\}\cup\{c^\con w(a):a\in[k]\}$ where $c$ is a word over $k$ and $w$ is a variable word over $k$ of pattern $p$.
\item[(iii)] If $\mathcal{F}$ is a family of subsets of $[k]^{<\nn}$ which is $\mathcal{F}$ is $(\ee,\tau,L)$-regular for some $0<\ee\mik 1$,
then for every $i\in\{0,...,|L|-1\}$ and every $A\in\mathcal{F}$ we have
\begin{equation} \label{11e15}
|\dens_{[k]^{l_i'}}\big(\Phi_{p,L}^{-1}(A)\big)-\dens_{[k]^{l_i}}(A)|\mik \ee.
\end{equation}
\end{enumerate}
\end{lem}
Parts (i) and (ii) of Lemma \ref{11l9} are immediate consequences of Definition \ref{11d8}. Part (iii) follows easily by Lemma \ref{11l7}.
We leave the details to the reader.
\medskip

\noindent 11.3. \textbf{Preliminary lemmas.} As in \eqref{7e15} and \eqref{7e16} for every $k\in\nn$ with $k\meg 2$ and every
$0<\delta\mik 1$ we set
\begin{equation} \label{11e16}
\Lambda(k,1,\delta) = \lceil\delta^{-1}\mathrm{DCS}(k,1,\delta)\rceil
\end{equation}
and
\begin{equation} \label{11e17}
\Theta(k,1,\delta) = \frac{2\delta}{|\subtr_1\big( [k]^{<\Lambda(k,1,\delta)}\big)|}.
\end{equation}
We have the following analogue of Lemma \ref{7l9}.
\begin{lem} \label{11l10}
Let $k\in\nn$ with $k\meg 2$ and $0<\rho,\gamma\mik 1$. Also let $\tau\in\nn$ with $\tau\meg1$ and $M$ be a finite subset of $\nn$ with
\begin{equation}\label{11e18}
|M| \meg \tau\cdot\reg_\tau\big(k,\Lambda(k,1,\rho\gamma/8),1,\rho/2\big)
\end{equation}
where $\Lambda(k,1,\rho\gamma/8)$ is as in \eqref{11e16}. Finally let $B\subseteq [k]^{<\nn}$ with $\dens_{[k]^n}(B)\meg \rho$ for every
$n\in M$. If $\{A_t: t\in B\}$ is a family of measurable events in a probability space $(\Omega,\Sigma,\mu)$ satisfying $\mu(A_t)\meg\gamma$
for every $t\in B$, then for every variable word $p$ over $k$ of length $\tau$ there exist a word $c$ over $k$ and a variable word $w$ over $k$
of pattern $p$ such that, setting $V=\{c\}\cup\{c^\con w(a):a\in[k]\}$, we have $V\subseteq B$ and
\begin{equation} \label{11e19}
\mu \Big( \bigcap_{t\in V} A_t\Big) \meg \Theta(k,1,\rho\gamma/8)
\end{equation}
where $\Theta(k,1,\rho\gamma/8)$ is as in \eqref{11e17}.
\end{lem}
\begin{proof}
By \eqref{11e18}, we may select a $\tau$-sparse subset $N$ of $M$ with
\begin{equation}\label{11e20}
|N| \meg \reg_\tau\big(k,\Lambda(k,1,\rho\gamma/8),1,\rho/2\big).
\end{equation}
By Lemma \ref{11l4}, there exists a subset $L$ of $N$ with $|L|=\Lambda(k,1,\rho\gamma/8)$ such that the family $\mathcal{F}:=\{B\}$
is $(\rho/2,\tau,L)$-regular.

Fix a variable word $p$ over $k$ of length $\tau$. Let $L^{(\tau)}$ and $\Phi_{p,L}$ be as in Definition \ref{11d8}. We set
$B'=\Phi_{p,L}^{-1}(B)$. Since the singleton $\{B\}$ is $(\rho/2,\tau,L)$-regular and $\dens_{[k]^{l}}(B)\meg \rho$ for every $l\in L$,
by part (iii) of Lemma \ref{11l9}, we get that $\dens_{[k]^{l'}}(B')\meg \rho/2$ for every $l'\in L^{(\tau)}$. Next for every $t'\in B'$
let $C_{t'}=A_{\Phi_{p,L}(t')}$. By our assumptions we have $\mu(A_t)\meg \gamma$ for every $t\in B$, and so, $\mu(C_{t'})\meg\gamma$
for every $t'\in B'$. Finally notice that $|L^{(\tau)}|=|L|=\Lambda(k,1,\rho\gamma/8)$. By the previous discussion, we may apply
Lemma \ref{7l9} and we obtain a Carlson--Simpson line $V'\subseteq B'$ with $L(V')\subseteq L^{(\tau)}$ and such that
\begin{equation} \label{11e21}
\mu\Big( \bigcap_{t'\in V'} C_{t'}\Big) \meg \Theta(k,1,\rho\gamma/8).
\end{equation}
We set $V=\Phi_{p,L}(V')$. By part (ii) of Lemma \ref{11l9} and \eqref{11e21}, we see that $V$ is as desired. The proof is completed.
\end{proof}
To state the next result we need to introduce some numerical invariants. Specifically, for every $k,\tau\in\nn$ with $k\meg 2$
and $\tau\meg 1$ and every $0<\delta\mik 1$ we set
\begin{equation} \label{11e22}
\Lambda_P=\Lambda_P(k,\delta)=\Lambda(k,1,\delta^2/32) \text{ and } \Theta_P=\Theta_P(k,\delta)=\Theta(k,1,\delta^2/32)
\end{equation}
and we define $h_{\tau,\delta}:\nn\to\nn$ by the rule
\begin{equation} \label{11e23}
h_{\tau,\delta}(n)=\tau\cdot\reg_\tau(k,\Lambda_P,1,\delta/4)+\lceil 2\Theta_P^{-1}\cdot n\rceil.
\end{equation}
The following proposition corresponds to Proposition \ref{7p5}.
\begin{prop} \label{11p11}
Let $k\in\nn$ with $k\meg 2$ and $0<\delta\mik 1$ and define $\Lambda_P$ and $\Theta_P$ as in \eqref{11e22}. Also let $\tau\in\nn$ with
$\tau\meg 1$, $L$ be a nonempty finite subset of $\nn$ and $A$ be a subset of $[k]^{<\nn}$ such that $\dens_{[k]^l}(A)\meg \delta$ for
every $l\in L$. Finally let $n\in\nn$ with $n\meg 1$ and assume that $|L|\meg h_{\tau,\delta}(n)$ where $h_{\tau,\delta}$ is as in
\eqref{11e23}. Then, setting $L_0$ to be the set of the first $\Lambda_P$ elements of $L$, we have that either
\begin{enumerate}
\item[(i)] there exist a subset $L'$ of $L\setminus L_0$ with $|L'|\meg n$ and a word $t_0\in [k]^{\ell_0}$ for some $\ell_0\in L_0$ such that
\begin{equation} \label{11e24}
\dens_{[k]^{\ell-\ell_0}} \big( \{s\in[k]^{<\nn}: t_0^{\con}s\in A\}\big) \meg \delta+\delta^2/8
\end{equation}
for every $\ell\in L'$, or
\item[(ii)] for every variable word $p$ over $k$ of length $\tau$ there exist a word $c$ over $k$, a variable word $w$ over $k$ of pattern $p$
and a subset $L''$ of $L\setminus L_0$ with $|L''|\meg n$ such that the following are satisfied.
\begin{enumerate}
\item[(a)] The set $V=\{c\}\cup\{c^\con w(a):a\in[k]\}$ is contained in $\bigcup_{\ell\in L_0}A\cap [k]^{\ell}$.
\item[(b)] Setting $V(1)=\{c^\con w(a):a\in [k]\}$ and $\ell_1$ the unique integer with $V(1)\subseteq [k]^{\ell_1}$,
for every $\ell\in L''$ we have
\begin{equation} \label{11e25}
\dens_{[k]^{\ell-\ell_1}}\big( \{s\in[k]^{<\nn}: t^{\con}s\in A \text{ for every } t\in V(1)\}\big) \meg \Theta_P/2.
\end{equation}
\end{enumerate}
\end{enumerate}
\end{prop}
The proof of Proposition \ref{11p11} is identical to the proof of Proposition \ref{7p5} using Lemma \ref{11l10} instead of Lemma \ref{7l9}.
The details are left to the reader.

We close this subsection with the following consequence of Proposition \ref{11p11}. It is the main tool for the proof of Theorem C.
\begin{cor} \label{11c12}
Let $k\in\nn$ with $k\meg 2$ and $0<\delta\mik 1$. Also let $L$ be an infinite subset of $\nn$ and $A$ be a subset of $[k]^{<\nn}$ such that
$\dens_{[k]^l}(A)\meg \delta$ for every $l\in L$. Then for every variable word $p$ over $k$ there exist a word $c$ over $k$, a variable word
$w$ over $k$ of pattern $p$ and an infinite subset $L'$ of $L$ with the following properties.
\begin{enumerate}
\item [(i)]  The set $V=\{c\}\cup\{c^\con w(a):a\in[k]\}$ is contained in $\bigcup_{\ell\in L_0} A\cap [k]^{\ell}$ where
$L_0=\{\ell\in L: \ell<\min (L')\}$.
\item [(ii)] Setting $V(1)=\{c^\con w(a):a\in [k]\}$ and $\ell_1$ the unique integer such that $V(1)\subseteq [k]^{\ell_1}$,
for every $\ell\in L'$ we have
\begin{equation} \label{11e26}
\dens_{[k]^{\ell-\ell_1}}\big( \{s\in[k]^{<\nn}: t^{\con}s\in A \text{ for every } t\in V(1)\}\big) \meg 2^{-1} \Theta_P(k,\delta/2)
\end{equation}
where $\Theta_P(k,\delta/2)$ is as in \eqref{11e22}.
\end{enumerate}
\end{cor}
\begin{proof}
For every $t\in [k]^{<\nn}$ let $A_t=\{s\in [k]^{<\nn}:t^\con s\in A\}$ and define
\begin{equation} \label{11e27}
\delta_t=\limsup_{\ell\in L}\dens_{[k]^{\ell-|t|}}(A_t).
\end{equation}
We set $\delta^*=\sup_{t\in [k]^{<\nn}}\delta_t$ and we notice that $\delta\mik \delta^*\mik 1$. Hence, we may select $0<\delta_0\mik 1$,
$t_0\in[k]^{<\nn}$ and an infinite subset $M$ of $L$ with $\min(M)>|t_0|$ such that
\begin{equation} \label{11e28}
\delta/2<\delta_0<\delta^*<\delta_0+\delta_0^2/8
\end{equation}
and
\begin{equation}\label{11e29}
\delta_0<\dens_{[k]^{\ell-|t_0|}}(A_{t_0})
\end{equation}
for every $\ell\in M$.

Fix a variable word $p$ over $k$ and denote by $\tau$ its length. Let $M_0$ be the initial segment of $M$ of cardinality
\begin{equation} \label{11e30}
|M_0|=\tau\cdot\reg_\tau(k,\Lambda_P(k,\delta_0),1,\delta/4).
\end{equation}
By the definition of $\delta^*$ and \eqref{11e28}, there exists $q\in M\setminus M_0$ such that for every $s\in \bigcup_{\ell\in M_0}
[k]^{\ell-|t_0|}$ and every $\ell\in M$ with $\ell\meg q$ we have
\begin{equation}\label{11e31}
\dens_{[k]^{\ell-|t_0^{\con}s|}}(A_{t_0^{\con}s}) < \delta_0+\delta_0^2/8.
\end{equation}
We set $d=\lceil 2\Theta_P(k,\delta_0)^{-1}\rceil$ and we select a sequence $(E_n)$ of pairwise disjoint subsets of $\{m\in M:m\meg q\}$
such that $|E_n|=d$ for every $n\in\nn$.

Let $n\in\nn$ be arbitrary. We set $F_n=M_0\cup E_n$ and we observe that
\begin{equation}\label{11e32}
|F_n|=|M_0|+|E_n| \stackrel{\eqref{11e30}}{=}\tau\cdot\reg_\tau(k,\Lambda_P(k,\delta_0),1,\delta/4) + d
\stackrel{\eqref{11e23}}{=} h_{\tau,\delta_0}(1).
\end{equation}
By \eqref{11e32}, we may apply Proposition \ref{11p11}. Notice that the first alternative of Proposition \ref{11p11} contradicts
\eqref{11e31}. Hence, there exist a word $c_n$ over $k$, a variable word $w_n$ over $k$ of pattern $p$ and $m_n\in E_n$ such that
\begin{equation} \label{11e33}
V_n=\{ c_n\}\cup\{c_n^\con w_n(a):a\in[k]\}\subseteq \bigcup_{\ell\in M_0} A_{t_0}\cap [k]^{\ell-|t_0|}
\end{equation}
and, setting $Q_n=\{s\in[k]^{<\nn}: t^{\con}s\in A_{t_0} \text{ for every } t\in V_n(1)\}$,
\begin{equation} \label{11e34}
\dens_{[k]^{m_n-\ell_n-|t_0|}}(Q_n) \meg 2^{-1}\Theta_P(k,\delta_0)  \stackrel{\eqref{11e28}}{\meg} 2^{-1} \Theta_P(k,\delta/2)
\end{equation}
where $V_n(1)=\{c_n^\con w_n(a):a\in [k]\}$ and $\ell_n$ is the unique integer with $V_n(1)\subseteq [k]^{\ell_n}$.

By the classical pigeonhole principle, there exists an infinite subset $N$ of $\nn$, a word $c'$ over $k$ and a variable word $w$ over $k$
of pattern $p$ such that $c_n=c'$ and $w_n=w$ for every $n\in N$. We set $L'= \{m_n: n\in N\}$ and $c=t_0^\con c'$. Using \eqref{11e33} and
\eqref{11e34} it is easy to see that $L', c$ and $w$ are as desired.
\end{proof}
\noindent 11.4. \textbf{Proof of Theorem \ref{11t1}.} The proof proceeds by induction on $m$. For ``$m=1$'' we notice that
\begin{equation} \label{11e35}
\mathrm{DP}(k,\tau,\delta)\mik \tau\cdot\reg_\tau(\tau,\Lambda(k,1,\delta/8),1,\delta/2)
\end{equation}
for every $0<\delta\mik1$ and every $k,\tau\in\nn$ with $k\meg2$ and $\tau\meg 1$. Indeed, let $M$ be a finite subset of $\nn$ with
$|M|\meg \tau \cdot \reg_\tau(k,\Lambda(k,1,\delta/8),1,\delta/2)$ and $A$ be a subset of $[k]^{<\nn}$ such that $\dens_{[k]^n}(A)\meg\delta$
for every $n\in M$. Let $(\Omega,\Sigma,\mu)$ be an arbitrary probability space and set $A_t=\Omega$ for every $t\in A$. By Lemma \ref{11l10}
applied for ``$\rho=\delta$'', ``$\gamma=1$'' and ``$B=A$'', we see that \eqref{11e35} is satisfied.

Let $m\in\nn$ with $m\meg 1$ and assume that for every integer $k\meg 2$, every $0<\beta\mik 1$ and every finite sequence $(\sigma_n)_{n=0}^{m-1}$
of positive integers the number $\mathrm{DP}(k,(\sigma_n)_{n=0}^{m-1},\beta)$ has been defined. Let $k\meg 2$, $0<\delta\mik1$ and
$\tau_0,...,\tau_m\in\nn$ with $\tau_0,...,\tau_m\meg1$ be arbitrary. We set $\tau_n'=\tau_{n+1}$ for every $n\in\{0,...,m-1\}$ and
\begin{equation} \label{11e36}
N_0=\mathrm{DP}(k,(\tau'_n)_{n=0}^{m-1},\Theta_P/2).
\end{equation}
We claim that
\begin{equation} \label{11e37}
\mathrm{DP}(k,(\tau_n)_{n=0}^m,\delta)\mik h_{\tau_0,\delta}^{(\lceil8\delta^{-2}\rceil)}(N_0).
\end{equation}
Clearly this will finish the proof. To see that \eqref{11e37} is satisfied let $L$ be a finite subset of $\nn$ with
$|L|\meg h_{\tau_0,\delta}^{(\lceil8\delta^{-2}\rceil)}(N_0)$ and $A$ be a subset of $[k]^{<\nn}$ with $\dens_{[k]^l}(A)\meg\delta$
for every $l\in L$. Also fix a finite sequence $(p_n)_{n=0}^m$ of variable words over $k$ with $|p_n|=\tau_n$ for every $n\in\{0,...,m\}$.
By our assumptions on the cardinality of the set $L$ and repeated applications of Proposition \ref{11p11} for ``$\tau=\tau_0$'' and ``$p=p_0$'',
we see that there exist a word $c$ over $k$, a variable word $w$ over $k$ of pattern $p_0$ and a subset $L''$ of $L$ with
\begin{equation} \label{11e38}
|L''| \meg N_0 \stackrel{\eqref{11e36}}{=}\mathrm{DP}(k,(\tau'_n)_{n=0}^{m-1},\Theta_P/2)
\end{equation}
such that the following properties are satisfied.
\begin{enumerate}
\item[(a)] The set $V=\{c\}\cup\{c^\con w(a):a\in[k]\}$ is contained in $A$. Moreover, setting $V(1)=\{c^\con w(a):a\in [k]\}$ and
$\ell_1$ the unique integer with $V(1)\subseteq [k]^{\ell_1}$, we have $\ell_1<\min (L'')$.
\item[(b)] For every $\ell\in L''$ we have
\begin{equation} \label{11e39}
\dens_{[k]^{\ell-\ell_1}}(\Delta) \meg \Theta_P/2
\end{equation}
where $\Delta=\{s\in[k]^{<\nn}: t^{\con}s\in A \text{ for every } t\in V(1)\}$.
\end{enumerate}
For every $n\in\{0,...,m-1\}$ we set $p'_n=p_{n+1}$ and we notice that the length of $p'_n$ is $\tau'_n$. Therefore, by \eqref{11e38} and
\eqref{11e39} and our inductive assumptions, there exist a word  $c'$ over $k$ and a finite sequence $(w'_n)_{n=0}^{m-1}$ of variable words
over $k$ of pattern $(p'_n)_{n=0}^{m-1}$ such that the set
\begin{equation} \label{11e40}
\{c'\} \cup \big\{c^\con w'_0(a_0)^\con...^\con w'_n(a_{n}): n\in\{0,...,m-1\} \text{ and } a_0,...,a_n\in[k]\big\}
\end{equation}
is contained in $\Delta$. We set $w_0=w^\con c'$ and $ w_n=w'_{n-1}$ for every $n\in [m]$. It is easily verified that the set
\begin{equation} \label{11e41}
\{c\} \cup \big\{c^\con w_0(a_0)^\con...^\con w_n(a_{n}):n\in\{0,...,m\} \text{ and }a_0,...,a_n\in[k]\big\}
\end{equation}
is contained in $A$. This shows that \eqref{11e37} is satisfied and so the proof of Theorem \ref{11t1} is completed.
\medskip

\noindent 11.5. \textbf{Proofs of Theorem A and Theorem C.} As we indicated in the introduction, Theorem A is a special case of Theorem C.
Indeed, let $k\in\nn$ with $k\meg 2$ and set $q_n=(v)$ for every $n\in\nn$. Notice that a sequence $(w_n)$ of variable words over $k$ consists
of left variable words if and only if it is of pattern $(q_n)$. Thus, Theorem A follows from Theorem C applied to the sequence $(q_n)$.

So, we only need to prove Theorem C. To this end, fix an integer $k\meg 2$ and a sequence $(p_n)$ of variable words over $k$.
Let $0<\delta\mik 1$ and $A\subseteq [k]^{<\nn}$ such that
\begin{equation} \label{11e42}
\limsup_{n\to\infty} \frac{|A\cap [k]^n|}{k^n} >\delta.
\end{equation}
We fix an infinite subset $L$ of $\nn$ such that $\dens_{[k]^\ell}(A)\meg \delta$ for every $\ell\in\ L$. Recursively, we define a sequence
$(\delta_n)$ in $(0,1]$ by the rule
\begin{equation} \label{11e43}
\delta_0=\delta \text{ and } \delta_{n+1}=2^{-1}\Theta_P(k,\delta_{n}/2).
\end{equation}
Using Corollary \ref{11c12} we may select
\begin{enumerate}
\item[(i)] a sequence $(c_n)$ of words over $k$,
\item[(ii)] a sequence $(v_n)$ of variable words over $k$ of pattern $(p_n)$,
\item[(iii)] a sequence $(A_n)$ of subsets of $[k]^{<\nn}$ with $A_0=A$ and
\item[(iv)] two sequences $(L_n)$ and $(L'_n)$ of infinite subsets of $\nn$
\end{enumerate}
such that for every $n\in\nn$ the following conditions are satisfied.
\begin{enumerate}
\item[(C1)] The set $L'_n$ is contained in $L_n$; moreover, $L_0=L$.
\item[(C2)] The set $V_n=\{c_n\}\cup\{c_n^\con v_n(a): a\in[k]\}$ is contained in $\bigcup_{\ell\in L^0_n} A_n\cap [k]^{\ell}$
where $L^0_n=\{\ell\in L_n: \ell<\min(L'_n)\}$.
\item[(C3)] Let $V_n(1)=\{c_n^\con v_n(a):a\in [k]\}$ and $\ell_n$ be the unique integer such that $V_n(1)\subseteq [k]^{\ell_n}$.
Then
\begin{equation} \label{11e44}
A_{n+1}=\{s\in[k]^{<\nn}: t^{\con}s\in A_n \text{ for every } t\in V_n(1)\}
\end{equation}
and
\begin{equation} \label{11e45}
L_{n+1}=L'_n-\ell_n=\{ \ell-\ell_n: \ell\in L'_n\}.
\end{equation}
\item[(C4)] For every $\ell\in L_n$ we have $\dens_{[k]^\ell}\big(A_n\big) \meg \delta_n$.
\end{enumerate}
The recursive selection is fairly standard and the details are left to the reader.

We set $c=c_0$ and $w_n=v_n^\con c_{n+1}$ for every $n\in\nn$. By (ii) above, the sequence $(w_n)$ is of pattern $(p_n)$.
Moreover, using conditions (C2) and (C3), it is easily verified that the set
\begin{equation} \label{11e46}
\{c\}\cup \big\{ c^\con w_0(a_0)^\con...^\con w_n(a_n): n\in\nn \text{ and } a_0,...,a_n\in [k]\big\}
\end{equation}
is contained in $A$. The proof of Theorem C is completed.
\medskip

\noindent 11.6. \textbf{Further implications.} In this subsection we will discuss the relation of Theorem B and Theorem A with the
density Hales--Jewett Theorem and the density Halpern--L\"{a}uchli Theorem respectively. Notice, first, that the density Hales--Jewett
Theorem follows from Theorem A via a standard compactness argument. In fact, we have the following finer quantitative information.
\begin{prop} \label{11p13}
For every integer $k\meg 2$ and every $0<\delta\mik 1$ we have
\begin{equation} \label{11e47}
\mathrm{DHJ}(k,\delta)\mik \dcs(k,1,\delta).
\end{equation}
\end{prop}
\begin{proof}
Let $n\meg \dcs(k,1,\delta)$ and fix a subset $A$ of $[k]^n$ with $|A|\meg \delta k^n$. For every $\ell\in [n]$ and every
$y\in [k]^{n-\ell}$ let $A_y=\{x\in [k]^\ell: y^{\con}x\in A\}$ and observe that
\begin{equation} \label{11e48}
\ave_{y\in [k]^{n-\ell}} \dens(A_y)=\dens(A)\meg \delta.
\end{equation}
Hence, for every $\ell\in [n]$ we may select $y_{\ell}\in [k]^{n-\ell}$ such that $\dens(A_{y_{\ell}})\meg \delta$. We set
\begin{equation} \label{11e49}
B=\bigcup_{\ell\in [n]} A_{y_{\ell}}
\end{equation}
and we notice that $\dens_{[k]^{\ell}}(B)\meg \delta$ for every $\ell\in [n]$. Since $n\meg \dcs(k,1,\delta)$, there exists a Carlson--Simpson
line $R$ of $[k]^{<\nn}$ which is contained in $B$. Let $(c,w)$ be the generating sequence of $R$. Also let $\ell_R\in [n]$ be the unique
integer such that the $1$-level $R(1)$ of $R$ is contained in $[k]^{\ell_R}$. Then, setting
\begin{equation} \label{11e50}
V=\{y_{\ell_R}^{\con}c^{\con}w(a):a\in [k]\},
\end{equation}
we see that $V$ is a combinatorial line of $[k]^n$ and $V\subseteq A$. This shows that \eqref{11e47} is satisfied, as desired.
\end{proof}
We proceed to discuss how one can deduce the density Halpern--L\"{a}uchli Theorem from Theorem A. The argument is well-known
(see, e.g., \cite{CS,PV}) but we will comment on it for the benefit of the reader.

Recall that a \textit{tree} is a partially ordered set $(T,<)$ such that the set $\{s\in T: s<t\}$ is finite and linearly ordered under $<$
for every $t\in T$. A tree $T$ is said to be \textit{homogeneous} if it is uniquely rooted and there exists an integer $b\meg 2$, called the
\textit{branching number} of $T$, such that every $t\in T$ has exactly $b$ immediate successors. A typical example of a homogeneous tree with
branching number $b\meg 2$ is the set consisting of all finite sequence having values in a set $\mathbb{A}$ of cardinality $b$ and equipped
with the partial order of end-extension; it is denoted by $\mathbb{A}^{<\nn}$ and can, of course, be identified with $[b]^{<\nn}$. Part of the
interest in homogeneous trees of this form is based on the fact that they can be used to ``code'' the level product
\begin{equation} \label{11e51}
\otimes (T_1,...,T_d):=\bigcup_{n\in\nn} T_1(n)\times ... \times T_d(n)
\end{equation}
of a finite sequence $(T_1,...,T_d)$ of homogeneous trees. Specifically, we have the following lemma.
\begin{lem} \label{11l14}
Let $d\in\nn$ with $d\meg 1$. Also let $\mathbf{b}=(b_1,...,b_d)\in\nn^d$ with $b_i\meg 2$ for every $i\in [d]$ and set
$\mathbb{A}_{\mathbf{b}}=[b_1]\times ... \times [b_d]$. Finally let $(T_1,...,T_d)$ be a finite sequence of homogeneous trees such that
the branching number of $T_i$ is $b_i$ for every $i\in [d]$. Then there exists a bijection
\begin{equation} \label{11e52}
\Phi_{\mathbf{b}}: \mathbb{A}_{\mathbf{b}}^{<\nn} \to \otimes (T_1,...,T_d)
\end{equation}
with the following properties.
\begin{enumerate}
\item[(i)] For every $n\in\nn$ we have $\Phi_{\mathbf{b}}(\mathbb{A}_{\mathbf{b}}^n)=T_1(n)\times...\times T_d(n)$.
\item[(ii)] For every $c\in \mathbb{A}_{\mathbf{b}}^{<\nn}$ and every sequence $(w_n)$ of left variable words over $\mathbb{A}_{\mathbf{b}}$
there exist strong subtrees $(S_1,...,S_d)$ of $(T_1,...,T_d)$ having common level set such that, setting
\begin{equation} \label{11e53}
\mathcal{S}=\{c\}\cup \big\{ c^\con w_0(a_0)^\con...^\con w_n(a_n): n\in\nn \text{ and } a_0,...,a_n\in \mathbb{A}_{\mathbf{b}}\big\},
\end{equation}
we have $\Phi_{\mathbf{b}}(\mathcal{S})=\otimes(S_1,...,S_d)$.
\end{enumerate}
\end{lem}
\begin{proof}
Let $i\in [d]$ be arbitrary and $\pi_i:\mathbb{A}_{\mathbf{b}}\to [b_i]$ be the natural projection. Clearly we may assume that the tree $T_i$
coincides with $[b_i]^{<\nn}$ and so we may consider the ``extension'' $\bar{\pi}_i:\mathbb{A}_{\mathbf{b}}^{<\nn}\to T_i$ of $\pi_i$
defined by $\bar{\pi}_i(\varnothing)=\varnothing$ and
\begin{equation} \label{11e54}
\bar{\pi}_i\big( (a_0,...,a_{n-1})\big)= \big(\pi_i(a_0),...,\pi_i(a_{n-1})\big)
\end{equation}
for every integer $n\meg 1$ and every $(a_0,...,a_{n-1})\in \mathbb{A}_{\mathbf{b}}^n$. The map $\Phi_{\mathbf{b}}$ is then
defined by the rule
\begin{equation} \label{11e55}
\Phi_{\mathbf{b}}(s)=\big( \bar{\pi}_1(s),..., \bar{\pi}_d(s) \big).
\end{equation}
It is easily verified that $\Phi_{\mathbf{b}}$ is a bijection and satisfies all desired properties.
\end{proof}
With Lemma \ref{11l14} at our disposal, let us see how Theorem A yields the density Halpern--L\"{a}uchli Theorem. To this end,
fix  a finite sequence $(T_1,...,T_d)$ of homogeneous trees and a subset $A$ of the level product of $(T_1,...,T_d)$ such that
\begin{equation} \label{11e56}
\limsup_{n\to\infty} \frac{|A\cap \big( T_1(n)\times ... \times T_d(n)\big)|}{|T_1(n)\times ... \times T_d(n)|}>0.
\end{equation}
Let $\mathbf{b}=(b_1,...,b_d)$ where $b_i$ is the branching number of the tree $T_i$ for every $i\in [d]$ and consider the bijection
$\Phi_{\mathbf{b}}$ obtained by Lemma \ref{11l14}. We set $B=\Phi_{\mathbf{b}}^{-1}(A)$. By part (i) of Lemma \ref{11l14} and
\eqref{11e56}, we see that
\begin{equation} \label{11e57}
\limsup_{n\to\infty} \frac{|B \cap \mathbb{A}_{\mathbf{b}}^n|}{|\mathbb{A}_{\mathbf{b}}|^n}>0.
\end{equation}
Hence, by Theorem A, there exist $c\in \mathbb{A}_{\mathbf{b}}^{<\nn}$ and a sequence $(w_n)$ of left variable words over
$\mathbb{A}_{\mathbf{b}}$ such that the set
\begin{equation} \label{11e58}
\{c\}\cup \big\{ c^\con w_0(a_0)^\con...^\con w_n(a_n): n\in\nn \text{ and } a_0,...,a_n\in \mathbb{A}_{\mathbf{b}}\big\}
\end{equation}
is contained in $B$. Invoking the definition of the set $B$ and part (ii) of Lemma \ref{11l14}, we conclude that there exist strong subtrees
$(S_1,...,S_d)$ of $(T_1,...,T_d)$ having common level set such that the level product of $(S_1,...,S_d)$ is contained in $A$.


\end{document}